\documentclass[reqno,11pt]{amsart}
\usepackage{amsmath,enumerate,amsfonts,amssymb,amsthm,mathrsfs}
\usepackage{graphicx}
\usepackage{bm}
\allowdisplaybreaks[4]
\usepackage{times}
\usepackage{comment}
\usepackage{float}
\usepackage[numbers,sort&compress]{natbib}
\usepackage{hyperref}
\usepackage{framed}
\usepackage[normalem]{ulem}
\usepackage{geometry}

\numberwithin{equation}{section}
\newtheorem{thm}{Theorem}[section]
\newtheorem{lem}[thm]{Lemma}
\newtheorem{prop}[thm]{Proposition}
\newtheorem{defn}[thm]{Definition}

\newtheorem{cor}[thm]{Corollary}
\newtheorem{rem}[thm]{Remark}
\theoremstyle{plain}

\renewcommand{\d}{\mathrm{d}}
\renewcommand{\H}{\mathbb{H}}
\newcommand{\B}{\mathcal{B}}
\newcommand{\R}{\mathbb{R}}
\newcommand{\Cn}{\C^n}

\renewcommand{\S}{\mathbb{S}}
\newcommand{\C}{\mathbb{C}}
\newcommand{\lam}{\lambda}
\newcommand{\pa}{\partial}
\newcommand{\M}{\mathscr{M}}
\newcommand{\var}{\varepsilon}
\newcommand{\be}{\begin{equation}}
\newcommand{\ee}{\end{equation}}
\newcommand{\Lam}{\Lambda}
\newcommand{\si}{\sigma}
\newcommand{\al}{\alpha}
\newcommand{\om}{\omega}
\newcommand{\Om}{\Omega}
\newcommand{\Hn}{\H ^{n}}

\def\Xint#1{\mathchoice
 {\XXint\displaystyle\textstyle{#1}}
 {\XXint\textstyle\scriptstyle{#1}}
 {\XXint\scriptstyle\scriptscriptstyle{#1}}
 {\XXint\scriptscriptstyle\scriptscriptstyle{#1}}\!\int}
\def\XXint#1#2#3{{\setbox0=\hbox{$#1{#2#3}{\int}$}\vcenter{\hbox{$#2#3$}}\kern-.5\wd0}}

\def\dashint{\Xint-}

\author{Jiechen Qiang}
\address{School of Mathematical Sciences, Laboratory of Mathematics and Complex Systems, MOE,
 Beijing Normal University, Beijing 100875, China}
\email{jcqiang@mail.bnu.edu.cn}

\author{Zhongwei Tang}
\address{School of Mathematical Sciences, Laboratory of Mathematics and Complex Systems, MOE,
 Beijing Normal University, Beijing 100875, China}
\email{tangzw@bnu.edu.cn}
\thanks{Z. Tang is supported by NSFC grants 12471099 and 12271223.}
\author{Heming Wang}
\address{State Key Laboratory of Mathematical Sciences, Academy of Mathematics and Systems Science,
 Chinese Academy of Sciences, Beijing 100190, China}
\email{hmwang@amss.ac.cn}
\thanks{H. Wang is supported by China Postdoctoral Science Foundation 2026M793418.}
\date{}
\title[CR Nirenberg problem, Part II]{The CR Nirenberg Problem, Part II: \\ Existence and Topological Degree}
\hypersetup{
 colorlinks,citecolor=blue,filecolor=blue,linkcolor=blue,urlcolor=black}

\begin{document}
\begin{abstract}
This paper is a continuation of Part~I \cite{QTWPartI}, where blow-up analysis and compactness results for the CR Nirenberg problem were established. We prove some existence and degree results for the CR Nirenberg problem on the standard sphere. More precisely,  we first prove existence results under antipodal and, more generally, finite-group symmetries. Without symmetry, we derive a Leray--Schauder degree formula under natural flatness conditions, yielding existence results for positive prescribed curvatures with certain homogeneous expansions at critical points. We also obtain Aubin-type inequalities under a vanishing first-moment condition, uniformly near the critical exponent.
\end{abstract}
\maketitle
{\small\tableofcontents}

\section{Introduction}

Let $n\ge1$ and let $(\S^{2n+1},\theta_0)$ denote the
standard CR sphere. The CR Nirenberg problem asks whether a
given function can be realized as the Webster scalar curvature
of a contact form conformal to $\theta_0$. Under the conformal
change $\theta=v^{2/n}\theta_0$, with $v>0$, the curvature
transformation law, after normalization by the constant Webster
scalar curvature of $\theta_0$, leads to the equation studied
in Part~I \cite{QTWPartI}:
\begin{equation}\label{maineq}
L_{\theta_0}v=\frac{n^2}{4}Rv^{1+\frac2n},
\quad v>0\quad\text{ on }\,\S^{2n+1},
\end{equation}
where $R\in C^0(\S^{2n+1})$ is the normalized prescribed
Webster scalar curvature and
$L_{\theta_0}:=-\Delta_{\theta_0}+\frac{n^2}{4}$
is the conformal sub-Laplacian associated with $\theta_0$,
with $\Delta_{\theta_0}$ denoting the corresponding sub-Laplacian.

We set $Q=2n+2$ and $Q^*=2Q/(Q-2)$, where $Q^*$ is the
critical exponent for the Folland--Stein embedding on
$\S^{2n+1}$. In our normalization, the sharp Sobolev inequality
of Jerison--Lee \cite{JL1988} takes the form
\begin{equation}\label{eq:sharp-sphere-sobolev}
\Big(\dashint_{\S^{2n+1}}|v|^{Q^*}\,\d\si\Big)^{2/Q^*}
\leq \frac{4}{n^2}
\dashint_{\S^{2n+1}}vL_{\theta_0}v\,\d\si,
\quad v\in S_1^2(\S^{2n+1}),
\end{equation}
where $\dashint$ denotes the average with respect to the
Euclidean surface measure $\d\si$, and $S_1^2(\S^{2n+1})$
is the Folland--Stein space equipped with the norm
\eqref{sb norm} in Section~\ref{sec2.2}; see \cite{FS1974}. The positive extremals
of \eqref{eq:sharp-sphere-sobolev} are precisely the positive
multiples of conformal images of the constant function $1$.
The conformal images of $1$ solve \eqref{maineq} with
$R\equiv1$, and the noncompact CR conformal group generates
concentrating families of such solutions. The resulting lack of
compactness is a central difficulty in the variational
analysis of problem \eqref{maineq}.

Apart from the lack of compactness, the solvability of
\eqref{maineq} is subject to restrictions on the prescribed
curvature $R$. Testing \eqref{maineq} against $v$ shows that
$R$ must be positive somewhere. When
$R\in C^1(\S^{2n+1})$, the Kazdan--Warner identity
\eqref{KW 0} in Section~\ref{sec:2.1} gives a further obstruction even for positive
prescribed curvatures. For instance, if
$R(\zeta)=c+\operatorname{Re}\zeta_1$, where
$\zeta=(\zeta_1,\ldots,\zeta_{n+1})$ and $c>1$,
then \eqref{maineq} admits no positive solution.

Further geometric and analytic background is recalled in
Section~\ref{sec 2}; see also
\cite[Sections~1 and~2]{QTWPartI} and the references therein.

Perturbation and degree methods for the classical Nirenberg problem have been developed extensively; see, e.g.,
\cite{BC1991,CGY1993,CY1991,Li1995,SZ1996} and the references
therein. For the fractional Nirenberg problem, Jin--Li--Xiong established compactness theorems in \cite{JLX2014} and existence results in \cite{JLX2015}. These methods motivate the corresponding CR problem, where the different scaling of horizontal and Reeb derivatives requires flatness conditions adapted to the Heisenberg dilations.

The prescribed Webster scalar curvature problem has been studied
extensively by various methods. Among the many contributions,
Malchiodi--Uguzzoni \cite{MU2002} established perturbative
existence results on the CR sphere for curvatures close to a
positive constant. Felli--Uguzzoni \cite{FU2004} obtained existence
results under symmetry assumptions by variational and perturbation
methods. Topological approaches based on critical points at infinity
were developed by Chtioui--Ahmedou--Yacoub \cite{CAY2013}, while
existence results under $\beta$-flatness conditions were obtained
by Gamara--Hafassa \cite{GH2017}. Ho \cite{Ho20152} proved existence
results using the Webster scalar curvature flow. More recently,
Afeltra \cite{Afeltra2024-2} established compactness and a
Leray--Schauder degree formula for the prescribed Webster curvature
problem on three-dimensional CR manifolds.

In Part~I \cite{QTWPartI}, we established local estimates near isolated blow-up points, uniform lower bounds on the distances between distinct blow-up points, and energy bounds under weighted flatness conditions. Under additional nondegeneracy and integral nonvanishing conditions on the homogeneous terms at critical points, these estimates exclude blow-up along
\begin{equation}\label{eq:curvature-homotopy}
R_\mu=1+\mu(R-1), \quad \var \leq \mu \leq 1,
\end{equation}
for every fixed $\var>0$. In particular, positive solutions satisfy uniform upper and positive lower bounds on each such parameter interval. Part I \cite{QTWPartI} also establishes a single-point concentration alternative and the convergence of conformally normalized solutions to 1, which are used when $\mu$ tends to zero.

The purpose of this paper is to derive existence and degree results from the estimates of Part~I \cite{QTWPartI}. The degree criterion is expressed in terms of the averaged map $G_\lam$ in \eqref{0.10} below, where the CR automorphisms $\varphi_{P,\lam}$ are defined in \eqref{varphi new} in Section~\ref{sec2.2}. For explicit homogeneous expansions at the critical points of $R$, the degree is computed directly from the leading terms.

The first two results require symmetry and do not use the degree computation. We first specify the notation in their statements. In Heisenberg coordinates, write $\xi=(z,t)=(x,y,t)\in\Hn=\C^n\times\R$, where $z=x+\mathrm{i}y$ and $x,y\in\R^n$. We use the group product $\circ$ recalled in Section~\ref{sec2.2}, the homogeneous norm $|\xi|=\rho(\xi)=(|z|^4+t^2)^{1/4}$, and the dilations $\delta_\lam (z,t)=(\lam  z,\lam ^2t)$ for $\lam >0$. The centered Cayley chart $F_P$, defined by \eqref{eq:normalized-cayley} in Section~\ref{sec2.2}, satisfies $F_P(0)=P$; in a specified chart, $R(\xi)$ abbreviates $(R\circ F_P)(\xi)$.

Let $Z=(X_1,\ldots,X_n,Y_1,\ldots,Y_n)$ be the left-invariant horizontal fields in \eqref{vectorfields} of Section~\ref{sec2.2}, and let $T=\pa _t$. For $\al \in (0,1)$, the class $\mathcal A_\al $ consists of $C^1$ functions with bounded weak derivatives $Z_aZ_bR$ and with $TR$ bounded and $\al $-H\"older for the contact distance $d(\xi,\eta)=|\eta^{-1}\circ\xi|$, with the norm \eqref{eq:intrinsic-coefficient-norm} defined in Section~\ref{sec:coefficients}. On the sphere, the requirements are imposed in fixed centered charts, as specified there. 

Throughout the paper, we assume $n\geq1$ unless stated otherwise.
We first consider existence under antipodal symmetry. The flatness
condition is defined as follows:
\begin{defn}\label{def:value-flatness}
For $d>0$, we say that $R\in C^0(\S^{2n+1})$ has flatness order greater than $d$ at $\zeta$ if, in centered Cayley coordinates $\xi$ at $\zeta$, one has $R(F_\zeta(\xi))=R(\zeta)+o(|\xi|^d)$ as $\xi\to0$.
\end{defn}

\begin{thm}\label{thm:1}
Let $\al \in(0,1)$ and let $R\in\mathcal A_\al (\S^{2n+1})$ be antipodally symmetric and positive somewhere.  If $R$ has flatness order greater than $Q-2$ at one of its maximum points, then \eqref{maineq} admits a positive $C^2$ solution.
\end{thm}

The existence result can be extended to finite-group symmetries.
Let $G$ be a finite subgroup of the CR isometry group of
$\S^{2n+1}$. For $\zeta\in\S^{2n+1}$, write
$O_G(\zeta)=\{g(\zeta):g\in G\}$ for its orbit, and denote
its cardinality by $|O_G(\zeta)|$.
We say that $R$ is $G$-invariant if $R\circ g=R$ for every
$g\in G$, and that $G$ has no global fixed point if
$|O_G(\zeta)|\geq2$ for every $\zeta\in\S^{2n+1}$.
The stabilizers are allowed to be nontrivial.
\begin{thm}\label{thm:1-2}
Let $\al \in(0,1)$. Assume that $G$ has no global fixed point and that $R\in\mathcal A_\al (\S^{2n+1})$ is positive and $G$-invariant.  Suppose $R$ has flatness order greater than $Q-2$ at $\zeta_0$ and
\[
 \frac{R(\zeta_0)}{|O_G(\zeta_0)|^{2/(Q-2)}}
 \ge \frac{R(\zeta)}{|O_G(\zeta)|^{2/(Q-2)}},
 \quad\forall\,\zeta\in\S^{2n+1}.
\]
Then \eqref{maineq} admits a positive $G$-invariant $C^2$ solution.
\end{thm}

We next turn to the problem without symmetry assumptions.
Our existence criterion is based on the topological degree and
homogeneous expansions of $R$ at its critical points. We first
recall the notation and assumptions needed to state the results.

Let $\nabla_{\theta_0}$ denote the horizontal gradient and
$T_{\theta_0}$ the Reeb field. We write
\[
\operatorname{Crit}(R):=\{P\in\S^{2n+1}:\d R(P)=0\}
=\{P\in\S^{2n+1}:\nabla_{\theta_0}R(P)=0,\ T_{\theta_0}R(P)=0\}
\]for the set of full critical points of $R$.
In Heisenberg coordinates, the field $T=\pa_t$ acts on
pullbacks and is distinct from $T_{\theta_0}$.
Using the right-invariant vector fields defined in
\eqref{right-vectorfields} of Section~\ref{sec2.2}, set $\widetilde\nabla=(\overline X_1,\ldots,\overline X_n,\overline Y_1,\ldots,\overline Y_n)$ and $\overline\nabla=(\widetilde\nabla,T)$. We use Euclidean norms for these derivative arrays. The
weighted horizontal derivatives are $\nabla_\H=Z/2$ and $\nabla_\H^2=\tfrac14(Z_aZ_b)_{a,b=1}^{2n}$.

For $\beta>2$, we say that $R$ satisfies condition
$(*)'_\beta$ on a coordinate domain $U\subset\Hn$ if there
exist constants $L>0$, $C_*\geq1$, and a continuous function
$r:U\to[0,\infty)$, locally $C_*$-Lipschitz with respect to the contact
distance, such that
\begin{gather*}
C_*^{-1} L r^{\beta-1} \leq|\nabla_{\H } R|+r|T R| \leq C_* L r^{\beta-1}, \quad|\nabla_{\H }^2 R|+|T R| \leq C_* L r^{\beta-2}, \\
|\nabla_{\H }(T R)| \leq C_* L r^{\beta-3} \quad \text{ on }\,\{\xi\in U:r(\xi)>0\} .
\end{gather*}
The first inequality holds at every point of $U$.
The second horizontal derivatives are understood weakly
on $U$, whereas the mixed derivatives $Z_a(TR)$ are
understood weakly on $\{\xi\in U:r(\xi)>0\}$.
The second and third inequalities hold almost everywhere
on their respective domains, with respect to Haar measure.
For $\beta=2$, condition $(*)'_2$ reduces to
$|\nabla_\H^2R|+|TR|\leq C_*L$ almost everywhere in $U$,
without a lower bound on the first derivatives.
No bound on $T^2R$ is required.
For a sequence $\{R_i\}$, we use the notation
$L_i=L(\beta,i)$ and $r_i$ as in
\cite[Definition~1.2]{QTWPartI}. On the sphere, the
constant $C_*$ and the coordinate radius are required
to be uniform in $i$ and in the centered chart.

The integral condition below involves the standard bubble $U_{0,1}(z,t)=2^n((1+|z|^2)^2+t^2)^{-n/2}$ defined in \eqref{standard-bubble} of Section~\ref{sec2.2}.
We distinguish $(Df)(\eta\circ\xi)$, which denotes the
derivative evaluated at $\eta\circ\xi$, from
$D_\xi[f(\eta\circ\xi)]$, which denotes differentiation
of the composition. For $P\in\S^{2n+1}$ and $\lam\geq1$,
let
$\varphi_{P,\lam}
=F_P\circ\delta_{\lam^{-1}}\circ F_P^{-1}$,
extended across the omitted point, as in
\eqref{varphi new} of Section~\ref{sec2.2}.

Let $\mathbb B^{n+1}=\{b\in\C^{n+1}:|b|<1\}$. For a nonvanishing continuous map $g:\S^{2n+1}\to\C^{n+1}$, the notation $\deg(g,\mathbb B^{n+1},0)$ means the Brouwer degree of any continuous extension to the closed ball, equivalently the degree of $g/|g|:\S^{2n+1}\to\S^{2n+1}$. Here $\C^{n+1}$ is identified with $\R^{2n+2}$.
The degree in \eqref{0.11} is understood in this sense,
whereas degrees in function spaces are Leray--Schauder degrees.

The following theorem gives a degree-theoretic existence criterion for problem \eqref{maineq}.
\begin{thm}\label{thm:0.5}
Let $\al \in (0,1)$ and let  $R \in \mathcal{A}_\al (\S^{2 n+1})$ be positive. In every fixed contact chart assume that the weak derivatives $Z_a(T R)$ are locally essentially bounded away from $\operatorname{Crit}(R)$.  For each critical point $\zeta_0$ of $R$, let $F_{\zeta_0}$ be the centered Cayley chart in \eqref{eq:normalized-cayley}, with $F_{\zeta_0}(0)=\zeta_0$, and set $R_{\zeta_0}(\xi):=R(F_{\zeta_0}(\xi))$. Assume that there exist a small neighborhood $\mathscr{O}_{\zeta_0}$ of $0$ in $\H^n$,  a constant $L>0$, and $\beta=\beta(\zeta_0)\in(Q-2,Q)$ such that  $R_{\zeta_0}$ satisfies condition $(*)^\prime_\beta$ in $\mathscr{O}_{\zeta_0}$ for $L$, and
\[
R_{\zeta_0}(\xi)=R(\zeta_0)+Q_{\zeta_0}^{(\beta)}(\xi)+W_{\zeta_0}(\xi),\quad \xi\in\mathscr{O}_{\zeta_0},
\]
where $Q_{\zeta_0}^{(\beta)}\in C^1_{\mathrm{loc}}(\H^{n}\backslash \{0\})$ satisfies
    $Q_{\zeta_0}^{(\beta)}(\delta_{\lam}\xi)=\lam ^{\beta}Q_{\zeta_0}^{(\beta)}(\xi)$, $\forall\, \lam  >0$, $\xi\in \Hn $,
  \be\label{0.6}
	c_{\mathrm{nd}}|\xi|^{\beta-1}\leq |\widetilde{\nabla} Q_{\zeta_0}^{(\beta)}(\xi)|+|\xi||TQ_{\zeta_0}^{(\beta)}(\xi)|,
\ee
for some $c_{\mathrm{nd}}>0$, which may depend on $\zeta_0$, and
\be \label{0.8in jde}
	\begin{pmatrix}
	\displaystyle\int_{\Hn } \overline{\nabla} Q_{\zeta_0}^{(\beta)}(\eta\circ\xi)U_{0,1}^{Q^* } \, \d  \xi\\[2.5mm]
    \displaystyle\int_{\Hn } Q_{\zeta_0}^{(\beta)}(\eta\circ\xi)U_{0,1}^{Q^* } \, \d  \xi 
	\end{pmatrix} \neq 0, \quad \forall\,  \eta\in \Hn ,
	\ee 
    and $W_{\zeta_0}\in C^{1}(\mathscr{O}_{\zeta_0})$ satisfies $\lim _{\xi \to 0}(|W_{\zeta_0}(\xi)||\xi|^{-\beta}+|\widetilde{\nabla}W_{\zeta_0}(\xi)||\xi|^{1-\beta}+|TW_{\zeta_0}(\xi)||\xi|^{2-\beta})=0$.
Then, for all sufficiently large $\lam $,
\begin{equation}\label{0.10}
G_\lam (P):=\dashint_{\S^{2n+1}}R\circ\varphi_{P,\lam }(\zeta)\,\zeta\,\d\si \ne0,
\quad\forall\,P\in\S^{2n+1}.
\end{equation}
If
\begin{equation}\label{0.11}
\deg(G_\lam ,\mathbb B^{n+1},0)\ne0,
\end{equation}
then \eqref{maineq} admits a positive $C^2$ solution.
\end{thm}

\begin{rem}
The map $G_\lam$ defined in \eqref{0.10} is the CR analogue
of the map introduced by Chang and Yang \cite{CY1987}
in their study of prescribed Gaussian curvature on $\S^2$.
For each fixed $\lam\geq 1$, it defines a continuous map
from $\S^{2n+1}$ into $\C^{n+1}\simeq\R^{2n+2}$.
\end{rem}

We next derive from Theorem~\ref{thm:0.5} an explicit
existence criterion for prescribed Webster curvatures
admitting a particular class of homogeneous expansions
at their critical points.
\begin{thm}\label{thm:2}
Let $\al\in(0,1)$ and let $R\in\mathcal A_\al(\S^{2n+1})$
be positive. Assume that, in every fixed contact chart, the weak
derivatives $Z_a(TR)$ are locally essentially bounded away from
$\operatorname{Crit}(R)$. At each $\zeta_0\in\operatorname{Crit}(R)$,
suppose that $R_{\zeta_0}=R\circ F_{\zeta_0}$ satisfies $(*)'_\beta$
in a neighborhood $\mathscr O_{\zeta_0}$ of $0$ for some
$\beta=\beta(\zeta_0)\in(Q-2,Q)$ and admits the expansion
\begin{equation}\label{eq:alln-model-expansion}
R_{\zeta_0}(\xi)
=R(\zeta_0)+\sum_{j=1}^n
(a_j|x_j|^\beta+b_j|y_j|^\beta)
+cJ_{\beta,\vartheta}(z,t)+W_{\zeta_0}(\xi),
\end{equation}
where $\xi=(z,t)=(x,y,t)\in\Hn$,
$J_{\beta,\vartheta}(z,t)=(t^2+\vartheta|z|^4)^{\beta/4}$
with $\vartheta>0$, and
$W_{\zeta_0}\in C^1(\mathscr O_{\zeta_0})$ satisfies the
remainder condition in Theorem~\ref{thm:0.5}.
Assume that $a_j,b_j,c$ are nonzero and satisfy
$0<|c|<\var(n,\beta,\vartheta)
\min_{1\le j\le n}\{|a_j|,|b_j|\}$
for a sufficiently small constant $\var(n,\beta,\vartheta)>0$.
Define
\begin{equation}\label{eq:alln-kappa}
\kappa_\vartheta=
\frac{\int_{\Hn}J_{\beta,\vartheta}(z,t)
((1+|z|^2)^2+t^2)^{-n-1}\,\d z\,\d t}
{\int_{\Hn}|x_1|^\beta
((1+|z|^2)^2+t^2)^{-n-1}\,\d z\,\d t}>0
\end{equation}
and, for each $\zeta\in\operatorname{Crit}(R)$, set
\[
\widetilde a(\zeta)
=(a_1,\ldots,a_n,b_1,\ldots,b_n,\kappa_\vartheta c),
\quad 
i(\zeta,R)
=\#\{j\in\{1,\ldots,2n+1\}:\widetilde a_j(\zeta)<0\},
\]
where $a_j,b_j,c$, $\beta$, $\vartheta$, and
$\kappa_\vartheta$ are associated with the local expansion
at $\zeta$, and $\#$ denotes cardinality.
Suppose further that
$\sum_{j=1}^{2n+1}\widetilde a_j(\zeta)\neq0$
at every critical point. If
\begin{equation}\label{sum index}
\sum_{\substack{\zeta\in\operatorname{Crit}(R)\\
\sum_{j=1}^{2n+1}\widetilde a_j(\zeta)<0}}
(-1)^{i(\zeta,R)}\neq-1,
\end{equation}
then \eqref{maineq} admits a positive $C^2$ solution.
\end{thm}

We finally extend the degree criterion to the borderline
flatness order $Q-2$, allowing this order at at most one
critical point when $Q-2>2$.
\begin{thm}\label{thm:0.6}
Assume the hypotheses of Theorem~\ref{thm:0.5}, except that the local exponent is allowed to lie in $[Q-2,Q)\cap(2,\infty)$.  Let
\[
\mathscr K_\gamma=\{\zeta\in\operatorname{Crit}(R):\beta(\zeta)=\gamma\}.
\]
If $\#\mathscr K_{Q-2}\le1$, then \eqref{0.10} holds for all sufficiently large $\lam $.  If, in addition, condition \eqref{0.11} holds, then \eqref{maineq} admits a positive $C^2$ solution.
\end{thm}

The existence result under antipodal symmetry has counterparts in the classical Nirenberg problem, notably the work of Moser \cite{M1973} on $\S^2$ and Escobar--Schoen \cite{ES1986} on higher-dimensional spheres. The degree criteria are closely related to the results of Bahri--Coron \cite{BC1991} in dimension three and Li \cite{Li1995} in higher dimensions. Li \cite{Li1996} also treated the case in which at most one critical point has the critical flatness order. Corresponding existence results under symmetry assumptions and degree conditions for the fractional Nirenberg problem were obtained by Jin--Li--Xiong \cite{JLX2015}.

We begin with the symmetric problem. For Theorems~\ref{thm:1} and \ref{thm:1-2}, we minimize the energy among functions with the prescribed symmetry. The flatness assumption allows us to construct test functions with energy below the concentration threshold. Subcritical approximation and blow-up analysis then exclude concentration under the strict energy inequality and yield a positive symmetric solution.

For Theorems~\ref{thm:0.5}, \ref{thm:2} and \ref{thm:0.6}, we first consider prescribed curvatures close to a positive constant. After separating the CR conformal directions, the constrained energy has a uniformly positive second variation near 1. We solve the constrained minimization problem and compare the remaining finite-dimensional equation with the averaged map in \eqref{0.10}. Theorem~\ref{thm:3.1} in Section~\ref{sec:degree-comparison} gives the degree comparison \eqref{eq: degree}.

In the CR setting, the horizontal and Reeb directions have different scaling under Heisenberg dilations. The local analysis of the averaged map therefore involves the homogeneous term and its right-invariant derivatives.
For \eqref{eq:alln-model-expansion}, the rescaled expansion \eqref{4.47 new} in Section~\ref{sec:homogeneous-degree} gives the local index \eqref{4.49 new}, from which Proposition~\ref{Cor new} follows.

To pass from the perturbative case to general prescribed curvatures, we use the homotopy \eqref{eq:curvature-homotopy}. By the results of Part I \cite{QTWPartI} and the conformal normalization in Section~\ref{sec:perturbation}, the normalized solutions converge to 1 in the energy space as $\mu$ tends to zero. After fixing a sufficiently small positive $\mu$, compactness along the remaining parameter interval allows us to continue the degree to $\mu=1$. For Theorem~\ref{thm:0.6}, any configuration of multiple blow-up points contains a point with $\beta>Q-2$, where the local Pohozaev identity excludes such a configuration. The remaining single-point blow-up is excluded by \eqref{0.6} and \eqref{0.8in jde}.

Together with Part~I \cite{QTWPartI}, these results give existence and degree criteria under the stated weighted flatness conditions in every CR dimension $n\ge1$, without restricting the prescribed curvature to a small perturbation of a constant. The comparison \eqref{eq: degree} and the continuation argument connect the local homogeneous expansions to the degree of the full equation. For \eqref{eq:alln-model-expansion}, the existence criterion reduces to the explicit condition \eqref{sum index}. The local degree is computed without assuming that the actual zeros of the associated tangential field are isolated or nondegenerate.

The paper is organized as follows. Section~\ref{sec 2} recalls the geometric and analytic preliminaries. Section~\ref{sec:3} proves the existence results under symmetry. In Section~\ref{sec:perturbation}, we establish the perturbation theorem and compute the degree for the explicit homogeneous expansions. Section~\ref{sec:main-existence} completes the proofs of the remaining main theorems using compactness from Part~I  \cite{QTWPartI} and homotopy invariance. Appendix~\ref{appendix} contains some technical lemmas and Aubin-type inequalities.

%\subsection*{Acknowledgments}
% We thank Professor YanYan Li for many helpful personal communications and valuable suggestions since the beginning of this project in May 2024. Z. Tang is supported by NSFC grants 12471099 and 12271223.  H. Wang is supported by China Postdoctoral Science Foundation 2026M793418.

\section{Preliminaries}\label{sec 2}
We collect the geometric conventions, function spaces, and conclusions of Part I \cite{QTWPartI} needed in the existence proofs. The distinction between horizontal, full, and conformal gradients is essential in the computations below. We retain the notation of the introduction and set $q=Q^*$ and $c_0=n^2/4$. All variational spaces consist of real-valued functions. Complex coordinates are used to express real vectors and operators on $\C^{n+1} \simeq \mathbb{R}^{2 n+2}$.

The letter $C>0$ denotes a constant that may change from line to line. Unless otherwise stated, constants may depend on $n$; further dependencies and uniformity are specified where needed. Constants with subscripts are fixed within the relevant statement, unless explicitly reselected.
\subsection{Conformal identities on the CR sphere}\label{sec:2.1}
 The standard CR sphere, defined by $|\zeta|=1$ in $\mathbb C^{n+1}$, carries the contact form
$$
\theta_0=\mathrm{i} \sum_{j=1}^{n+1}(\zeta_j \, \d  \bar{\zeta}_j-\bar{\zeta}_j \, \d  \zeta_j).
$$
Its horizontal distribution is $H=\operatorname{ker}\theta_0$, and its complex structure $J$ is the restriction of multiplication by $\mathrm{i}$ to $H$. The Levi form is $G_{\theta_0}(X, Y)=-\d  \theta_0(J X, Y)$ on $H$. The Reeb field is characterized by $\theta_0(T_{\theta_0})=1$ and $\d  \theta_0(T_{\theta_0}, \cdot)=0$; here $T_{\theta_0}(\zeta)=\mathrm{i} \zeta / 2$. The Webster metric extends the Levi form by declaring the Reeb field unit and orthogonal to $H$. On the sphere, $\operatorname{grad}_{g_{\theta_0}}$ denotes the full gradient, $\nabla_{\theta_0}$ the horizontal gradient, and $\Delta_{\theta_0}=\operatorname{div} \nabla_{\theta_0}$ the sub-Laplacian, as in Part~I \cite{QTWPartI}.

Under a conformal change $\widehat{\theta}=v^{2 / n} \theta_0$, the Webster scalar-curvature transformation gives \eqref{maineq} after normalizing the prescribed curvature by its constant value at $\theta_0$; see \cite{DT2006,JL1987}. The nonnegative operator $-\Delta_{\theta_0}$ has the following coordinate expression, applied to a smooth extension off the sphere:
\[
-\Delta_{\theta_0}=-\sum_{j=1}^{n+1}\frac{\pa ^2}{\pa \zeta_j\pa \bar\zeta_j}
+\sum_{j,k=1}^{n+1}\zeta_j\bar\zeta_k\frac{\pa ^2}{\pa \zeta_j\pa \bar\zeta_k}
+\frac n2\sum_{k=1}^{n+1}\Big(\zeta_k\frac\pa {\pa \zeta_k}
+\bar\zeta_k\frac\pa {\pa \bar\zeta_k}\Big),
\]
see, e.g., \cite{BFM2013,JL1987}.
Unless a measure is displayed, sphere integrals use the Euclidean surface measure $\d  \si $, and Heisenberg integrals use Lebesgue measure $\d\xi=\d z\,\d t$, where $\d z=\d x\,\d y$. Set $\om _{2 n+1}=|\S^{2 n+1}|$ and $\dashint_{\S^{2 n+1}} f\,\d\si =\om _{2 n+1}^{-1} \int_{\S^{2 n+1}} f \, \d  \si $. The contact volume is $\d  V_{\theta_0}=2^{2 n+1} n!\d  \si $; normalized averages are the same for either measure. Applying integration by parts, we obtain  $\int_{\S^{2n+1}} v L_{\theta_0} v\,\d\si =\int_{\S^{2n+1}}(|\nabla_{\theta_0} v|^2+c_0 v^2)\,\d\si $.

Distances in concentration estimates are measured by the contact distance
\begin{align}\label{eq:contact-distance}
d_{\theta_0}(P, P')=\sqrt{2|1-\langle P, P'\rangle|}, \quad \mathcal{B}(P, r)=\{P'\in\S^{2n+1}:d_{\theta_0}(P,P')<r\},
\end{align}
where $\langle\zeta,\eta\rangle=\sum_{j=1}^{n+1}\zeta_j\overline{\eta_j}$ and the bar denotes complex conjugation. In the degree computations, $a\cdot b=\operatorname{Re}\langle a,b\rangle$ is the real Euclidean pairing. The distance is locally comparable to the Carnot--Carath\'eodory distance. For $K\subset\S^{2n+1}$, set $d_{\theta_0}(P,K)=\inf_{P'\in K}d_{\theta_0}(P,P')$, with $d_{\theta_0}(P,\varnothing)=+\infty$. 
Geller \cite{G1981} showed that, for regular $f:\S^{2n+1}\to\C$,
\begin{equation}\label{Green}
L_{\theta_0}^{-1}(f)(\zeta)
=c_n\int_{\S^{2n+1}}\operatorname{dist}(\zeta,\eta)^{-2n}f(\eta)\,\d \si (\eta),
\end{equation}
where $c_n=\frac{2^{n-1}\Gamma(\frac n2)^2}{\pi^{n+1}}$, $\Gamma$ denotes the Euler gamma function, and
$\operatorname{dist}(\zeta,\eta)=d_{\theta_0}(\zeta,\eta)$ as in \eqref{eq:contact-distance}.

For $R\in C^1(\S^{2n+1})$, every positive solution of \eqref{maineq} satisfies the Kazdan--Warner identity \cite[Theorem~B]{Cheng}:
\begin{equation}\label{KW 0}
\int_{\S^{2n+1}}X(R)v^{Q^*}\,\d V_{\theta_0}=0,
\end{equation}
for every infinitesimal CR automorphism $X$ of $\S^{2n+1}$.
To apply \eqref{KW 0} to the conformal variations below, we distinguish a second gradient from the horizontal gradient $\nabla_{\theta_0}$. Let $g_{\rm rd}$ be the standard round metric, let $T_0(\zeta)=\mathrm{i}\zeta=2T_{\theta_0}(\zeta)$, and set $\eta_0=g_{\rm rd}(T_0,\cdot)$. Define
\[
\widehat g=g_{\rm rd}-\frac12\eta_0\otimes\eta_0,
\quad \widehat\nabla=\operatorname{grad}_{\widehat g},\quad \widehat\Delta=\operatorname{div}_{\widehat g}\widehat\nabla.
\]
For $f_a(\zeta)=\operatorname{Re}\langle\zeta,a\rangle$,
\[
\widehat\nabla f_a=a-\langle\zeta,a\rangle\zeta,
\quad -\widehat\Delta f_a=Qf_a.
\]
These gradients generate CR automorphisms, and the volume form of $\widehat g$ is a constant multiple of $\d V_{\theta_0}$. Thus every positive solution of \eqref{maineq} satisfies
\begin{equation}\label{KW new}
\int_{\S^{2n+1}}
\langle\widehat\nabla R,\widehat\nabla f_a\rangle_{\widehat g}
 v^{Q^*}\,\d V_{\theta_0}=0,
\quad a\in\C^{n+1}.
\end{equation}
Pairings involving two $\widehat\nabla$-gradients are taken with respect to $\widehat g$ and extended complex bilinearly when needed. Notice also that
\[
|\widehat\nabla f_a|_{\widehat g}^2
=|a|^2-\operatorname{Re}(\langle\zeta,a\rangle^2)\ge0,
\]
with strict inequality on a nonempty open set whenever $a\ne0$.

\subsection{Cayley coordinates and variational spaces}\label{sec2.2}

The Heisenberg group $\Hn$ is the Lie group $\Cn\times\R$ with group law
\[
(z,t)\circ(\hat z,\hat t)
=(z+\hat z,t+\hat t+2\operatorname{Im}(z\cdot\overline{\hat z})).
\]
Here $z\cdot\overline{\hat z}=\sum_{j=1}^n z_j\overline{\hat z_j}$, the identity is $0=(0,0)$, and $(z,t)^{-1}=(-z,-t)$. We write $\xi=(z,t)=(x,y,t)$, with $z=x+\mathrm{i} y$, and set
\[
|\xi|=\rho(\xi)=(|z|^4+t^2)^{1/4},\quad
 d(\xi,\xi_0)=|\xi_0^{-1}\circ\xi|.
\]
Left translations and Heisenberg dilations are denoted by
\begin{equation}\label{heisenberg-dilation}
\tau_{\xi_0}(\xi)=\xi_0\circ\xi,
\quad \delta_\lam (z,t)=(\lam  z,\lam ^2t).
\end{equation}
For $r>0$ we write $B_r(\xi_0)=\{\xi\in\Hn:d(\xi,\xi_0)<r\}$ and $B_r=B_r(0)$ when the center is the origin.
The standard left-invariant vector fields are
\begin{equation}\label{vectorfields}
X_j=\frac\pa {\pa  x_j}+2y_j\frac\pa {\pa  t},
\quad
Y_j=\frac\pa {\pa  y_j}-2x_j\frac\pa {\pa  t},
\quad T=\frac\pa {\pa  t},\quad 1\leq j\leq n,
\end{equation}
and
\[
\nabla_{\H}u=\frac12(X_1u,\ldots,X_nu,Y_1u,\ldots,Y_nu),
\quad
\Delta_{\H}=\frac14\sum_{j=1}^n(X_j^2+Y_j^2).
\]

The Cayley transform relates the conformal sub-Laplacian on the sphere to the Heisenberg sub-Laplacian. More precisely, it maps $\S^{2n+1}\backslash\{\mathcal S\}$ to $\Hn$ and is the inverse of
\be\label{Cayley}
	F: \Hn \to \S^{2n+1}\backslash\{\mathcal{S}\}, \quad (z,t) \mapsto\Big(\frac{2 z}{1+|z|^2-\mathrm{i} t}, \frac{1-|z|^2+\mathrm{i} t}{1+|z|^2-\mathrm{i} t}\Big),
\ee
where $\mathcal{S}=(0,\ldots,0,-1)$ is the south pole of $\S^{2n+1}$. For $P\in\S^{2n+1}$ choose a unitary matrix $U_P$ with $U_P(0,\ldots,0,1)=P$ and set
\begin{equation}\label{eq:normalized-cayley}
 F_P=U_P\circ F\circ\delta_{1/2},\quad F_P(0)=P.
\end{equation}
The choices of $U_P$ range over the compact unitary group; a global smooth choice is not required. The hypotheses on $R$ are understood in these centered Cayley charts on fixed small coordinate balls. Their contact distances are uniformly comparable with $d$, and their conformal factors equal $1$ at the origin. The usual pseudo-Hermitian normal coordinates on a general CR manifold and these exact sphere charts are not identified without a coordinate change. %Lemma~\ref{lem:wf-contact} below records the required derivative estimates for such changes.
Using \eqref{Green} and \eqref{Cayley}, we obtain
\be\label{1.7}
(L_{\theta_{0}} \phi) \circ F=-(2|J_F|)^{-(Q+2)/(2Q)} \Delta_{\H}((2|J_F|)^{(Q-2)/(2Q)}(\phi \circ F))\quad \text{ for }\,\phi \in C^{\infty}(\S^{2n+1}),
\ee
where  $$|J_F|=\frac{2^{2 n+1}}{((1+|z|^2)^2+t^2)^{n+1}}$$ is the surface Jacobian of $F$, characterized by $F^*(\d\si )=|J_F|\,\d\xi$, as in Part~I \cite{QTWPartI}. 

The transformation law \eqref{1.7} gives the corresponding equation on $\Hn$. Setting $u=(2|J_F|)^{(Q-2)/(2Q)}(v\circ F)$ and writing $R(\xi)=(R\circ F)(\xi)$, every positive solution of \eqref{maineq} satisfies
\be\label{maineq1}
-\Delta_{\H} u=\frac{n^2}{4}R(\xi) u^{\frac{Q+2}{Q-2}}, \quad u>0 \quad  \text{ in }\,\Hn.
\ee
Conversely, if $u$ is a positive $C^2$ solution of \eqref{maineq1}, then
\[
v=[(2|J_F|)^{-\frac{Q-2}{2 Q}} u] \circ F^{-1}
\]
solves \eqref{maineq} on $\S^{2n+1}\backslash\{\mathcal S\}$. It corresponds to a positive $C^2$ solution on the whole sphere if and only if $v$ extends to a positive $C^2$ function across $\mathcal S$. The extended function satisfies \eqref{maineq} at $\mathcal S$ by continuity. The equivalence between \eqref{maineq} and \eqref{maineq1} is understood for solutions with this extension property.

Let $E=\dot S_1^2(\Hn)$ be the completion of $C_c^\infty(\Hn)$ under the norm
\[
\|u\|_E=\Big(\int_{\Hn}|\nabla_{\H}u|^2\,\d\xi\Big)^{1/2}.
\]
The Folland--Stein embedding $E\hookrightarrow L^{Q^*}(\Hn)$ holds, and the sharp Sobolev inequality \cite{FS1974,JL1988} is
\be\label{Sobolevineq}
S_n \Big(\int_{\Hn} |u|^{Q^*}\,\d\xi\Big)^{1 / Q^*} \leq \Big(\int_{\Hn} |\nabla_{\H} u|^2\,\d\xi\Big)^{1 / 2},
\ee
where $S_{n}=\frac{ n \sqrt{\pi}}{(2^{2 n} n !)^{1 /(2(n+1))}}$ is the best constant. 
Set
\begin{equation}\label{standard-bubble}
\Lam (z,t):=((1+|z|^2)^2+t^2)^{(2-Q)/4},
\quad U_{0,1}:=2^n\Lam .
\end{equation}
As in Part~I \cite{QTWPartI}, the conformal factors of $F$ and $F_P$ are $H_F=2^n\Lam $ and $H_P=\Lam \circ\delta_{1/2}$, respectively; thus $H_F(0)=2^n$ and $H_P(0)=1$.
The translates and dilates of $U_{0,1}$ form the family
\begin{equation}\label{bubbles}
U_{a,\lam }(\xi)
=2^n\lam ^{(2-Q)/2}\Lam (\delta_{\lam ^{-1}}(a^{-1}\circ\xi)),
\quad a\in\Hn,\quad \lam >0.
\end{equation}
Each function in \eqref{bubbles} solves \eqref{maineq1} with $R\equiv1$, and the positive extremals of \eqref{Sobolevineq} are precisely their positive multiples.

Garofalo--Vassilev \cite[Theorem~1.5]{GV2001} classified finite-energy entire solutions under cylindrical symmetry on groups of Iwasawa type, including $\Hn$.
Catino--Li--Monticelli--Roncoroni \cite[Theorems~1.1 and~1.2]{CLMR2023} classified all positive solutions in $\H^1$ and positive solutions in $\Hn$, $n\geq2$, under a suitable decay condition at infinity. More recently, a complete classification of positive entire solutions to the critical CR Yamabe equation for $n\geq2$ has been obtained independently by Li--Liu--Ma \cite{LLM2026} and Li \cite{L2026}.

For a CR automorphism $\varphi$, let $J_{\varphi}=|\operatorname{det}\d\varphi|$ be its Jacobian relative to $\d\si $ and set $T_{\varphi}v=(v\circ\varphi)J_{\varphi}^{1/q}$, with $T_\varphi^{-1}=T_{\varphi^{-1}}$. Then $\varphi^* \theta_0=J_{\varphi}^{2 / Q} \theta_0$. Using \eqref{heisenberg-dilation}, we define the canonical maps and their ball parameters by
\begin{align}\label{varphi new}
\varphi_{P, \lam }=F_P \circ \delta_{\lam ^{-1}} \circ F_P^{-1}, \quad p=\frac{\lam -1}{\lam } P \in \mathbb{B}^{n+1}, \quad b=\frac{\lam ^2-1}{\lam ^2+1} P,
\end{align}
where $P \in \S^{2 n+1}$, $\lam  \geq 1$, and $\mathbb{B}^{n+1}$ is the unit ball in $\C^{n+1}$ and $\varphi_0=\mathrm{Id}$ (see \cite{CY1987,CY1991}). We write $\varphi_p$ and $\varphi_b$ for the map in \eqref{varphi new} parametrized by $p$ and $b$, respectively. We also write $\mathbb B_s^{n+1}=\{p\in\C^{n+1}:|p|<s\}$.

For later use, define the right-invariant vector fields by
\begin{equation}\label{right-vectorfields}
\overline X_i=\frac{\pa }{\pa  x_i}-2y_i\frac{\pa }{\pa  t},\quad
\overline Y_i=\frac{\pa }{\pa  y_i}+2x_i\frac{\pa }{\pa  t},\quad
T=\frac{\pa }{\pa  t}.
\end{equation}
Here $1\leq i\leq n$, $\overline X=(\overline X_1,\ldots,\overline X_n)$ and $\overline Y=(\overline Y_1,\ldots,\overline Y_n)$. Thus $\widetilde\nabla=(\overline X,\overline Y)$ and $\overline\nabla=(\widetilde\nabla,T)$ have $2n$ and $2n+1$ components, respectively, without the factor $1/2$ in $\nabla_\H$. We use Euclidean norms for these arrays. As in Part~I \cite{QTWPartI}, $\nabla$ in Heisenberg coordinates denotes the Euclidean gradient. The symbols $D$, $D_w$ and $D^2$ denote ordinary derivatives in the indicated real coordinates; for functionals they denote Fr\'echet derivatives.

We denote by $S_1^2(\S^{2n+1})$ the Folland--Stein space obtained as the completion of $C^\infty(\S^{2n+1})$ under the norm
\begin{align}\label{sb norm}
\|v\|_{S_1^2(\S^{2n+1})}:=\Big(\int_{\S^{2n+1}} v L_{\theta_0}v\,\d V_{\theta_0}\Big)^{1/2};
\end{align}
see \cite{FS1974}. For a solution $v$ of \eqref{maineq} with $R$ bounded, testing the equation against $v$ and using \eqref{eq:sharp-sphere-sobolev} show that a bound in $L^{Q^*}(\S^{2n+1})$ is equivalent to a bound in the norm \eqref{sb norm}.

\subsection{Intrinsic regularity and local flatness conditions}\label{sec:coefficients}
We now recall the intrinsic regularity and weighted flatness conditions used in Part~I \cite{QTWPartI}.

Using the array $Z$ introduced above, for each positive integer $k$ set
$$\mathscr B_k=\{Z_{a_1}\cdots Z_{a_j}:1\leq j\leq k,\ a_\ell\in\{1,\ldots,2n\},\ 1\leq\ell\leq j\}.$$

Let $\Om \subset\Hn$ be open. We recall the Folland--Stein spaces $\Gamma_\beta(\Om )$, following \cite{FS1974}. For $\beta\in(0,1)$, define
$$\Gamma_{\beta}(\Om)=\Big\{f\in L^{\infty}(\Om)\cap C^0(\Om):\sup_{\substack{\xi\in \Om, \zeta\in\H^n\backslash\{0\}\\ \xi\circ\zeta\in\Om }}\frac{|f(\xi)-f(\xi\circ\zeta)|}{|\zeta|^{\beta}}<\infty\Big\}.$$
If $\beta=1$, define $$\Gamma_{1}(\Om)=\Big\{f\in L^{\infty}(\Om)\cap C^0(\Om):\sup_{\substack{\xi\in \Om, \zeta\in\H^n\backslash\{0\}\\ \xi\circ\zeta, \xi\circ \zeta^{-1}\in\Om }}\frac{|f(\xi\circ\zeta)-2f(\xi)+f(\xi\circ\zeta^{-1})|}{|\zeta|}<\infty\Big\}.$$
If $\beta=k+\al$ with $k\in \mathbb{N}^+$  and $\al\in (0,1)$, define
$$\Gamma_{k+\al}(\Om)=\Big\{f\in L^{\infty}(\Om)\cap C^0(\Om):\mathscr{L}f\in \Gamma_{\al}(\Om)\,\text{ for $\mathscr{L}\in \mathscr{B}_k$} \Big\}.$$
All derivatives are understood in the distributional sense. For $0<\beta \leq 1$, the norm is the sum of the $L^{\infty}(\Om)$ norm and the corresponding supremum above. For higher orders, we set
$$
\|f\|_{\Gamma_{k+\al }(\Om )}=\|f\|_{L^{\infty}(\Om )}+\sum_{\mathscr{L} \in \mathscr{B}_k}\|\mathscr{L} f\|_{\Gamma_\al (\Om )} .
$$
%We say $u\in \Gamma_{\beta}(\Om,loc)$ if $\phi u\in \Gamma_{\beta}(\Om)$ for every $\phi\in C_{c}^{\infty}(\Om)$.

  On $\S^{2 n+1}$, these spaces and their norms are defined using a finite contact atlas and a subordinate smooth partition of unity; see \cite{FS1974}.
Using exponential coordinates, we identify $\Hn$ with $\R^{2n+1}$ and denote the Euclidean norm by $|\cdot|_{\mathrm{Eu}}$. The ordinary spaces $C^k$, $0\leq k\leq\infty$, refer to derivatives in these coordinates. See \cite{FS1974,F1975} for further details and regularity results.

To state the norm of $\mathcal A_\al$ from \cite[Definition~1.1]{QTWPartI}, write $Z^2u=(Z_aZ_bu)_{a,b=1}^{2n}$ and $ZTu=(Z_aTu)_{a=1}^{2n}$; thus $ZR=2\nabla_\H R$ and $Z^2R=4\nabla_\H^2R$.

Let $\al \in (0,1)$ and let $U\subset\Hn$ be a relatively compact open set. We write $R\in\mathcal A_\al (U)$ if $R\in C^1(U)$, every ordered derivative $Z_aZ_bR$ is represented by a bounded measurable function, and $TR\in\Gamma_\al (U)$. We also require finiteness of the norm
\be\label{eq:intrinsic-coefficient-norm}
 \|R\|_{\mathcal A_\al (U)}
=\|R\|_{L^\infty(U)}+\|ZR\|_{L^\infty(U)}+\sum_{a,b=1}^{2n}\|Z_aZ_bR\|_{L^\infty(U)}+\|TR\|_{\Gamma_\al (U)}.
\ee
On the sphere, the definition of $\mathcal A_\al$ and the norm \eqref{eq:intrinsic-coefficient-norm} are understood uniformly on fixed smaller balls in the centered charts \eqref{eq:normalized-cayley}. The local version is denoted by $\mathcal A_{\al,\mathrm{loc}}$. All distributional derivatives are taken along the indicated fields.

Both ordinary $C^{1,1}$ functions and functions in $\Gamma_{2+\al }$ belong locally to $\mathcal A_\al $; no bound on \(T^2R\) is required.

We use the condition $(*)'_\beta$ stated before Theorem~\ref{thm:0.5}, with the same uniform chart conventions as in \eqref{eq:intrinsic-coefficient-norm}.

For a homogeneous expansion at a full critical point, write $R \circ F_P=R(P)+Q_P^{(\beta)}+W_P$. The leading term is $C^1$ away from zero, is homogeneous under $\delta_r$, and satisfies
\be\label{2.7 new}
c|\xi|^{\beta-1} \leq |\widetilde{\nabla} Q_P^{(\beta)}|+|\xi||T Q_P^{(\beta)}| \leq C|\xi|^{\beta-1} \quad(\xi \neq 0) .
\ee
The first-derivative remainder required for the general homogeneous criterion is
\be\label{2.8 new}
|\xi|^{-\beta}|W_P|+|\xi|^{1-\beta}|\widetilde{\nabla} W_P|+|\xi|^{2-\beta}|T W_P| \to 0\quad\text{ as }\, \xi\to0.
\ee
%Under the coefficient assumptions of Theorem \ref{thm:2}, the following stronger remainder condition is sufficient for $(*)_\beta':$
%$$
%\rho^{-\beta}|W|+\rho^{1-\beta}|Z W|+\rho^{2-\beta}(|T W|+|Z^2 W|)+\rho^{3-\beta}|Z(T W)| \to  0, \quad \rho=|\xi| \to  0 .
%$$
The lower bound in \eqref{2.7 new}, together with \eqref{2.8 new}, isolates every full critical point. Hence $\operatorname{Crit}(R)$ is finite, and the local constants may be chosen uniformly.
%\begin{rem}
 %    These assumptions allow the radial example $R=R_0+c \rho^\beta, R_0>0, c \neq 0, \beta \geq 2+\al $. Homogeneity gives $|Z R| \leq C \rho^{\beta-1},|Z^2 R|+|T R| \leq C \rho^{\beta-2}$, and $|Z(T R)| \leq C \rho^{\beta-3}$ away from zero. The combined first derivative has the lower bound \eqref{2.7 new}, so $r=\rho$ verifies flatness. The size and horizontal derivative estimates for $T R$ give its contact Hölder regularity. When $2<\beta<4$, however, $T^2 R(0, t)$ is unbounded near zero. This explains the coefficient class used in CR dimension one. More generally, the full-critical homogeneous assumptions force $\beta \geq 2+\al $ : otherwise $T R(\delta_r \xi)=O(r^\al )$ and homogeneity imply $T Q_P^{(\beta)} \equiv 0$, contradicting \eqref{2.7 new} on the central axis.
%\end{rem}

\begin{rem}
The restrictions on the flatness exponent and on $c$ in
\eqref{eq:alln-model-expansion} have the following roles.
\begin{enumerate}[(i)]
\item For a nonzero $\beta$-homogeneous function $Q^{(\beta)}$,
the absolute convergence at infinity of
$\int_{\Hn}|Q^{(\beta)}|U_{0,1}^{Q^*}\,\d\xi$
is equivalent to
$\int_1^\infty r^{\beta-Q-1}\,\d r<\infty$,
and hence to $\beta<Q$. The restriction $\beta<Q$ is sharp
for these integrals; its necessity for solvability is not claimed.

\item The lower bound $\beta>Q-2$ arises from the interaction
estimates in the blow-up analysis. The condition $\beta>Q-2$
and the smallness assumption on $c$ in Theorem~\ref{thm:2}
are sufficient conditions. We do not claim
their necessity or the optimality of the smallness constant.

\item For fixed $\al\in(0,1)$, membership in $\mathcal A_\al$,
together with \eqref{2.7 new} and \eqref{2.8 new}, implies
$\beta\ge2+\al$. The inequality $\beta\ge2+\al$ holds automatically
when $n\ge2$ within the flatness ranges of the main theorems.
For $n=1$, Theorem~\ref{thm:0.6} requires $\beta>2$ and
therefore does not include the critical order $Q-2=2$.
\end{enumerate}
\end{rem}

%\medskip
%\noindent\emph{Remarks on the flatness assumptions.}
%For a nonzero $\beta$-homogeneous term, the absolute convergence at infinity %of
%$\int_{\Hn}|Q^{(\beta)}|U_{0,1}^{Q^*}\,\d\xi$ is equivalent to
%$\int_1^\infty r^{\beta-Q-1}\,\d r<\infty$, hence to $\beta<Q$.
%This threshold is sharp for the integrals used here, but is not asserted to %be necessary for solvability.
%The lower threshold $Q-2$ is the interaction order in the blow-up argument.
%Both $\beta>Q-2$ and the smallness assumption on the central coefficient in %Theorem~\ref{thm:2}
%are sufficient conditions; no general necessity or optimal smallness %constant is asserted.
%For a fixed $\al $, membership in $\mathcal A_\al $ and the full %nondegeneracy condition \eqref{2.7 new}, together with \eqref{2.8 new}, %imply $\beta\ge2+\al $.
%This restriction is automatic when $n\ge2$ within the flatness ranges of the %main theorems.

%%%%%%%%%%%%%%%%%%%%%%%%%%%%%%%%%%%%%%%%%%%%%%%%%%%%%%%%%%%%%%%%%%%%%%%%%%%%%%%%%%%%%%
\section{Proofs of Theorems~\ref{thm:1} and \ref{thm:1-2}}\label{sec:3}

We prove the existence results under symmetry assumptions by
minimizing the associated energy in the corresponding invariant
spaces. We first consider antipodal symmetry and then extend the
argument to finite-group symmetries.

Let $S_{\mathrm{as}}$ denote the space of antipodally symmetric
functions in $S_1^2(\S^{2n+1})$, and define
\[
\lam_{\mathrm{as}}(R)=\inf _{v \in S_{\mathrm{as}}}\Big\{\int_{\S^{2n+1}} v L_{\theta_0}(v)\,\d\si : \int_{\S^{2n+1}} R|v|^{Q^*}\,\d\si =1\Big\}.
\]
The following proposition shows that an energy level strictly
below the concentration threshold guarantees the existence
of a positive solution. We subsequently verify this strict
inequality under the flatness assumption of Theorem~\ref{thm:1}.

\begin{prop}\label{prop:2.1}
	Let $\al \in(0,1)$ and let $R\in\mathcal A_\al (\S^{2n+1})$ be antipodally symmetric and positive somewhere. If
	\be\label{2.1}
	\lam_{\mathrm{as}}(R)<\frac{\frac{n^2}{4}\om_{2n+1}^{\frac{2}{Q}} 2^{\frac{2}{Q}}}{(\max _{\S^{2n+1}} R)^{\frac{2}{Q^*}}},
	\ee
	then \eqref{maineq} admits a positive, antipodally
symmetric $C^2$ solution.
\end{prop}

We prove Proposition~\ref{prop:2.1} by subcritical
approximation. For
$p\in(1,\frac{Q+2}{Q-2})$, set
\[
\lam_{\mathrm{as}, p}(R)=\inf _{v \in  S_{\mathrm{as}}}\Big\{\int_{\S^{2n+1}} vL_{\theta_{0}}(v)\,\d\si : \int_{\S^{2n+1}} R|v|^{p+1}\,\d\si =1\Big\} .
\]
The subcritical variational problem admits a positive
minimizer, as recorded in the following lemma.
\begin{lem}\label{lem:2.1}
 Let $\al \in(0,1)$ and let $R\in\mathcal A_\al (\S^{2n+1})$ be antipodally symmetric and positive somewhere. Then $\lam_{\mathrm{as},p}(R)$
is attained by a positive, antipodally symmetric function
$v_p\in C^2(\S^{2n+1})$ satisfying
 \be\label{2.3}
 L_{\theta_0}(v_p)=\lam_{\mathrm{as}, p}(R) R v_p^p \quad \text { and } \quad \int_{\S^{2n+1}} R v_p^{p+1}\,\d\si =1.
 \ee
\end{lem}

\begin{proof}
The direct method of the calculus of variations, together with
\[
\int_{\S^{2n+1}}|v|L_{\theta_0}(|v|)\,\d\si
\leq \int_{\S^{2n+1}}vL_{\theta_0}(v)\,\d\si,
\quad v\in S_1^2(\S^{2n+1}),
\]
yields a nonnegative minimizer $v_p$ satisfying \eqref{2.3}.
Subcritical bootstrap estimates \cite{JL1987} give
$v_p\in L^\infty(\S^{2n+1})$, and the Harnack inequality
\cite{CGL1993} implies $v_p>0$. Applying
\cite[Lemma~2.4]{QTWPartI} in Cayley charts then yields
$v_p\in C^2(\S^{2n+1})$.
\end{proof}

\begin{proof}[Proof of Proposition~\ref{prop:2.1}]
 We first establish
 \be\label{claim1}
 \limsup _{p \to  \frac{Q+2}{Q-2}} \lam_{\mathrm{as}, p}(R) \leq \lam_{\mathrm{as}}(R) .
 \ee
 Indeed, for any $\var >0$, there exists a nonnegative function $v \in S_{\mathrm{as}}$ such that
\[
 \int_{\S^{2n+1}} v L_{\theta_0}(v)\,\d\si <\lam_{\mathrm{as}}(R)+\var \quad \text { and } \quad \int_{\S^{2n+1}} R v^{Q^*}\,\d\si =1 .
 \]
 Let $V_p:=\int_{\S^{2n+1}} R v^{p+1}\,\d\si $. Since $\lim _{p \to  \frac{Q+2}{Q-2}} V_p=\int_{\S^{2n+1}} R v^{Q^*}\,\d\si =1$, we have, for $p$ sufficiently close to $\frac{Q+2}{Q-2}$,
 \[
 \lam_{\mathrm{as}, p}(R) \leq \int_{\S^{2n+1}} \frac{v}{V_p^{1 /(p+1)}} L_{\theta_0}\Big(\frac{v}{V_p^{1 /(p+1)}}\Big)\,\d\si  \leq \lam_{\mathrm{as}}(R)+2 \var  .
\]
Letting $\var\to0$ proves \eqref{claim1}.

Choose a sequence $p_i\to\frac{Q+2}{Q-2}$ such that $\lam_{\mathrm{as},p_i}(R)\to\lam\leq\lam_{\mathrm{as}}(R)$, and let $v_i$ be the corresponding minimizers given by Lemma~\ref{lem:2.1}. The sequence $\{v_i\}$ is bounded
in $S_1^2(\S^{2n+1})$. Passing to a subsequence, we may
assume that
$v_i\rightharpoonup v$ weakly in
$S_1^2(\S^{2n+1})$, where $v\geq0$.

We distinguish two cases.

\textbf{Case 1:}  $v\not\equiv 0$.  The normalization in \eqref{2.3} and the Sobolev inequality
imply $\lam>0$. Passing to the limit in the weak formulation
of \eqref{2.3}, we obtain
$$
L_{\theta_0} v=\lam  R v^{(Q+2) /(Q-2)}.
$$
Subelliptic bootstrap estimates and the Harnack inequality yield $v\in L^\infty(\S^{2n+1})$ and $v>0$; see, e.g., \cite{JL1987,CGL1993}. The regularity result \cite[Lemma 2.4]{QTWPartI} gives $v\in C^2(\S^{2n+1})$. Hence
$$
V=\Big(\frac{4 \lam }{n^2}\Big)^{n / 2} v
$$
is a positive, antipodally symmetric solution of \eqref{maineq}.

\textbf{Case 2:} $v\equiv0$.  We show that this case contradicts \eqref{2.1}.   The compact embedding into $L^2(\S^{2n+1})$ gives $v_i\to0$ in $L^2(\S^{2n+1})$. If $\|v_i\|_{L^\infty(\S^{2n+1})}\le M$ for some fixed $M\ge1$, then
\[
1=\int_{\S^{2n+1}}Rv_i^{p_i+1}\,\d\si 
\le \|R\|_{L^\infty(\S^{2n+1})}M^{p_i-1}
\int_{\S^{2n+1}}v_i^2\,\d\si \to0,
\]
a contradiction.

Consequently, after passing to a subsequence, there are
maximum points $\zeta_i\in\S^{2n+1}$ such that $v_i(\zeta_i)=\max_{\S^{2n+1}}v_i\to\infty$ and  $\zeta_i\to\widehat\zeta\in \S^{2n+1}$.
 Without loss of generality, we assume that $\widehat \zeta$ is the north pole. Via the inverse Cayley transform  \eqref{Cayley}, equation \eqref{2.3} becomes
 \be \label{eq:subcritical in plane1}
 -\Delta_{\H}u_{i}(\xi)=\lam_{\mathrm{as},p_i}(R)R\circ F(\xi)H_F^{\tau_i}(\xi)u_{i}^{p_i}(\xi),\quad \xi \in\Hn ,
 \ee
 where $u_{i}(\xi)=(2|J_F|)^{(Q-2)/(2Q)}v_{i}\circ F(\xi)$, $H_F(\xi)=(2|J_F|)^{(Q-2) / (2Q)}$ and $\tau_i=\frac{Q+2}{Q-2}-p_i$, as in Part~I \cite{QTWPartI}.

 Write $\xi_i=F^{-1}(\zeta_i)\to0$ and
$m_i=u_i(\xi_i)\to\infty$. For all sufficiently large
$i$, the maximality of $v_i(\zeta_i)$ and the conformal
factor give
$u_i(\xi)\leq2m_i$ for all $\xi\in\Hn$.
 Define
 \[
 \widetilde u_{i}(\xi)=m_{i}^{-1}u_{i}(\xi_{i}\circ\delta_{m_i^{(1-p_i)/2}}(\xi)).
 \]
 Rescaling \eqref{eq:subcritical in plane1} gives, for all $\xi\in\Hn$,
 \[
 -\Delta_{\H}\widetilde u_{i}(\xi)=\lam_{\mathrm{as},p_i}(R)R\circ F(\xi_{i}\circ\delta_{m_i^{(1-p_i)/2}}(\xi))H_F^{\tau_i}(\xi_{i}\circ\delta_{m_i^{(1-p_i)/2}}(\xi))\widetilde u_{i}^{p_i}(\xi).
 \]
 The rescaled functions multiplying $\widetilde u_i^{p_i}$ are locally uniformly bounded in $\mathcal A_\al $. Since $\widetilde u_i(0)=1$ and
$0<\widetilde u_i\leq2$, the local Harnack inequality
\cite[Lemma~3.3]{QTWPartI} gives positive lower bounds
on compact subsets of $\Hn$.
By \cite[Lemma~2.4]{QTWPartI},
$\{\widetilde u_i\}$ is bounded in
$C^{2,\iota}_{\mathrm{loc}}(\Hn)$ for some $\iota>0$.
Since $\{v_{i}\}$ is bounded in $S_1^2(\S^{2n+1})$, $\{\widetilde u_{i}\}$ is bounded in $\dot S_1^2(\Hn)$. Thus, after passing to a subsequence, there exists $u\in C^{2}(\Hn)\cap  \dot S_1^2(\Hn)$ such that $\widetilde u_{i}\to u$ in $C^{2}_{\mathrm{loc}}(\Hn)$, $\widetilde u_{i}\rightharpoonup u$ weakly in $\dot S_1^2(\Hn )$. The limit satisfies $u(0)=1$ and
 \be \label{eq:critical in plane1}
 -\Delta_{\H} u=\lam R(\widehat \zeta)u^{\frac{Q+2}{Q-2}}.
 \ee
 Here $\dot S_1^2(\Hn)=E$ is the homogeneous energy space of Section~\ref{sec2.2}.
 Testing \eqref{eq:critical in plane1} against $u\in E\setminus\{0\}$ yields $\lam R(\widehat\zeta)>0$. Since $\lam\ge0$, both $\lam$ and $R(\widehat\zeta)$ are positive.

For $S>0$, set $\Om_i:=F(B_{m_i^{(1-p_i)/2}S}(\xi_i))$. Then
 \begin{align*}	\int_{\Om_i}Rv_{i}^{p_i+1}\,\d\si =&\,\frac{1}{2}\int_{B_{m_i^{(1-p_i)/2}S}(\xi_i)}R\circ F(\xi) H_F^{\tau_i}u_{i}^{p_i+1}\,\d\xi\\
 =&\,\frac{1}{2}m_{i}^{n\tau_i}\int_{B_S}R\circ F(\xi_{i}\circ \delta_{m_i^{(1-p_i)/2}}\xi)H_F^{\tau_i}(\xi_{i}\circ\delta_{m_i^{(1-p_i)/2}}(\xi))\widetilde u_{i}^{p_i+1}(\xi)\,\d\xi\\
 \ge&\, \frac{1}{2}R(\widehat \zeta)\int_{B_S}u^{Q^*}\,\d\xi+o(1),
 \end{align*}
 as $\ p_i\to\frac{Q+2}{Q-2}$, where we used the facts that  $R$ is positive near $\widehat \zeta$, $\tau_i\to 0$ and $\widetilde u_{i}\to u$ in $C^{2}_{\mathrm{loc}}(\Hn)$.
 Since $R$ and $v_{i}$ are antipodally symmetric,
 we have, by taking $\delta$ small and $S$ sufficiently large,
 \be\label{norm less1}
 1=\int_{\S^{2n+1}}Rv_{i}^{p_i+1}\,\d\si \ge 2\int_{\mathcal{B}(\widehat{\zeta},\delta)}Rv_{i}^{p_i+1}\,\d\si +\int_{\{R<0\}}Rv_{i}^{p_i+1}\,\d\si ,
 \ee
where $\{R<0\}$ denotes the set $\{\zeta \in \S^{2n+1}:R(\zeta)<0\}$.

 We next show that
 \be\label{claim2}
 \int_{\{R<0\}}Rv_{i}^{p_i+1}\,\d\si \to 0\quad \text{ as }\, p_i\to\frac{Q+2}{Q-2}.
 \ee
 Indeed, set $\Om _a=\{\zeta\in\S^{2n+1}:R(\zeta)<-a\}$ for $a>0$. On $\Om _{\var/2}$, equation \eqref{2.3} gives
$$
-\Delta_{\theta_0} v_i=\lam _{\mathrm{as}, p_i}(R) R v_i^{p_i}-\frac{n^2}{4} v_i \leq 0 .
$$
Local subsolution estimates therefore yield
$$
\sup _{\Om _{\var}} v_i \leq C(\var)\|v_i\|_{L^2(\S^{2 n+1})} \to  0,
$$
where the last convergence follows from the compact Sobolev embedding.
 Thus when $p_i$ is sufficiently close to $\frac{Q+2}{Q-2}$,
 \[\int_{\Om_{\var}}|R|v_{i}^{p_i+1}\,\d\si <\var.
 \]

On the other hand, H\"older's inequality and the Sobolev inequality give
 \[
 \int_{\{\zeta\in \S^{2n+1}: R(\zeta)\in [-\var,0)\}}|R|v_{i}^{p_i+1}\,\d\si <C(n)\var\|v_{i}\|_{L^{Q^*}(\S^{2n+1})}^{p_i+1}\leq C(n,\lam_{\mathrm{as}})\var,
 \]
 which finishes the proof of \eqref{claim2}.

 Combining \eqref{norm less1} with \eqref{claim2} and letting $S\to\infty$, we obtain
 \be\label{norm less2}
 1\ge R(\widehat \zeta)\int_{\Hn}u^{Q^*}\,\d\xi.
 \ee
 The Sobolev inequality \eqref{Sobolevineq}, the equation \eqref{eq:critical in plane1}, and \eqref{norm less2} give
 \begin{align*}
2^{2/Q}\frac{n^2}{4}\om_{2n+1}^{\frac{2}{Q}}\leq \frac{\int_{\Hn}|\nabla_{\H}u|^2\,\d\xi}{(\int_{\Hn}
  u^{Q^*}\,\d\xi)^{2/Q^*}}=&\,\lam R(\widehat \zeta)\Big(\int_{\Hn}u^{Q^*}\,\d\xi\Big)^{(Q^*-2)/Q^*}
 \\
 \leq& \,\lam_{\mathrm{as}}(R)R(\widehat \zeta)^{1-(Q^*-2)/Q^*}\\
 \leq&\, \lam_{\mathrm{as}}(R)\Big(\max_{\S^{2n+1}} R\Big)^{2/Q^*},
 \end{align*}
 which contradicts \eqref{2.1}. Thus Case 2 cannot occur, and
Proposition~\ref{prop:2.1} follows.
\end{proof}

We now show that the flatness assumption in
Theorem~\ref{thm:1} ensures the strict energy
inequality \eqref{2.1}.
\begin{prop}\label{prop:2.2}
	Let $\al \in(0,1)$ and let $R\in\mathcal A_\al (\S^{2n+1})$ be antipodally symmetric and positive somewhere. If there exists a maximum point of $R$ at which $R$ has flatness order greater than $Q-2$, then \eqref{2.1} holds.
\end{prop}

\begin{proof}
 Let $\zeta_1$ be a maximum point of $R$ at which $R$ has flatness order greater than $Q-2$. Suppose $\zeta_2$ is the antipodal point of $\zeta_1$.
 Here $\beta>1$ is a bubble parameter, unrelated to the flatness exponent. For $i=1,2$, define
 \be\label{2.9}
 v_{i,\beta}(\zeta)=\Big(\frac{\sqrt{\beta^2-1}}{|\beta-\langle \zeta,\zeta_i \rangle|}\Big)^{\frac{Q-2}{2}}.
 \ee
 The bubble formula \eqref{2.9} gives
$$
L_{\theta_0}(v_{i,\beta})=\frac{n^2}{4}v_{i,\beta}^{\frac{Q+2}{Q-2}}\quad\text{ and }\quad\int_{\S^{2n+1}}v_{i,\beta}^{Q^*}\,\d\si =\om_{2n+1}.
$$
 Let $v_{\beta}=v_{1,\beta}+v_{2,\beta}$,
 which is antipodally symmetric. Then
 \begin{align*}
 \int_{\S^{2n+1}}v_{\beta}L_{\theta_0}(v_{\beta})\,\d\si =&\,\frac{n^2}{4}\int_{\S^{2n+1}}\Big(\sum_{i=1}^2 v_{i,\beta}^{\frac{Q+2}{Q-2}}\sum_{j=1}^2 v_{j,\beta}\Big)\,\d\si \\
 =&\,\frac{n^2}{4}\int_{\S^{2n+1}}\Big(\sum_{i=1}^2 v_{i,\beta}^{Q^*}+2 v_{1,\beta}^{\frac{Q+2}{Q-2}}v_{2,\beta}\Big)\,\d\si \\
 =&\,\frac{n^2}{2}\om_{2n+1} \Big(1+\om_{2n+1}^{-1}\int_{\S^{2n+1}} v_{1,\beta}^{\frac{Q+2}{Q-2}}v_{2,\beta}\,\d\si \Big).
 \end{align*}
A change of variables gives
$$
 \int_{\S^{2n+1}}v_{1,\beta}^{\frac{Q+2}{Q-2}}v_{2,\beta}\,\d\si =A(\beta-1)^{\frac{Q-2}{2}}+o((\beta-1)^{\frac{Q-2}{2}}),
$$
for $\beta$ close to $1$, where $A$ is a positive constant.
 Choose a sufficiently small neighborhood $V_1$ of $\zeta_1$ and let $V_2=\{\zeta\in\S^{2n+1}: -\zeta\in V_1\}$. Then $R$ is positive in $V_1\cup V_2$ and
 \begin{align*}
 \int_{\S^{2n+1}}R v_{\beta}^{Q^*}\,\d\si 
 =&\, \int_{\cup_{i=1}^2 V_i}R(v_{1,\beta}+v_{2,\beta})^{Q^*}\,\d\si +\int_{\S^{2n+1}\backslash \cup_{i=1}^2 V_i}Rv_{\beta}^{Q^*}\,\d\si \\
 =&\, 2\int_{V_1}R(v_{1,\beta}+v_{2,\beta})^{Q^*}\,\d\si +\int_{\S^{2n+1}\backslash \cup_{i=1}^2 V_i}Rv_{\beta}^{Q^*}\,\d\si \\
 \geq&\,  2\int_{V_1}R(v_{1,\beta}^{Q^*}+Q^*v_{1,\beta}^{\frac{Q+2}{Q-2}} v_{2,\beta})\,\d\si +\int_{\S^{2n+1}\backslash \cup_{i=1}^2 V_i}Rv_{\beta}^{Q^*}\,\d\si .
 \end{align*}
In Cayley coordinates centered at $\zeta_1$, the flatness assumption in Definition~\ref{def:value-flatness} reads
 \[
 R(\xi)=R(\zeta_1)+o(|\xi|^{Q-2}).
\]
A direct computation also gives \[
\int_{\S^{2n+1}\backslash \cup_{i=1}^2 V_i}Rv_{\beta}^{Q^*}\,\d\si =o((\beta-1)^{\frac{Q-2}{2}}).
\]
 Thus,
 \begin{align*}
 \int_{\S^{2n+1}}Rv_{\beta}^{Q^*}\,\d\si \geq&\, 2R(\zeta_1)\int_{\S^{2n+1}}v_{1,\beta}^{Q^*}\,\d\si +2Q^* R(\zeta_1)A(\beta-1)^{\frac{Q-2}{2}}+o((\beta-1)^{\frac{Q-2}{2}})\\
 =&\,2R(\zeta_1)\om_{2n+1}(1+Q^*\om_{2n+1}^{-1}A(\beta-1)^{\frac{Q-2}{2}}+o((\beta-1)^{\frac{Q-2}{2}})),
 \end{align*}
 for $\beta$ close to $1$. Hence
\[
 \frac{\int_{\S^{2n+1}}v_{\beta}L_{\theta_0}(v_{\beta})\,\d\si }{(\int_{\S^{2n+1}}Rv_{\beta}^{Q^*}\,\d\si )^{\frac{2}{Q^*}}}\leq
 \frac{\frac{n^2}{4}\om_{2n+1}^{\frac{2}{Q}}2^{\frac{2}{Q}}}{R(\zeta_1)^{\frac{2}{Q^*}}}
 \Big(1-\frac{A}{\om_{2n+1}}(\beta-1)^{\frac{Q-2}{2}}+o((\beta-1)^{\frac{Q-2}{2}})\Big),
\]
 which proves \eqref{2.1}.
\end{proof}

\begin{proof}[Proof of Theorem~\ref{thm:1}]
	By Proposition~\ref{prop:2.2}, the strict inequality
\eqref{2.1} holds. The conclusion follows from
Proposition~\ref{prop:2.1}.
\end{proof}

We now prove Theorem~\ref{thm:1-2} by extending the preceding
variational argument to finite-group symmetries.
Let $S_G$ denote the subspace of $G$-invariant functions in
$S_1^2(\S^{2n+1})$, and define
\[
\lam_G(R)
=\inf_{v\in S_G}
\Big\{
\int_{\S^{2n+1}}vL_{\theta_0}(v)\,\d\si:
\int_{\S^{2n+1}}R|v|^{Q^*}\,\d\si=1
\Big\}.
\]
As in the antipodally symmetric case, we first establish an
existence criterion based on a strict energy inequality and
then verify this inequality using $G$-invariant test functions.

\begin{prop}\label{prop:2.3}
	Let $G$ be a finite subgroup of the CR isometry group of $\S^{2n+1}$. Let $\al \in(0,1)$ and let $R\in\mathcal A_\al (\S^{2n+1})$ be a positive and $G$-invariant function. If 
	\[
	\lam_{G}(R)<\frac{\frac{n^2}{4}\om_{2n+1}^{\frac{2}{Q}}|O_{G}(\zeta)|^{\frac{2}{Q}}}{R(\zeta)^{\frac{Q-2}{Q}}},\quad \forall\, \zeta\in\S^{2n+1},
	\]
then \eqref{maineq} admits a positive $G$-invariant
$C^2$ solution.
\end{prop}
\begin{proof}
The subcritical approximation and blow-up argument in the
proof of Proposition~\ref{prop:2.1} extend to the
$G$-invariant subspace. Indeed, concentration at a point
$\zeta$ forces concentration at every point of its orbit
$O_G(\zeta)$, giving the corresponding energy threshold
with the factor $|O_G(\zeta)|^{2/Q}$.
The assumed strict inequality excludes concentration, and
the nonzero weak limit yields a positive $G$-invariant
solution of \eqref{maineq}.
\end{proof}

The flatness assumption provides the strict inequality
required in Proposition~\ref{prop:2.3}.
\begin{prop}\label{prop:2.4}
	Let $G$ be a finite subgroup of the CR isometry group of $\S^{2n+1}$ with no global fixed point.
	Let $\al \in(0,1)$ and let $R\in\mathcal A_\al (\S^{2n+1})$ be a positive and $G$-invariant function. If $R$ has flatness order greater than $Q-2$ at some $\zeta_1\in\S^{2n+1}$, then
\[
	\lam_{G}(R)<\frac{\frac{n^2}{4}\om_{2n+1}^{\frac{2}{Q}}|O_{G}(\zeta_1)|^{\frac{2}{Q}}}{R(\zeta_1)^{\frac{Q-2}{Q}}}.
\]
\end{prop}
\begin{proof}
Following the construction in Proposition~\ref{prop:2.2},
we consider the $G$-invariant test function
\[
v_{\beta}=\sum\limits_{i=1}^m v_{i,\beta},
\]
where $m=|O_G(\zeta_1)|$, $O_{G}(\zeta_1)=\{\zeta_1,\dots,\zeta_m\}$, $\zeta_i=g_i(\zeta_1)$ for some $g_i\in G$, $g_1=\operatorname{Id}$, $v_{i,\beta}:=v_{1,\beta}\circ g_i^{-1}$ and $v_{1,\beta}$ is as in \eqref{2.9}.
\end{proof}

\begin{proof}[Proof of Theorem~\ref{thm:1-2}]
The conclusion follows from Propositions~\ref{prop:2.3} and~\ref{prop:2.4}.
\end{proof}

\section{Perturbation and degree reduction}\label{sec:perturbation}
We now consider the CR Nirenberg problem without symmetry
assumptions. Following the perturbation methods of
\cite{CY1987,CY1991,Li1995}, we first construct a constrained
minimizer after separating the CR conformal directions.
We then establish the degree comparison in
Theorem~\ref{thm:3.1} and compute the degree associated
with the homogeneous expansions in
\eqref{eq:alln-model-expansion}.

From this section onward, we use the normalized energy norm
\begin{equation}\label{norm 4}
\|v\|^2
:=\dashint_{\S^{2n+1}}vL_{\theta_0}v\,\d\si.
\end{equation}
In view of \eqref{sb norm}, we have
\[
\|v\|_{S_1^2(\S^{2n+1})}^2
=2^{2n+1}n!\,\om_{2n+1}\|v\|^2.
\]
Norms of functional derivatives refer to the corresponding
dual or operator norms. Unless otherwise specified, Lebesgue
norms on $\S^{2n+1}$ are taken with respect to $\d\si$.

\subsection{Conformal normalization and constrained minimization}
We begin by recalling the conformal invariance of the energy
and the critical Lebesgue norm. Let $T_\varphi$ denote the
conformal action defined in Section~\ref{sec2.2}.

\begin{lem}\label{lempe1}
	For every CR automorphism $\varphi$ of $\S^{2n+1}$
and every $v\in S_1^2(\S^{2n+1})$, we have
	\[
	\int_{\S^{2n+1}}T_{\varphi}vL_{\theta_0}(T_{\varphi}v)\,\d V_{\theta_0}=\int_{\S^{2n+1}} vL_{\theta_0}(v)\,\d V_{\theta_0},
	\]
	and
	\[
	\int_{\S^{2n+1}}|T_{\varphi}v|^{Q^*}\,\d V_{\theta_0}=\int_{\S^{2n+1}}|v|^{Q^*}\,\d V_{\theta_0}.
	\]
\end{lem}

\begin{proof}
The second identity follows directly from the definition
of $T_\varphi$ and a change of variables. To prove the
first identity, recall that
$\varphi^*\theta_0=J_\varphi^{2/Q}\theta_0$,
where $J_\varphi=|\det\d\varphi|$.
Conformal covariance of the sub-Laplacian gives
\begin{align*}
	\int_{\S^{2n+1}} vL_{\theta_0}(v)\,\d V_{\theta_0}=&\, \int_{\S^{2n+1}} v \circ \varphi L_{\varphi^*\theta_{0}}(v\circ \varphi)\,\d V_{\varphi^*\theta_{0}}\\
	=&\,\int_{\S^{2n+1}} v \circ \varphi L_{|\det \d\varphi|^{2/Q}\theta_{0}}(v\circ \varphi)|\det \d\varphi|\,\d V_{\theta_0}\\
	=&\,\int_{\S^{2n+1}} v \circ \varphi |\det \d\varphi|^{-\frac{Q+2}{2Q}} L_{\theta_0}(v\circ \varphi |\det \d\varphi|^{\frac{Q-2}{2Q}})|\det \d\varphi|\,\d V_{\theta_0}\\
	=&\,  \int_{\S^{2n+1}}T_{\varphi}vL_{\theta_0}(T_{\varphi}v)\,\d V_{\theta_0}.
	\end{align*}
The identities extend from smooth functions to
$S_1^2(\S^{2n+1})$ by density.
\end{proof}
For later use, we record an explicit formula for the
conformal maps $\varphi_{P,\lam}$ defined in \eqref{varphi new}.
\begin{lem}\label{varphi}
    For any $\zeta\in\S^{2n+1}$, $P\in\S^{2n+1}$, $\lam >0$, we have
   \[
    \varphi_{P,\lam }(\zeta)=D^{-1}\lbrace 2\lam (\zeta-\langle \zeta,P\rangle P)+[\lam ^2(1+\langle \zeta,P\rangle)-(1-\langle \zeta,P\rangle)]P\rbrace,
    \]
    where \[
    D=\lam ^2(1+\langle \zeta,P \rangle)+(1-\langle \zeta,P \rangle).
  \]
\end{lem}
\begin{proof}
  The formula follows by substituting the Cayley transform
\eqref{Cayley} into \eqref{varphi new} and simplifying.
\end{proof}

We next impose a normalization to separate the conformal
directions. Define
\begin{equation}\label{eq:balanced-spaces}
\begin{aligned}
\M=&\, \Big\{v \in S_1^2(\S^{2n+1}):\dashint_{\S^{2n+1}}|v|^{Q^*}\,\d\si =1\Big\},\\
\M_0=&\, \Big\{v\in \M:\dashint_{\S^{2n+1}} \zeta|v|^{Q^*}\,\d\si =0\Big\}.
\end{aligned}
\end{equation}
For $w\in\M_0$ and $p\in\mathbb B^{n+1}$, set
\[
\varpi(w,p)=T_{\varphi_p}^{-1}w.
\]
The following lemma shows that each function in $\M$ admits
a unique representation of this form.

\begin{lem}\label{lem:balanced-decomposition}
The map	$\varpi: \M_0\times \mathbb{B}^{n+1}\to \M$ is a homeomorphism with respect to the $S_1^2(\S^{2n+1})$ topology.
\end{lem}

\begin{proof}
For $v \in \M$ and $p \in \mathbb B^{n+1}$,  define
$$
F(v, p)=\dashint_{\S^{2 n+1}} \varphi_p^{-1}(\zeta)|v(\zeta)|^{Q^*}\,\d\si =\dashint_{\S^{2 n+1}} \zeta|T_{\varphi_p} v|^{Q^*}\,\d\si  .
$$
We first show that $F(v,p)=0$ for some
$p\in\mathbb B^{n+1}$.
By Lemma~\ref{varphi} and the absolute continuity
of the integral,
$$
\sup _{P \in \S^{2 n+1}}|F(v, r P)+P|=o(1),
$$
as $r$ tends to 1. Indeed, the convergence of $\varphi_{r P}^{-1}$ to $-P$ is uniform outside any fixed cap about $P$, and the integrals of $|v|^{Q^*}$ over such caps are uniformly small when their radii are small. Hence, for $r<1$ sufficiently close to 1,
$$
\operatorname{deg}(F(v,\cdot),\mathbb B_r^{n+1},0)=\operatorname{deg}(-\mathrm{Id},\mathbb B_r^{n+1},0)=1.
$$
The nonzero degree gives $F(v,p)=0$ for some $p\in\mathbb B^{n+1}$. The constraints \eqref{eq:balanced-spaces} then give $T_{\varphi_p}v\in\M_0$, so $\varpi$ is onto.

Next, suppose that
$$
\varpi(w_1, p_1)=\varpi(w_2, p_2) .
$$
The decomposition of CR automorphisms of the ball gives $U \in U(n+1), P \in \S^{2 n+1}$, and $\lam  \geq 1$ such that
$$
\varphi_{p_1}^{-1} \circ \varphi_{p_2}=U \circ \varphi_{P, \lam }, \quad w_2=T_{\varphi_{P, \lam }} T_U w_1 .
$$
Since unitary transformations preserve $\M_0$, both $T_U w_1$ and $w_2$ belong to $\M_0$. The formula in Lemma~\ref{varphi} gives
$$
\operatorname{Re}\langle\varphi_{P, \lam }^{-1}(\zeta), P\rangle<\operatorname{Re}\langle\zeta, P\rangle, \quad \lam >1, \quad \zeta \neq \pm P .
$$

If $\lam >1$, a change of variables yields
\begin{align*}
0 & =\dashint_{\S^{2 n+1}} \operatorname{Re}\langle\zeta, P\rangle|w_2|^{Q^*}\,\d\si  \\
& =\dashint_{\S^{2 n+1}} \operatorname{Re}\langle\varphi_{P, \lam }^{-1}(\zeta), P\rangle|T_U w_1|^{Q^*}\,\d\si  \\
& <\dashint_{\S^{2 n+1}} \operatorname{Re}\langle\zeta, P\rangle|T_U w_1|^{Q^*}\,\d\si =0,
\end{align*}
a contradiction. Thus $\lam=1$, and consequently $\varphi_{p_2}=\varphi_{p_1} \circ U$. The value at the origin of the ball is
$$
\varphi_{P, \lam }(0)=\frac{\lam ^2-1}{\lam ^2+1} P
$$
so evaluating $\varphi_{p_2}=\varphi_{p_1}\circ U$ at $0$ gives $p_1=p_2$. Hence $w_1=w_2$, and $\varpi$ is one-to-one.

Finally, let $v_j, v \in \M$ satisfy $\|v_j-v\|_{S_1^2(\S^{2n+1})}=o(1)$, and let
$$
F(v_j, p_j)=F(v, p)=0 .
$$
The Sobolev embedding gives
$$
\sup _{p' \in \mathbb B^{n+1}}|F(v_j, p')-F(v, p')| \leq \dashint_{\S^{2 n+1}}|\,|v_j|^{Q^*}-|v|^{Q^*}|\,\d\si =o(1) .
$$
The uniform limit $F(v,rP)=-P+o(1)$ keeps $p_j$ in a compact subball. Uniqueness then gives $p_j\to p$. The same argument proves that $p$ depends continuously on the normalized density in $L^1(\S^{2n+1})$.
Lemma~\ref{lempe1} and approximation by smooth functions give
\[
\|T_{\varphi_{p_j}}v_j-T_{\varphi_p}v\|_{S_1^2(\S^{2n+1})}
\leq\|v_j-v\|_{S_1^2(\S^{2n+1})}+\|T_{\varphi_{p_j}}v-T_{\varphi_p}v\|_{S_1^2(\S^{2n+1})}
=o(1).
\]
Thus $\varpi^{-1}$ is continuous. The continuity of $\varpi$ follows in the same way from $T_{\varphi_p}^{-1}=T_{\varphi_p^{-1}}$.
\end{proof}

The decomposition in Lemma~\ref{lem:balanced-decomposition}
allows us to study the energy in terms of the normalized
function and the conformal parameters. For a positive function
$R$ and $v\in S_1^2(\S^{2n+1})\setminus\{0\}$, define
\begin{equation}\label{eq:energy-quotient}
E_{R}(v)=\frac{\dashint_{\S^{2n+1}}vL_{\theta_0}(v)\,\d\si }{(\dashint_{\S^{2n+1}}R|v|^{Q^*}\,\d\si )^{\frac{Q-2}{Q}}}.
\end{equation}
By the sharp Sobolev inequality
\eqref{eq:sharp-sphere-sobolev}, $E_R$ is bounded below
when $R$ is positive and bounded. If $R$ is constant,
the conformal invariance in Lemma~\ref{lempe1} gives
$E_R(\varpi(w,p))=E_R(w)$.
Moreover, the classification of extremals for
\eqref{eq:sharp-sphere-sobolev}, together with
Lemma~\ref{lem:balanced-decomposition}, yields $\min_{w\in\M_0}E_1=L_{\theta_0}(1)$,  with equality attained precisely at $w=\pm1$.

The constraints in \eqref{eq:balanced-spaces} define
$C^2$ submanifolds $\M$ and $\M_0$ of the Hilbert space
$S_1^2(\S^{2n+1})$.
To study the second variation of $E_1$ at $1$, we recall
the spherical harmonic decomposition on the CR sphere;
see, e.g., \cite{BFM2013}.

The space $L^2(\S^{2n+1})$, equipped with the inner product
\[
(F,G)=\int_{\S^{2n+1}}F\overline G\,\d\si,
\]
admits the orthogonal decomposition
\[
L^2(\S^{2n+1})
=\bigoplus_{j,k\geq0}\mathcal H_{j,k},
\]
where $\mathcal H_{j,k}$ consists of the restrictions
to $\S^{2n+1}$ of harmonic polynomials on $\C^{n+1}$
that are homogeneous of degrees $j$ and $k$ in
$\zeta$ and $\overline\zeta$, respectively.
Its dimension is
\[
\dim\mathcal H_{j,k}
=m_{j,k}
:=\frac{(j+n-1)!(k+n-1)!(j+k+n)}
{n!(n-1)!j!k!}.
\]
In particular, $\mathcal H_{0,0}$ consists of the
constant functions. Throughout the following tangent-space
identities, orthogonal complements are taken in
$L^2(\S^{2n+1})$ within the real energy space, and
sums of complex harmonic spaces are understood through
their real-valued elements.
\begin{lem}\label{lempe2}
	The tangent space of $\M$ at $v=1$ is
\[
	T_1\M=\Big\{\phi\in S_1^2(\S^{2n+1}):\int_{\S^{2n+1}}\phi\,\d\si =0\Big\}={\mathcal{H}_{0,0}}^\perp.
\]
\end{lem}

\begin{lem}\label{lempe3}
	The tangent space of $\M_0$ at $v=1$ is
	\[
	T_1\M_0=( \mathcal{H}_{0,0} \oplus \mathcal{H}_{1,0} \oplus \mathcal{H}_{0,1}) ^{\perp} .
	\]
\end{lem}

Both identities follow by differentiating the constraints
in \eqref{eq:balanced-spaces} at $1$ and identifying the
corresponding directions through the spherical harmonic
decomposition.

The implicit function theorem provides local coordinates
on $\M_0$ near $1$, as in \cite[Lemma~6.4]{Li1995}.
\begin{lem}\label{lempe4}
            	For $\widetilde w\in T_1\M_0$ sufficiently close to $0$,
there exist $C^2$ functions
$\mu(\widetilde w)\in\R$ and
$\eta(\widetilde w)\in\C^{n+1}$ such that
	\be\label{pe2}
	\dashint_{\S^{2n+1}}|1+\widetilde{w}+\mu+2\operatorname{Re}\langle\eta, \zeta\rangle|^{Q^* }\,\d\si =1,
	\ee
	and
	\[
	\int_{\S^{2n+1}}|1+\widetilde{w}+\mu+2\operatorname{Re}\langle\eta, \zeta\rangle|^{Q^* }\zeta\,\d\si =0.
	\]
Moreover, $\mu(0)=0$, $\eta(0)=0$, $D\mu(0)=0$ and $D\eta(0)=0$.
\end{lem}

We use $\widetilde w\in T_1\M_0$ as local coordinates for $w\in\M_0$ near $1$, as in Lemma~\ref{lempe4}; $\widetilde w=0$ corresponds to $w=1$. Here $\mu(\widetilde w)$ and $\eta(\widetilde w)$ are the scalar and vector corrections in Lemma~\ref{lempe4}, not homotopy or translation parameters.

Let
\[
\widetilde{E}(\widetilde{w})=E_1(w)=\dashint_{\S^{2n+1}}wL_{\theta_0}(w)\,\d\si ,
\]
where $\widetilde{w}\in T_1\M_0$ and $w=1+\widetilde{w}+\mu(\widetilde{w})+2\operatorname{Re}\langle\eta(\widetilde{w}), \zeta\rangle$ as in Lemma~\ref{lempe4}.
The eigenvalue of $L_{\theta_0}$ on $\mathcal H_{j,k}$ is
\begin{equation}\label{eq:spherical-eigenvalues}
\lam_{j,k}=\Big(j+\frac n2\Big)\Big(k+\frac n2\Big),\quad j,k\geq0,
\end{equation}
where $\dim\mathcal H_{j,k}=m_{j,k}$; see \cite[formula~(1.11)]{BFM2013}. Note that $\lam _{0,0}=L_{\theta_0}(1)$.
The orthogonality in Lemmas~\ref{lempe2} and~\ref{lempe3}, together with the coordinates in Lemma~\ref{lempe4}, gives
\[
\widetilde{E}(\widetilde{w})=L_{\theta_0}(1)(1+2\mu(\widetilde{w}))+\dashint_{\S^{2n+1}}\widetilde{w}L_{\theta_0}(\widetilde{w})\,\d\si +o(\|\widetilde{w}\|^2).
\]
Expanding \eqref{pe2} at $\widetilde w=0$ gives
\[
\mu(\widetilde{w})=\frac12 D^2\mu(0)(\widetilde{w},\widetilde{w})+o(\|\widetilde{w}\|^2)
=-\frac{1}{2}\cdot\frac{Q+2}{Q-2} \dashint_{\S^{2n+1}}\widetilde{w}^2\,\d\si +o(\|\widetilde{w}\|^2).
\]
Since $\lam_{1,0}=\frac{Q+2}{Q-2}L_{\theta_0}(1)$ by \eqref{eq:spherical-eigenvalues}, substitution yields
\be\label{pe4}
\widetilde{E}(\widetilde{w})=L_{\theta_0}(1)+\dashint_{\S^{2n+1}}(\widetilde{w}L_{\theta_0}(\widetilde{w})-\lam _{1,0}\widetilde{w}^2)\,\d\si 
+o(\|\widetilde{w}\|^2).
\ee
Set $\mathfrak Q(\widetilde{w}):=\dashint_{\S^{2n+1}}(\widetilde{w}L_{\theta_0}(\widetilde{w})-\lam _{1,0}\widetilde{w}^2)\,\d\si $.
Lemma~\ref{lempe3} and \eqref{eq:spherical-eigenvalues} give, for $\widetilde w,\widetilde v\in T_1\M_0$,
\be \label{pe5}
	D^2\mathfrak Q(\widetilde{w})(\widetilde{v},\widetilde{v})=2\dashint_{\S^{2n+1}}(\widetilde{v}L_{\theta_0}(\widetilde{v})-
	\lam _{1,0}\widetilde{v}^2)\,\d\si 	\geq 2\Big(1-\frac{\lam _{1,0}}{\lam _{2,0}}\Big)\|\widetilde{v}\|^2.
\ee
Thus $\mathfrak Q$ is positive definite on $T_1\M_0$. The coordinates in Lemma~\ref{lempe4} also give
for some $\var_1>0$,
\be\label{pe6}
\|E_R|_{\M_0}-E_1|_{\M_0}\|_{C^2(B_{\var_1}(1))}= O(\var),
\ee
provided $\|R-1\|_{L^\infty(\S^{2n+1})}\leq \var$. Here $B_{\var_1}(1)$ denotes the ball in $\M_0$ of
radius $\var_1$ centered at $1$, with the $C^2$ norm computed in the coordinates of Lemma~\ref{lempe4}.

For $\widetilde w\in T_1\M_0$ and any constant $c$, differentiation gives
\[
\langle DE_R|_{\M_0}(1),\widetilde{w}\rangle=-2L_{\theta_0}(1)\Big(\dashint_{\S^{2n+1}}R\,\d\si \Big)^{(2-2Q)/Q}\dashint_{\S^{2n+1}}(R-c)\widetilde{w}\,\d\si .
\]
H\"older's inequality gives
\begin{align*}
| \langle DE_R|_{\M_0}(1),\widetilde{w}\rangle|\leq&\,  C\|R-c\|_{L^{2Q/(Q+2)}(\S^{2n+1})}\|\widetilde{w}\|_{L^{Q^* }(\S^{2n+1})}\\
\leq&\,  C\|R-c\|_{L^{2Q/(Q+2)}(\S^{2n+1})}\|\widetilde{w}\|.
\end{align*}
Therefore,
\be\label{pe7}
\| DE_R|_{\M_0}(1)\|\leq C\|R-c\|_{L^{2Q/(Q+2)}(\S^{2n+1})}.
\ee

The positivity of the second variation in \eqref{pe5},
together with the perturbation estimates \eqref{pe6}
and \eqref{pe7}, yields a unique local minimizer of
$E_R$ on $\M_0$ when $R$ is sufficiently close to $1$.
\begin{lem}\label{lempe5}
	Let $R\in C^1(\S^{2n+1})$. There exist
$\var_2,\var_3>0$ such that, if
	$\|R-1\|_{L^\infty(\S^{2n+1})}\leq \var \leq \var_2$,  the minimization problem
	\[
	\min_{w\in \M_0,\|w-1\|\leq \var_3} E_R(w)
	\]
	admits a unique minimizer $w_R$. Furthermore, $D^2E_R|_{\M_0}(w_R)$ is positive definite and
	\be\label{pe8}
	w_R>0\quad \text{ on }\, \S^{2n+1},
	\ee
	\be\label{pe9}
	\|w_R-1\|\leq C\inf_{c\in \R}\|R-c\|_{L^{2Q/(Q+2)}(\S^{2n+1})},
	\ee
    and 
	\be\label{pe10}
	\|w_R-1\|_{L^{\infty}(\S^{2n+1})}+\|L_{\theta_0}(w_R-1)\|_{L^{\infty}(\S^{2n+1})}\leq o_{\var}(1),
	\ee
	where $o_{\var}(1)$ denotes some quantity which tends to $0$ as $\var\to 0$.
	Moreover, there exists $C(\var_2)>0$ such that
	\be\label{gradient estimate for wK}
	\|\operatorname{grad}_{g_{\theta_0}}w_R\|_{L^2(\S^{2n+1})}\leq C(\var_2) \inf_{\widetilde c\in [1/2,2]}\|R-\widetilde c\|_{L^2(\S^{2n+1})}.
	\ee
\end{lem}

\begin{proof}
The expansion \eqref{pe4} and the coercivity estimate
\eqref{pe5} show that $1$ is a strict local minimizer
of $E_1$ on $\M_0$. By the $C^2$ perturbation
estimate \eqref{pe6}, this strict convexity persists
in a sufficiently small neighborhood of $1$ when
$\var$ is small. Consequently, $E_R$ admits a unique
constrained minimizer $w_R$ in that neighborhood, and
$D^2E_R|_{\M_0}(w_R)$ is positive definite.
The estimates \eqref{pe5}--\eqref{pe7} further yield
\eqref{pe9}.

	Since $w_R$ is a constrained local minimum, $w_R$ satisfies the Euler--Lagrange equation for some
	Lagrange multiplier $\Lam_R\in \C^{n+1}$:
	\be\label{pe11}
	L_{\theta_0}(w_R)=(\lam _R R-2\operatorname{Re}\langle\Lam_R, \zeta\rangle)|w_R|^{4/(Q-2)}w_R\quad \text{ on }\, \S^{2n+1},
	\ee
	where
	\[
	\lam_R=\frac{\dashint_{\S^{2n+1}}w_RL_{\theta_0}(w_R)\,\d V_{\theta_0}}{
		\dashint_{\S^{2n+1}}R|w_R|^{Q^*}\,\d V_{\theta_0}
	}.
	\]
	The definition of $\lam_R$ and \eqref{pe9} give $|\lam_R-\frac{n^2}{4}|=O(\var)$. Since $L_{\theta_0}$ is a self-adjoint operator and $L_{\theta_0}(2\operatorname{Re}\langle\Lam_R, \zeta\rangle)=2\lam _{1,0}\operatorname{Re}\langle\Lam_R, \zeta\rangle$,
multiplying \eqref{pe11} by $2\operatorname{Re}\langle\Lam_R,\zeta\rangle$ and integrating gives
\begin{align}
\lam _{1,0}\int_{\S^{2n+1}}w_R2\operatorname{Re}\langle\Lam_R,\zeta\rangle\,\d\si 
&=\lam _R\int_{\S^{2n+1}}R|w_R|^{\frac{4}{Q-2}}w_R2\operatorname{Re}\langle\Lam_R,\zeta\rangle\,\d\si\nonumber \\
&\quad-\int_{\S^{2n+1}}(2\operatorname{Re}\langle\Lam_R,\zeta\rangle)^2|w_R|^{\frac{4}{Q-2}}w_R\,\d\si .\label{estimate for Lda}
\end{align}
	Combining \eqref{estimate for Lda} with $\|R-1\|_{L^\infty(\S^{2n+1})}\leq\var$ and \eqref{pe9}, we obtain
	\be\label{eq:lda bound}
	|\Lam_R|=O(\var).
	\ee
	Set $w_R=w_R^+-w_R^-$, where $w_R^\pm=\max\{\pm w_R,0\}$. Note that $\int_{\S^{2n+1}}|w_R^-|^{Q^* }\,\d\si \leq\int_{\S^{2n+1}}|w_R-1|^{Q^* }\,\d\si = o_\var(1)$. Testing \eqref{pe11} against $-w_R^-$ gives
\begin{align*}
C\int_{\S^{2n+1}}(w_R^-)^{Q^*}\,\d\si 
&\geq\int_{\S^{2n+1}}w_R^-L_{\theta_0}(-w_R)\,\d\si \\
&\geq\int_{\S^{2n+1}}w_R^-L_{\theta_0}(w_R^-)\,\d\si \\
&\geq 2^{-2/Q}S_n^2\Big(\int_{\S^{2n+1}}(w_R^-)^{Q^*}\,\d\si \Big)^{\frac{Q-2}{Q}},
\end{align*}
    where we used \eqref{eq:sharp-sphere-sobolev} and the value of $S_n$ in \eqref{Sobolevineq}.
The preceding Sobolev bound and \eqref{pe9} give $w_R^-=0$ for sufficiently small $\var$.
    
    Estimate \eqref{pe9} and the Folland--Stein embedding give $w_R\to1$ in $L^{Q^*}(\S^{2n+1})$. 
Together with $|\lam_R-\frac{n^2}{4}|=O(\var)$ and \eqref{eq:lda bound}, the strong convergence gives
$$
\Big\|(\lam _R R-2 \operatorname{Re}\langle\Lam _R, \zeta\rangle) w_R^{\frac{4}{Q-2}}-\frac{n^2}{4}\Big\|_{L^{Q / 2}(\S^{2n+1})}=o_{\var}(1) .
$$

    Fix $p>\max \{Q^*,(Q^*-1) Q / 2\}$. Taking $\var_2$ sufficiently small, we can absorb the potential term at each iteration. Since $p$ is fixed, finitely many applications of the Folland--Stein inequality then give $\|w_R\|_{L^p(\S^{2n+1})} \leq C$, uniformly in $R$. The right-hand side of \eqref{pe11} is therefore uniformly bounded in $L^{p /(Q^*-1)}(\S^{2n+1})$, with $p /(Q^*-1)>Q / 2$. Subelliptic estimates and compact embedding give precompactness in $C^0(\S^{2 n+1})$. The strong convergence in $L^{Q^*}(\S^{2n+1})$ then gives
    \be\label{3.15}
	\|w_R-1\|_{L^{\infty}(\S^{2n+1})} = o_{\var}(1).
	\ee
    In particular, $w_R \geq 1 / 2$ for sufficiently small $\var$, proving \eqref{pe8}. Substituting \eqref{3.15} into \eqref{pe11} gives the second estimate in \eqref{pe10}.

    %Then \eqref{pe8} follows from the Bony's maximum principle.
%    By the Folland-Stein embedding, \eqref{pe9} yields $w_R \to  1$ in $L^{Q^* }(\S^{2n+1})$. Applying the standard local Moser iteration to \eqref{pe11}, together with \eqref{eq:lda bound}, we obtain a uniform $L^{\infty}$-bound for $w_R$. Standard sub-elliptic regularity then implies that $\{w_R\}$ is precompact in $C^0(\S^{2n+1})$. Since $w_R \to  1$ in $L^{Q^* }(\S^{2n+1})$, every $C^0$-limit equals 1. Hence, we have
%	\be\label{3.15}
%	\|w_R-1\|_{L^{\infty}(\S^{2n+1})} = o_{\var}(1),
%	\ee
%	which, together with \eqref{pe11}, leads to
%	\begin{align*}
%	\|L_{\theta_0}(w_R-1)\|_{L^{\infty}(\S^{2n+1})} = o_{\var}(1).
%	\end{align*}
%	Then \eqref{pe10} follows immediately.

Finally, \eqref{pe9} and \eqref{3.15} imply that, for every $\widetilde c\in[1/2,2]$,
	\begin{align*}
	|\lam_R-L_{\theta_0}(1)/\widetilde c|\leq&\,  C(\var_2)(\|w_R-1\|+\|w_R-1\|_{L^1(\S^{2n+1})}+\|R-\widetilde c\|_{L^1(\S^{2n+1})})\\
	\leq&\,  C(\var_2)\|R-\widetilde c\|_{L^{2Q/(Q+2)}(\S^{2n+1})}.
	\end{align*}
	The balance condition \eqref{eq:balanced-spaces}, positivity \eqref{pe8}, and estimates \eqref{estimate for Lda} and \eqref{3.15} give
\[
	|\Lam_R|^2\leq C(n)|\Lam_R|\Big(\int_{\S^{2n+1}}|R-\widetilde c|\,\d\si +\int_{\S^{2n+1}}|w_R-1|\,\d\si \Big).
\]
	Estimate \eqref{pe9} and H\"older's inequality therefore give
\[
	|\Lam_R|\leq C(n)\|R-\widetilde c\|_{L^{2Q/(Q+2)}(\S^{2n+1})}.
\]
	Thus
	$$\|(\lam_R R-2\operatorname{Re}\langle\Lam_R, \zeta\rangle)w_R^{\frac{Q+2}{Q-2}}-L_{\theta_0}(1)\|_{L^2(\S^{2n+1})}\le C(\var_2)\|R-\widetilde c\|_{L^2(\S^{2n+1})}.
	$$
	Since
\[
	L_{\theta_0}(w_R-1)=(\lam_R R-2\operatorname{Re}\langle\Lam_R, \zeta\rangle)w_R^{\frac{Q+2}{Q-2}}-L_{\theta_0}(1),
\]
the spherical harmonic expansion with eigenvalues \eqref{eq:spherical-eigenvalues} gives
\[
    \|\operatorname{grad}_{g_{\theta_0}}(w_R-1)\|^2_{L^2(\S^{2n+1})} \leq C\|(L_{\theta_0}(w_R-1))\|^2_{L^2(\S^{2n+1})} \le C(\var_2)\|R-\widetilde c\|^2_{L^2(\S^{2n+1})}.
\]
	Hence \eqref{gradient estimate for wK} holds.
\end{proof}

\subsection{The reduced equation and degree comparison}\label{sec:degree-comparison}

We now express the variational problem in the conformal
coordinates given by Lemma~\ref{lem:balanced-decomposition}.
Write $v=\varpi(w,p)=T_{\varphi_p}^{-1}w$, where $w\in\M_0$ and $p=(\lam-1)P/\lam$.
By the conformal invariance in Lemma~\ref{lempe1},
\begin{equation}\label{eq:reduced-energy}
I(w,p):=E_R(v)=E_{R\circ\varphi_p}(w).
\end{equation}
For each $p\in\mathbb B^{n+1}$, consider the
constrained minimization problem
\[
\min_{w\in\M_0, \|w-1\|\leq \var_3}I(w,p)=\min_{w\in\M_0, \|w-1\|\leq \var_3}E_{R\circ\varphi_{p}}(w).
\]
If $\|R-1\|_{L^\infty(\S^{2n+1})}\leq\var\leq\var_2$, Lemma~\ref{lempe5} gives a unique minimizer $w_p$. Set $v_p=T_{\varphi_p}^{-1}w_p$. By the positivity in \eqref{pe8} and the Euler--Lagrange
equation \eqref{pe11}, there exists
$\Lam_p\in\C^{n+1}$ such that
\be\label{eq:lag-mul}
L_{\theta_0}(w_p)=(\lam _p R\circ\varphi_p-2\operatorname{Re}\langle\Lam_p, \zeta\rangle)w_p^{\frac{Q+2}{Q-2}}\quad \text{ on }\, \S^{2n+1},
\ee
where $\lam _p$ is the scalar normalization, distinct from the dilation parameter $\lam $, and
\be\label{lam p}
\lam _p=\frac{\dashint_{\S^{2n+1}}w_p L_{\theta_0}(w_p)\,\d\si }{
	\dashint_{\S^{2n+1}}R\circ\varphi_p w_p^{Q^*}\,\d\si 
}.
\ee 

We next establish the regularity needed to differentiate
along the CR conformal directions. For each fixed $p\in\mathbb B^{n+1}$,
equation \eqref{eq:lag-mul}, Lemma~\ref{lempe5},
and the subelliptic estimates in \cite{FS1974}
give $w_p\in\Gamma_{2+\gamma}(\S^{2n+1})$
for some $\gamma\in(0,1)$.
Since $T_0$ commutes with $L_{\theta_0}$,
differentiating \eqref{eq:lag-mul} in the
distributional sense yields
\begin{align*}
L_{\theta_0}(T_0 w_p)= & \, T_0(\lam _p R \circ \varphi_p-2 \operatorname{Re}\langle\Lam _p, \zeta\rangle) w_p^{Q^*-1} \\
& \, +(Q^*-1)(\lam _p R \circ \varphi_p-2 \operatorname{Re}\langle\Lam _p, \zeta\rangle) w_p^{Q^*-2} T_0 w_p \in L^2(\S^{2 n+1}) .
\end{align*}
Consequently,
$T_0w_p\in S_1^2(\S^{2n+1})$.
The first horizontal derivatives of $w_p$ also
belong to $S_1^2(\S^{2n+1})$.
Thus every infinitesimal CR variation of $w_p$
lies in the energy space.
Integrating these derivatives along CR flows
shows that, for each fixed $p_0\in\mathbb B^{n+1}$,
the map
$p\mapsto T_{\varphi_p}^{-1}w_{p_0}$
is $C^1$ from a neighborhood of $p_0$ into
$S_1^2(\S^{2n+1})$.

To identify the finite-dimensional equation associated with
the constrained minimizer, we apply the Kazdan--Warner
identity \eqref{KW new} to \eqref{eq:lag-mul}.
For each $1\leq i\leq n+1$, this gives
\begin{align}
&\, \sum_{j=1}^{n+1}\Lam_p^j\int_{\S^{2n+1}}\langle \widehat{\nabla} \bar{\zeta}_j,\widehat{\nabla} \zeta_i\rangle_{\widehat{g}} w_p^{Q^*}\,\d\si +\sum_{j=1}^{n+1}\bar{\Lam}_p^j\int_{\S^{2n+1}}\langle \widehat{\nabla} \zeta_j,\widehat{\nabla} \zeta_i\rangle_{\widehat{g}} w_p^{Q^*}\,\d\si \notag\\
=&\,\lam _p\int_{\S^{2n+1}}\langle \widehat{\nabla} (R\circ\varphi_p),\widehat{\nabla} \zeta_i\rangle_{\widehat{g}} w_p^{Q^*}\,\d\si \quad \text{ for }\, 1\leq i\leq n+1.\label{eq:lag-mul-2}
\end{align}

For $\Lam\in\C^{n+1}$ and $1\leq i\leq n+1$, define $$(\mathscr{Q}_p\Lam)_i:=\int_{\S^{2n+1}}\langle \widehat{\nabla} (2\operatorname{Re}\langle \Lam,\zeta\rangle),\widehat{\nabla} \zeta_i\rangle_{\widehat{g}} w_p^{Q^*}\,\d\si .$$
The implicit function theorem gives $C^1$ dependence of $w_p$ on $p$.
Under the real identification $\C^{n+1}\simeq\R^{2n+2}$, the operator $\mathscr{Q}_p:\C^{n+1}\to \C^{n+1}$ is symmetric and positive definite, hence invertible.
The formula \eqref{lam p} gives the $C^1$-dependence of $\lam_p$ on $p$. Testing \eqref{eq:lag-mul} against the real and imaginary parts of the coordinate functions gives a linear system for $\Lam _p$ whose matrix is positive definite and whose entries and right-hand side depend $C^1$ on $p$. Thus $\Lam _p$ also depends $C^1$ on $p$. For the reduced energy \eqref{eq:reduced-energy}, the derivative $\pa_p I(w_p,p)$ is taken with $w$ fixed at $w_p$; since $D_wI(w_p,p)=0$ on the constraint, it also equals $D_p[I(w_p,p)]$.

To specify the domains for the degree comparison, set
\begin{align*}
\mathcal{N}_1=&\,\{w\in \M_0: \|w-1\|< \var_3\},\\
\mathcal{N}_2(\widetilde \lam)=&\{v\in \M: v=\varpi (w,p)  \text{ for some  $w\in\mathcal{N}_1$ and $p=sP$,} \\
& \quad \quad P\in \S^{2n+1}, s=(\lam-1)/\lam, 1 \le \lam <\widetilde \lam\},\\
\mathcal{N}_2=&\,\mathcal{N}_2(\infty),\\
\mathcal{N}_3(\widetilde \lam)=&\, \Big\{v\in S_1^2(\S^{2n+1})\setminus\{0\}: \frac{v}{(\dashint_{\S^{2n+1}}|v|^{Q^*}\,\d\si )^{1/Q^*}} \in \mathcal{N}_2(\widetilde \lam) \Big\},\\
\mathcal{N}_3=&\, \mathcal{N}_3(\infty).
\end{align*}
Choose $0<\al '<\al <1$ and set $\mathcal X=\Gamma_{2+\al '}(\S^{2n+1})$. For a cutoff $C>0$, define
\begin{equation}\label{eq:degree-domain}
    \mathcal U(\lam )= \mathcal N_3(\lam )\cap
 \Big\{v\in \mathcal{X}:\|v\|_{\mathcal{X}}<C,\ \min_{\S^{2n+1}}v>C^{-1}\Big\}.
\end{equation}
The dependence of $\mathcal U(\lam)$ on $C$ is suppressed; its closure and boundary are taken in $\mathcal X$.

The preceding construction leads to a comparison
between the Leray--Schauder degree of the CR Nirenberg
equation \eqref{maineq} and the Brouwer degree of the map
$G_\lam$ defined in \eqref{0.10}.
\begin{thm}\label{thm:3.1}
 There exists some constant $\var_4\in (0,\var_2)$ with the following property.
	Let $\iota>0$, and let $\om:[1,\infty)\to(0,\infty)$ be a nonincreasing positive continuous function with $\lim_{\lam\to\infty}\om(\lam)=0$. Suppose that a nonconstant function $R\in C^1(\S^{2n+1})$ satisfies
	\[
	\|R-1\|_{L^{\infty}(\S^{2n+1})}\leq \var_4,
	\]
    and 
	\be\label{condition for not degree}
	\|R\circ\varphi_{P,\lam}-R(P)\|^2_{L^2(\S^{2n+1})}\leq \om(\lam)\Big|\dashint_{\S^{2n+1}}R\circ\varphi_{P,\lam}(\zeta)\zeta\,\d\si \Big|,
	\ee
	for all $P\in\S^{2n+1}$ and $\lam\geq\max\{1,\iota\}$.
    For $0<\al^\prime<\al<1$ and each fixed sufficiently large $\lam$, there exists $C_2>0$ depending on $n,\al,\al^\prime,\iota,\om$ and $\|R\|_{\Gamma_\al(\S^{2n+1})}$ such that, for every $C\geq C_2$, both degrees below are well-defined and 
\begin{align}
		&\,\deg\Big(v-\frac{n^2}{4}(L_{\theta_0})^{-1}(R|v|^{\frac{4}{Q-2}}v), \mathcal{U}(\lam)
        ,0\Big)=-\deg\Big(\dashint_{\S^{2n+1}}R\circ\varphi_{P,\lam}(\zeta)\zeta\,\d\si ,\mathbb{B}^{n+1},0\Big) .\label{eq: degree}
\end{align}
If, in addition,    
    \[
	\deg\Big(\dashint_{\S^{2n+1}}R\circ\varphi_{P,\lam}(\zeta)\zeta\,\d\si ,\mathbb{B}^{n+1},0\Big) \neq 0,
	\]
	then \eqref{maineq} has at least one positive solution.
\end{thm}

\begin{proof}
	For $P\in\S^{2n+1}$ and $\lam\geq1$, set
$s=(\lam-1)/\lam$ and $p=sP$. Define
	\begin{align*}
	A(P,\lam)=&\,\frac{1}{Q}\dashint_{\S^{2n+1}}\langle \widehat{\nabla}(R\circ\varphi_{P,\lam}),\widehat{\nabla} \zeta\rangle w_{p}^{Q^*}\,\d\si ,\\
	G(P,\lam)=&\,\dashint_{\S^{2n+1}}R\circ\varphi_{P,\lam}(\zeta)\zeta\,\d\si =G_\lam(P).
	\end{align*}
Condition~\eqref{condition for not degree} and the nonconstancy of $R$ imply that $G(P,\lam)\neq0$ for all $P\in\S^{2n+1}$ and $\lam>\max\{1,\iota\}$. Integration by parts, together with the balancing
condition in \eqref{eq:balanced-spaces}, gives
\[
	A(P,\lam)=G(P,\lam)+\mathcal E_1+\mathcal E_2,
\]
	where
	\begin{align*}
	\mathcal E_1=&\, \dashint_{\S^{2n+1}}(R\circ\varphi_{P,\lam}-R(P))\zeta(w_p^{Q^*}-1)\,\d\si ,\\
	\mathcal E_2=&\, -\frac{1}{Q}\dashint_{\S^{2n+1}}(R\circ\varphi_{P,\lam}-R(P))\langle \widehat{\nabla} \zeta, \widehat{\nabla}(w_p^{Q^*})\rangle\,\d\si .
	\end{align*}
	By the Cauchy--Schwarz inequality and
\eqref{pe9}--\eqref{pe10}, we have
\begin{align*}
|\mathcal E_1|
&\leq
C\|R\circ\varphi_{P,\lam}-R(P)\|_{L^2(\S^{2n+1})}
\|w_p^{Q^*}-1\|_{L^2(\S^{2n+1})}\\
&\leq
C\|R\circ\varphi_{P,\lam}-R(P)\|_{L^2(\S^{2n+1})}
\|R\circ\varphi_{P,\lam}-R(P)\|_{L^{2Q/(Q+2)}(\S^{2n+1})}\\
&\leq
C\om(\lam)|G(P,\lam)|.
\end{align*}
Similarly, \eqref{pe10} and
\eqref{gradient estimate for wK} yield
\begin{align*}
|\mathcal E_2|
&\leq
C\|R\circ\varphi_{P,\lam}-R(P)\|_{L^2(\S^{2n+1})}
\|\widehat\nabla(w_p^{Q^*})\|_{L^2(\S^{2n+1})}\\
&\leq
C\|R\circ\varphi_{P,\lam}-R(P)\|_{L^2(\S^{2n+1})}^2\\
&\leq
C\om(\lam)|G(P,\lam)|.
\end{align*}

     Consequently, for all sufficiently large $\lam$,
\[
	A(P,\lam)\cdot G(P,\lam)\geq (1-C\om(\lam))|G(P,\lam)|^2>0
	\]
	uniformly in $P\in\S^{2n+1}$.
Thus $A(P,\lam)$ and $G(P,\lam)$ are homotopic
through nonvanishing maps, and
	\[
	\deg(A(P,\lam),\mathbb{B}^{n+1},0)=\deg(G(P,\lam),\mathbb{B}^{n+1},0).
\]
	By \eqref{eq:lag-mul-2}, the positive definiteness
of $\mathscr Q_p$, and $\lam_p>0$, we further obtain
	\be\label{eq:degree equal}
	\deg(\Lam_p, \mathbb{B}^{n+1}_s,0)=\deg(A(P,\lam),\mathbb{B}^{n+1},0)=\deg(G(P,\lam),\mathbb{B}^{n+1},0),
	\ee
	where $s=(\lam-1)/\lam$ is sufficiently close to $1$.
	%It follows from \eqref{eq:degree equal} and our hypothesis that for $s$ sufficiently close to 1,
	%\begin{align*}
	%\deg(\Lam_p, \mathbb{B}^{n+1}_s,0)\neq 0.
	%\end{align*}
	%Therefore $\Lam_p$ has to have a zero inside $\mathbb{B}^{n+1}$ which immediately implies that \eqref{maineq} has at least one positive solution after scaling.

We next relate the multiplier $\Lam_p$ to the
derivative of the reduced energy \eqref{eq:reduced-energy}.
Fix $p_0=(\lam_0-1)P_0/\lam_0$, where
$P_0\in\S^{2n+1}$ and $\lam_0\geq1$.
We compute the partial derivative
$\pa_p I(w_{p_0},p)$ at $p=p_0$, keeping
$w_{p_0}$ fixed.

Set $v_{p_0}=T_{\varphi_{p_0}}^{-1}w_{p_0}$.
By conformal covariance,
equation \eqref{eq:lag-mul} becomes
\[
	L_{\theta_0}(v_{p_0})=(\lam _{p_0}R-2\operatorname{Re}\langle\Lam_{p_0}, \varphi^{-1}_{p_0}\rangle)v_{p_0}^{\frac{Q+2}{Q-2}}\quad \text{ on }\, \S^{2n+1}.
\]
It follows from the definition of $E_R$ that,
for any  $\psi\in S_1^2(\S^{2n+1})$,
\[
DE_R(v_{p_0})[\psi]
=-4\Big(\dashint_{\S^{2n+1}}Rv_{p_0}^{Q^*}\,\d\si\Big)^{-2/Q^*}
\dashint_{\S^{2n+1}}
\operatorname{Re}\langle\Lam_{p_0},\varphi_{p_0}^{-1}\rangle
v_{p_0}^{Q^*-1}\psi\,\d\si.
\]
The differentiability established after \eqref{lam p} allows us to take
$\psi=D_p(T_{\varphi_p}^{-1}w_{p_0})|_{p=p_0}$ componentwise.
Differentiating the critical power and changing variables then yield
\begin{equation}\label{eq:reduced-first-variation}
D_p I(w_{p_0},p)\big|_{p=p_0}
=-\frac{4}{Q^*}
\Big(\dashint_{\S^{2n+1}}Rv_{p_0}^{Q^*}\,\d\si\Big)^{-2/Q^*} D_p\Big[
\dashint_{\S^{2n+1}}
\operatorname{Re}\langle\Lam_{p_0},
\varphi_{p_0}^{-1}\circ\varphi_p\rangle
w_{p_0}^{Q^*}\,\d\si\Big]\Big|_{p=p_0}.
\end{equation}
	
To determine the degree of the reduced gradient,
we use the parameter relation in \eqref{varphi new}.
The corresponding ball coordinate $b$ satisfies
 \[
	b=\frac{2-|p|}{2-2|p|+|p|^2} p, \quad D b(0)=I, \quad \operatorname{det} D b(p)>0 \quad(|p|<1) .    
\]
   Thus $p\mapsto b$ is an orientation-preserving
$C^1$ diffeomorphism.
Applying Lemma~\ref{positive} at every
$p_0\in\mathbb B^{n+1}$, including $p_0=0$,
we find that the matrix
\[
	\pa_p\Big(\dashint_{\S^{2n+1}} (\varphi^{-1}_{p_0}\circ\varphi_p)(\zeta)  w_{p_0}^{Q^*}\,\d\si \Big)\Big|_{p=p_0}
\]
	is invertible and has positive determinant.  By \eqref{eq:reduced-first-variation}, the
derivative of the reduced energy is therefore
related to $\Lam_p$ by an invertible linear
transformation with positive determinant,
apart from a negative scalar factor.
Since the real dimension is $2n+2$, this
negative factor does not change the degree.
Consequently, for sufficiently large $\lam$,
\[
	\deg(\Lam_{p}, \mathbb B_s^{n+1}, 0)=\deg(\pa_p I(w_p,p), \mathbb B_s^{n+1}, 0).
	\]
  
    Finally, applying Lemma~\ref{prop A2} and
the degree identity \eqref{eq:degree equal},
we obtain, for all sufficiently large $C$,
    \begin{align*} &\, \operatorname{deg}(v-\frac{n^2}{4} L_{\theta_0}^{-1}(R|v|^{4 /(Q-2)} v), \mathcal{U}(\lam ), 0) =-\operatorname{deg}(\Lam _p, \mathbb{B}_s^{n+1}, 0) \\ 
  =&\,-\operatorname{deg}\Big(\dashint_{\S^{2 n+1}} R \circ \varphi_{P, \lam }(\zeta) \zeta\,\d\si , \mathbb{B}^{n+1}, 0\Big),
    \end{align*}
    which proves \eqref{eq: degree}.

   If the Brouwer degree on the right-hand side is
nonzero, the Leray--Schauder degree is also nonzero.
Hence the corresponding operator has a fixed point
in $\mathcal U(\lam)$, which gives a positive
solution of \eqref{maineq}.
\end{proof}

\subsection{Homogeneous expansions and degree computation}\label{sec:homogeneous-degree}
We next establish condition \eqref{condition for not degree}
using the homogeneous expansions of $R$ at its full critical
points. We follow the notation for the critical set and
gradients introduced in Section~\ref{sec:2.1}.

For the integral estimates, we use the unscaled Cayley charts $\mathcal F_P=U_P\circ F=F_P\circ\delta_2$.
Since $\delta_2$ commutes with $\delta_{\lam^{-1}}$,
these charts give the same conformal maps
$\varphi_{P,\lam}$ as in \eqref{varphi new}.
Under this change of coordinates, a homogeneous term of
degree $\beta$ is multiplied by $2^\beta$, while the
remainder is composed with $\delta_2$.
We retain the same notation for the resulting terms.
The scaling lemma in Part~I \cite{QTWPartI} gives the
corresponding change in the flatness data.
In particular, the zero sets of the integral conditions
and the signs entering the index formula are unchanged,
since the homogeneous terms are multiplied by the same
positive factor $2^\beta$.

We first establish the required estimate away from the
critical set at the scale $\lam^{-1}$. The corresponding
estimate near the critical points will be given in
Lemma~\ref{lem:3.8}. In both lemmas, $b_0$ denotes a
positive lower bound for $R$, while $c_{\mathrm{nd}}$
and $C_{\mathrm{nd}}$ denote uniform lower and upper
bounds for the derivatives of the homogeneous terms.
\begin{lem}\label{lem condition for not degree}
	Suppose $R\in\mathcal A_\al (\S^{2n+1})$ satisfies for some constant $b_0>0$, $R(P)\geq b_0$ for all $P\in\S^{2n+1}$, and there exists some constant $\var_1\in (0,1)$, such that for each critical point $P_0\in\S^{2n+1}$ of $R$, there exists some $\beta=\beta(P_0)\in(2,Q)$ such that in some Cayley coordinates $\xi=(x,y,t)$ centered at $P_0$,
\[
R(\xi)=R(0)+Q^{(\beta)}(\xi)+W(\xi)\quad \text{ for }\, |\xi|<\var_1,
\]
	where $Q^{(\beta)}(\delta_{\lam}\xi)=\lam ^{\beta}Q^{(\beta)}(\xi)$ for any $\lam >0$, $\xi\in\Hn $ and
\[
	c_{\mathrm{nd}}|\xi|^{\beta-1}\leq |\widetilde{\nabla} Q^{(\beta)}(\xi)|+|\xi||TQ^{(\beta)}(\xi)| \leq C_{\mathrm{nd}}|\xi|^{\beta-1},
\]
	for some positive constants $c_{\mathrm{nd}}, C_{\mathrm{nd}}$. Here $W(\xi)$ satisfies
    $$|W(\xi)||\xi|^{-\beta}+|\widetilde{\nabla} W(\xi)||\xi|^{1-\beta}+|TW(\xi)||\xi|^{2-\beta} \leq \om _W(|\xi|),$$
    for some continuous function $\om _W$ with $\lim_{r\to0^+}\om _W(r)=0$. Suppose also that for some constant $d>0$, $|\operatorname{grad}_{g_{\theta_0}}R(P)|\geq d$ for all $P\in\S^{2n+1}$ with $d_{\theta_0}(P,\operatorname{Crit}(R))\ge \var_1/20$. Then there exists some positive constant $C_3$ depending on $n,\al $, $\|R\|_{\mathcal A_\al(\S^{2n+1})}$, $b_0$, $c_{\mathrm{nd}}$, $C_{\mathrm{nd}}$, $d$, $\var_1$, $\min_{P_0\in\operatorname{Crit}(R)}\min\{\beta(P_0)-1,Q-\beta(P_0)\}$, $\om _W$ and the modulus of continuity of $\operatorname{grad}_{g_{\theta_0}}R$, such that for $P\in\S^{2n+1}$, $d_{\theta_0}(P,\operatorname{Crit}(R))\ge C_3/\lam$, we have
	\be\label{eq:lem condition for not degree}
	\|R\circ\varphi_{P,\lam}-R(P)\|^2_{L^2(\S^{2n+1})}= o\Big(\Big|\dashint_{\S^{2n+1}}R\circ\varphi_{P,\lam}(\zeta)\zeta\,\d\si \Big|\Big),
	\ee
	uniformly for such $P$ as $\lam\to\infty$.
\end{lem}
\begin{proof}
	We first consider points $P$ close to
$\operatorname{Crit}(R)$ but at a distance at least
$C_3/\lam$ from it. The estimate away from a fixed
neighborhood of the critical set follows from Taylor
expansion and the symmetry of the integral kernel.

For $s>0$, set $D_{s}:=(\widetilde{\nabla},sT)$.
In the Cayley coordinates centered at $P_0$,
for $A>20/\var_1$, $0<|\xi|<\var_1/2$, and $d(\eta,\xi)<|\xi|/A$, homogeneity gives
\begin{align} D_{|\xi|} R(\eta)  =&\, D_{|\xi|} Q^{(\beta)}(\eta)+o(|\xi|^{\beta-1})\notag \\
 =&\, |\xi|^{\beta-1}( \widetilde{\nabla} Q^{(\beta)}(\delta_{|\xi|^{-1}} \eta),TQ^{(\beta)}(\delta_{|\xi|^{-1}} \eta))+o(|\xi|^{\beta-1})\notag \\
 =&\, |\xi|^{\beta-1}( \widetilde{\nabla} Q^{(\beta)}(\delta_{|\xi|^{-1}} \xi)+o_A(1),TQ^{(\beta)}(\delta_{|\xi|^{-1}} \xi)+o_A(1))+o(|\xi|^{\beta-1})\notag \\
 =&\, D_{|\xi|}Q^{(\beta)}(\xi)+o_A(1)|\xi|^{\beta-1}+o(|\xi|^{\beta-1}),\label{6.16 in JDE}
\end{align}
where we have used $d(\delta_{|\xi|^{-1}} \eta,\delta_{|\xi|^{-1}} \xi)<1 / A$ and $|\delta_{|\xi|^{-1}} \xi|=1$ in the third equality, and $o_A(1)$ denotes some quantity which tends to 0 as $A$ tends to $\infty$.

Fix $P$ with
$r:=d_{\theta_0}(P,P_0)\leq\var_1/10$,
and use the unscaled chart $\mathcal F_P$
introduced above. After identifying $P$ with the
north pole, write   $R_P=R\circ\mathcal F_P$. Then $\mathcal F_P(0)=P$ and
\[
\varphi_{P,\lam}=\mathcal F_P\circ \delta_{\lam^{-1}} \circ \mathcal F_P^{-1}.
\]
We use Heisenberg coordinates
$\xi=(z,t)=(x,y,t)$, where $x,y\in\R^n$.
For the moment integrals, write
$\hat\zeta_1,\ldots,\hat\zeta_{2n+2}$
for the real coordinates of $\zeta$, ordered as
$$
(\operatorname{Re} \zeta_1, \ldots, \operatorname{Re} \zeta_n, \operatorname{Im} \zeta_1, \ldots, \operatorname{Im} \zeta_n, \operatorname{Re} \zeta_{n+1}, \operatorname{Im} \zeta_{n+1}).
$$
By the Cayley formula \eqref{Cayley},
for $i=1,\ldots,n$,
\begin{gather*}
\hat{\zeta}_i=\frac{2x_i(1+|z|^2)-2y_it}{t^2+(1+|z|^2)^2}, \quad \hat{\zeta}_{i+n}=\frac{2y_i(1+|z|^2)+2x_it}{t^2+(1+|z|^2)^2},\\
\hat{\zeta}_{2n+1}=\frac{1-|z|^4-t^2}{t^2+(1+|z|^2)^2},\quad \hat{\zeta}_{2n+2}=\frac{2t}{t^2+(1+|z|^2)^2}.
\end{gather*}

Let $\mathcal F_0$ be the fixed chart at $P_0$, set $\Psi=\mathcal F_0^{-1}\circ \mathcal F_P$, and define $\widehat Q_P^{(\beta)}=Q^{(\beta)}\circ\Psi$. The notation $\widehat Q_P^{(\beta)}$
distinguishes this pullback from the homogeneous
term in the coordinates centered at $P_0$.
For $D_r=(\widetilde\nabla,rT)$,
the corresponding derivative transformation
and its inverse are uniformly bounded for
small $r$. Indeed, the difference between the
left- and right-invariant horizontal fields
contributes $O(r)T$ near $\Psi(0)$, while the
central-to-horizontal block is multiplied by $r$.
The homogeneous derivative bounds consequently yield
\be\label{Dr}
c r^{\beta-1} \leq |D_r \widehat Q_P^{(\beta)}(0)| \leq C r^{\beta-1}.
\ee
Moreover, \eqref{6.16 in JDE} and the smooth
dependence of the derivative transformations give
$$
D_r R_P(\eta)=D_r \widehat Q_P^{(\beta)}(0)+(o_A(1)+o_r(1)) r^{\beta-1},\quad |\eta|\leq r/A,
$$
where $o_r(1)\to0$ as $r\to0$.

Fix a sufficiently small localization radius $0<\var<\var_1/10$. We next obtain a lower bound for $G_\lam(P)$ in \eqref{0.10}.
We first consider the horizontal components
$\hat\zeta_i$, $1\leq i\leq2n$.
For $1\leq i\leq n$, the Cayley formula and
integration by parts give
\begin{align*}
&\,\dashint_{\S^{2n+1}}[R \circ \varphi_{P, \lam }(\zeta)-R(P)] \hat{\zeta}_i\,\d\si  \\
= &\, |\S^{2n+1}|^{-1} \int_{|\xi| \leq \var  \lam }[R_P(\delta_{\lam ^{-1}} \xi)-R_P(0)] \frac{2x_i(1+|z|^2)-2y_it}{t^2+(1+|z|^2)^2} \frac{2^{2n+1} \, \d z \,\d t}{(t^2+(1+|z|^2)^2)^{n+1}}+O(C(\var ) \lam ^{-2n-2}) \\
= & \,\frac{2^{2n}}{n+1}|\S^{2n+1}|^{-1} \int_{|\xi| \leq \var  \lam } \lam ^{-1} \widetilde{\nabla}_i R_P(\delta_{\lam ^{-1}} \xi) \frac{\d z\,\d t}{(t^2+(1+|z|^2)^2)^{n+1}}+O(C(\var ) \lam ^{-2n-2}) \\
= &\, \frac{2^{2n}}{n+1}|\S^{2n+1}|^{-1} \int_{|\xi| \leq  A} \lam ^{-1} \widetilde{\nabla}_i R_P(\delta_{\lam ^{-1}} \xi) \frac{\d z\,\d t}{(t^2+(1+|z|^2)^2)^{n+1}} \\
& \,+\frac{2^{2n}}{n+1}|\S^{2n+1}|^{-1} \int_{A \leq |\xi| \leq  \var  \lam } \lam ^{-1} \widetilde{\nabla}_i R_P(\delta_{\lam ^{-1}} \xi) \frac{\d z\,\d t}{(t^2+(1+|z|^2)^2)^{n+1}}+O(C(\var ) \lam ^{-2n-2}).
\end{align*}
The same identity holds for $n+1\leq i\leq2n$, by integration by parts with the corresponding right-invariant fields.

For the central component, integration by parts
in $t$ gives
\begin{align*}
 & \,\dashint_{\S^{2n+1}}[R \circ \varphi_{P, \lam }(\zeta)-R(P)] \hat{\zeta}_{2n+2}\,\d\si  \\
= &\, \frac{2^{2n+1}}{n+1}|\S^{2n+1}|^{-1} \int_{|\xi| \leq  A} \lam ^{-2} T R_P(\delta_{\lam ^{-1}} \xi) \frac{\d z\,\d t}{(t^2+(1+|z|^2)^2)^{n+1}} \\
&\, +\frac{2^{2n+1}}{n+1}|\S^{2n+1}|^{-1} \int_{A \leq |\xi| \leq  \var  \lam } \lam ^{-2} T R_P(\delta_{\lam ^{-1}} \xi) \frac{\d z\,\d t}{(t^2+(1+|z|^2)^2)^{n+1}}+O(C(\var ) \lam ^{-2n-2}).
\end{align*}

Set $$
G_i=\dashint_{\S^{2n+1}}[R \circ \varphi_{P, \lam }(\zeta)-R(P)] \hat \zeta_i\,\d\si  \quad  \text{ and } \quad  \mathcal{I}_r=\Big(G_1, \cdots, G_{2n}, \frac{\lam r}{2}G_{2n+2}\Big).
$$
Choose $C_3=A^2$, with $A$ sufficiently large.
For $r\geq C_3/\lam$, the preceding integral
identities and \eqref{6.16 in JDE} show 
\begin{align*}
\mathcal{I}_r
=&\,  \frac{2^{2n}}{n+1}|\S^{2n+1}|^{-1} \int_{|\xi| \leq A} \lam ^{-1}\{D_{r} \widehat Q_P^{(\beta)}(0)+(o_A(1)+o(1))r^{\beta-1}\} \frac{\d z\,\d t}{(t^2+(1+|z|^2)^2)^{n+1}} \\
&\,  +O(\lam ^{-\beta}) \int_{|\xi| \geq A}|\xi|^{\beta-1} \frac{\d z\,\d t}{(t^2+(1+|z|^2)^2)^{n+1}} \\
&\,  +O(\lam ^{-1}r^{\beta-1}) \int_{|\xi| \geq A} \frac{\d z\,\d t}{(t^2+(1+|z|^2)^2)^{n+1}} \\
&\,  +o(r^{\beta-1} \lam ^{-1})+O(C(\var ) \lam ^{-2n-1}r) \\
=&\,  \frac{2^{2n}}{n+1}\lam ^{-1} D_r \widehat Q_P^{(\beta)}(0) |\S^{2n+1}|^{-1} \int_{\Hn} \frac{\d z\,\d t}{(t^2+(1+|z|^2)^2)^{n+1}}+o_A(1) \lam ^{-1}r^{\beta-1} \\
&\,  +o_r(1) \lam ^{-1}r^{\beta-1}+O(C(\var ) \lam ^{-2n-1}r).
\end{align*}
Moreover,
$$
\frac{r \lam ^{-2n-1}}{\lam ^{-1} r^{\beta-1}}=\lam ^{2-Q} r^{2-\beta} \leq C \lam ^{\beta-Q} \to  0 .
$$
Thus, by \eqref{Dr}, first fixing $A$ sufficiently large and then restricting $r$ to a sufficiently small neighborhood, we obtain, for all sufficiently large $\lam $,
\begin{equation}\label{eq:scaled-moment-lower}
|\mathcal{I}_r| \geq c \lam ^{-1} r^{\beta-1}.
\end{equation}
Since $|\mathcal I_r|\leq C\lam r|G_\lam(P)|$, estimate \eqref{eq:scaled-moment-lower} gives
\be\label{Gr}
\Big|\dashint_{\S^{2n+1}}[R \circ \varphi_{P, \lam }(\zeta)-R(P)] \zeta\,\d\si \Big| \geq c \lam ^{-2} r^{\beta-2} .
\ee

We now estimate the squared $L^2$ error.
Using the local expansion of $R$ and the bound
$|R_P(\eta)-R_P(0)|\leq C|\eta|(r+|\eta|)^{\beta-1}$,
we obtain
\begin{align}
\dashint_{\S^{2n+1}}|R \circ \varphi_{P, \lam }(\zeta)-R(P)|^2\,\d\si  \leq &\, C \int_{|\xi| \leq \var \lam }|R_P(\delta_{\lam ^{-1}} \xi)-R_P(0)|^2 \frac{\d z\,\d t}{(t^2+(1+|z|^2)^2)^{n+1}}\notag\\
&\,+C(\var) \lam ^{-2n-2} \notag\\
\leq &\, C \lam ^{-2} r^{2 \beta-2} \int_{\Hn }|\xi|^2 \frac{\d z\,\d t}{(t^2+(1+|z|^2)^2)^{n+1}}\notag \\
& \,+C \lam ^{-2 \beta} \int_{|\xi| \leq \var \lam }|\xi|^{2 \beta} \frac{\d z\,\d t}{(t^2+(1+|z|^2)^2)^{n+1}}+C(\var) \lam ^{-2n-2} \notag\\
\leq &\, C \lam ^{-2} r^{2 \beta-2}+C(\var) \lam ^{-\min \{2 \beta, 2n+2\}}(1+\log \lam ) \label{Nr}.
\end{align}

Combining \eqref{Gr} and \eqref{Nr}, and using $r \geq A^2 / \lam $ and $\beta \geq 2$, gives
$$
\mathcal{G}:=\frac{\dashint_{\S^{2n+1}}|R \circ \varphi_{P, \lam }(\zeta)-R(P)|^2\,\d\si }{|\dashint_{\S^{2n+1}}[R \circ \varphi_{P, \lam }(\zeta)-R(P)] \zeta\,\d\si |} \leq C r^\beta+C(\var) \lam ^{-\min \{\beta, Q-\beta\}}(1+\log \lam ) .
$$
Fix $\delta>0$ small. The ratio estimate obtained from \eqref{Gr} and \eqref{Nr} gives $\limsup_{\lam \to\infty}\mathcal G\le C\delta^\beta$
uniformly for $C_3/\lam \le r\le\delta$. Outside the $\delta$-neighborhood of the critical set, the full gradient is bounded away from zero. Taylor expansion and kernel symmetry then give $\mathcal{G} \to  0$ uniformly. Here, we use the horizontal component when $|\nabla_\H R_P(0)| \geq M / \lam $ and the central component otherwise, where $M$ is fixed and sufficiently large. Letting first $\lam  \to  \infty$ and then $\delta \to  0$, we obtain \eqref{eq:lem condition for not degree}. 
\end{proof}

We next consider conformal centers within a distance of order $\lam^{-1}$ from the critical set. The nonvanishing condition on the homogeneous terms provides a lower bound for the averaged map in this region.
\begin{lem}\label{lem:3.8}
	Suppose $R\in\mathcal A_\al (\S^{2n+1})$ satisfies for some constant $b_0>0$, $R(P)\geq b_0$ for all $P\in\S^{2n+1}$, and there exists some constant $\var_1\in (0,1)$, such that for each critical point $P_0\in\S^{2n+1}$ of $R$, there exists some $\beta=\beta(P_0)\in(2,Q)$ such that in some Cayley coordinate system $\xi=(x,y,t)$ centered at $P_0$,
	\[
R(\xi)=R(0)+Q^{(\beta)}(\xi)+W(\xi)\quad \text{ for }|\xi|<\var_1,
	\]
	where $Q^{(\beta)}(\delta_{\lam}\xi)=\lam ^{\beta}Q^{(\beta)}(\xi)$ for any $\lam >0$, $\xi\in\Hn $ and
	\[
	c_{\mathrm{nd}}|\xi|^{\beta-1}\leq |\widetilde{\nabla} Q^{(\beta)}(\xi)|+|\xi||TQ^{(\beta)}(\xi)| \leq C_{\mathrm{nd}}|\xi|^{\beta-1},
	\]
	for some positive constants $c_{\mathrm{nd}},C_{\mathrm{nd}}$. The remainder $W$ satisfies
\[
 W(\xi)=o(|\xi|^\beta),\quad
 |\widetilde\nabla W(\xi)|=o(|\xi|^{\beta-1}),\quad
 |TW(\xi)|=o(|\xi|^{\beta-2}).
\] Suppose also that for some constant $d>0$, $|\operatorname{grad}_{g_{\theta_0}}R(P)|\geq d$ for all $P\in\S^{2n+1}$ with $d_{\theta_0}(P,\operatorname{Crit}(R))\allowbreak\ge \var_1/20$, and
	\[
	\begin{pmatrix}
	\displaystyle\int_{\Hn } \overline{\nabla} Q^{(\beta)}(\eta\circ\xi)\frac{\d z\,\d t}{(t^2+(1+|z|^2)^2)^{n+1}}\\[3mm]
	\displaystyle\int_{\Hn } Q^{(\beta)}(\eta\circ\xi)\frac{1-|z|^4-t^2}{t^2+(1+|z|^2)^2}
    \frac{\d z\,\d t}{(t^2+(1+|z|^2)^2)^{n+1}}
	\end{pmatrix} \neq 0, \quad \forall\,  \eta\in \Hn .
	\]
	Then \eqref{eq:lem condition for not degree} holds. In particular, if $R$ is not identically equal to a constant, we have
	\[
	\dashint_{\S^{2n+1}}R\circ\varphi_{P,\lam}(\zeta)\zeta\,\d\si \neq 0,
	\]
	for large $\lam$.
\end{lem}
\begin{proof}
By Lemma~\ref{lem condition for not degree}, it remains to
establish \eqref{eq:lem condition for not degree} when
$d_{\theta_0}(P,P_0)\leq C_3/\lam$ for some
$P_0\in\operatorname{Crit}(R)$. We estimate the averaged
map from below and the squared $L^2(\S^{2n+1})$ error
from above.

After a unitary change of coordinates, we may take $P_0$
to be the north pole. Let $F$ be the inverse Cayley
transform with $F(0)=P_0$, and set
$R_0=R\circ F$ and
$\widetilde\varphi=F^{-1}\circ\varphi_{P,\lam}\circ F$.
The local expansion becomes
$R_0(\xi)=R(P_0)+Q^{(\beta)}(\xi)+W(\xi)$.

Write $\varphi=\varphi_{P,\lam}$, $\zeta=F(\xi)$, and $\varphi(\zeta)=(\varphi_1(\zeta),\ldots,\varphi_{n+1}(\zeta))$. The components of $\widetilde\varphi(\xi)=F^{-1}(\varphi(\zeta))$ are given by the inverse of \eqref{Cayley}; the arguments are suppressed below:
    \[
    \widetilde{\varphi}_i=\frac{\varphi_i}{1+\varphi_{n+1}}\quad  \text{ for } \,i=1, \ldots, n, \quad \text{ and } \quad \widetilde{\varphi}_{n+1}=\operatorname{Re}\Big(\mathrm{i}\frac{1-\varphi_{n+1}}{1+\varphi_{n+1}}\Big),
    \]
    and
        \begin{align*}
\varphi_i=&\, \frac{1}{\lam }\cdot\frac{(1+\langle \zeta,P\rangle )\lam ^2}{\lam ^2(1+\langle \zeta,P\rangle)+(1-\langle \zeta,P\rangle)}\cdot V_i,\\
V_i=&\, \frac{2\lam (\zeta_i-\langle \zeta,P\rangle P_i)+[\lam ^2(1+\langle \zeta,P\rangle)-(1-\langle \zeta,P\rangle)]P_i}{(1+\langle \zeta,P\rangle)\lam }.
        \end{align*}
    Write $\widetilde\varphi=(\widetilde\varphi_0,\widetilde\varphi_{n+1})$, where $\widetilde\varphi_0=(\widetilde\varphi_1,\ldots,\widetilde\varphi_n)\in\C^n$ and $\widetilde\varphi_{n+1}\in\R$. Here $\xi=(z,t)=(x,y,t)$ and $P=(P_1,\ldots,P_{n+1})$. The quantities $V_i$ in the preceding formula are auxiliary scalar expressions, not the components of the Heisenberg point $V$ introduced below.

Fix $\var>0$ sufficiently small, and then take
$\lam$ large depending on $\var$. When
$|\xi|\leq\var\lam$ and
$d_{\theta_0}(P,P_0)\leq C_3/\lam$,
we claim the following estimates:
  \begin{gather}  
\sum^n_{i=1}|P_i|  \leq  C \lam ^{-1}, \quad|1-P_{n+1}| \leq C \lam ^{-2},\label{6.17 in JDE} \\
|1+\langle \zeta, P \rangle|  \geq  \var ^{-2} \lam ^{-2},\label{6.18 in JDE} \\
\Big|\frac{1}{1+\langle \zeta, P\rangle}-\frac{1}{1+\zeta_{n+1}}\Big|  \leq  o_\var (1)\Big|\frac{1}{1+\zeta_{n+1}}\Big|,\label{6.19 in JDE} \\
\sum_{i=1}^n\Big|\frac{\langle \zeta, P\rangle P_i}{1+\langle \zeta, P\rangle}\Big|  \leq  C \lam ^{-1}(|t+\mathrm{i}(1+|z|^2)|), \label{6.20 in JDE}\\
\frac{1}{2}  \leq  \Big|\frac{(1+\langle \zeta, P\rangle) \lam ^2}{\lam ^2(1+\langle \zeta, P\rangle)+(1-\langle \zeta, P\rangle)}\Big| \leq 2, \label{6.21 in JDE}\\
1  \leq  |1+\varphi_{n+1}(\zeta)| \leq 3, \label{6.22 in JDE}
  \end{gather}
where $C>0$ denotes some constant independent of $\var >0$ and $\lam $, $o_{\var }(1)$ denotes some quantity which tends to 0 as $\var $ tends to $0^{+}$. For $P_0=(0,\ldots,0,1)$, \eqref{eq:contact-distance} gives $|1-P_{n+1}|=\tfrac12d_{\theta_0}(P,P_0)^2$, and $\sum_{j=1}^n|P_j|^2=1-|P_{n+1}|^2\le2|1-P_{n+1}|$. The two coordinate bounds give \eqref{6.17 in JDE}. For sufficiently small $\var$, depending on $C_3$, the contact-distance formula gives \eqref{6.18 in JDE}. To prove \eqref{6.19 in JDE}, we estimate
\begin{align*}
\, \Big|\frac{1}{1+\langle \zeta, P\rangle}-\frac{1}{1+\zeta_{n+1}}\Big|
=&\,\Big|\frac{\langle \zeta, P\rangle-\zeta_{n+1}}{(1+\langle \zeta, P\rangle)( 1+\zeta_{n+1}) }\Big| \\
\leq&\, \sum^n_{i=1}\Big|\frac{ \zeta_i \bar{P}_i}{(1+\langle \zeta, P\rangle)( 1+\zeta_{n+1}) }\Big|+\Big|\frac{\zeta_{n+1}(\bar{P}_{n+1}-1)}{(1+\langle \zeta, P\rangle)( 1+\zeta_{n+1}) }\Big| \\
\leq &\, C \lam ^{-1}|z|\Big|\frac{1}{1+\langle \zeta, P\rangle}\Big|+C \lam ^{-2}\Big|\frac{\zeta_{n+1}}{(1+\langle \zeta, P\rangle)( 1+\zeta_{n+1}) }\Big| \\
 \leq &\, o_{\var }(1)\Big|\frac{1}{1+\langle \zeta, P\rangle}\Big|+o_{\var }(1)\Big|\frac{1}{1+\zeta_{n+1}}\Big|.
\end{align*}

Combining \eqref{6.17 in JDE} and \eqref{6.19 in JDE},
we obtain
\[
\sum_{i=1}^n\Big|\frac{\langle \zeta, P\rangle P_i}{1+\langle \zeta, P\rangle}\Big| \leq C \lam ^{-1}\Big|\frac{1}{1+\zeta_{n+1}}\Big| \leq C \lam ^{-1}|t+\mathrm{i}(1+|z|^2)| .
\]
This proves \eqref{6.20 in JDE}. Estimate
\eqref{6.18 in JDE} implies \eqref{6.21 in JDE},
and substitution into the formula for
$\varphi_{n+1}$ gives \eqref{6.22 in JDE}.

The integral nonvanishing condition and continuity give $\var_6>0$ such that, for every $\eta\in\Hn$ with $|\eta|\leq CC_3$,
\[
\Big|\int_{\Hn } \frac{(\overline{\nabla}Q^{(\beta)})(\eta\circ\xi)}{(t^2+(1+|z|^2)^2)^{n+1}}\,\d z\,\d t\Big| +\Big|\int_{\Hn } Q^{(\beta)}(\eta \circ\xi) \frac{1-|z|^4-t^2}{t^2+(1+|z|^2)^2} \frac{\d z\,\d t}{(t^2+(1+|z|^2)^2)^{n+1}}\Big| \geq  \var _6.
\]

We now estimate $G_\lam(P)$ in \eqref{0.10} from below.
\begin{align*}
&\, \dashint_{\S^{2n+1}}  R \circ \varphi_{P, \lam }(\zeta) \zeta\,\d\si  \\
= &\, \dashint_{\S^{2n+1}} [R \circ \varphi_{P, \lam }(\zeta)-R(P_0)] \zeta\,\d\si  \\
= &\, \frac{2^{2n+1}}{|\S^{2n+1}|} \int_{|\xi| \leq \var  \lam }[R \circ \varphi_{P, \lam }(\zeta)-R(P_0)] \zeta \frac{\d z\,\d t}{(t^2+(1+|z|^2)^2)^{n+1}}+O(C(\var ) \lam ^{-2n-2}) \\
= &\, \frac{2^{2n+1}}{|\S^{2n+1}|} \int_{|\xi| \leq  \var  \lam } Q^{(\beta)}(\widetilde{\varphi}) \zeta \frac{\d z\,\d t}{(t^2+(1+|z|^2)^2)^{n+1}} \\
&\,+o_{\var }(1) \int_{|\xi| \leq  \var  \lam }|\widetilde{\varphi}|^\beta \frac{\d z\,\d t}{(t^2+(1+|z|^2)^2)^{n+1}}+O(C(\var ) \lam ^{-2n-2}).
\end{align*}

Estimates \eqref{6.17 in JDE}--\eqref{6.22 in JDE} give
\begin{align*}
\sum^n_{i=1}|\varphi_i|  \leq &\, C \lam ^{-1}  \sum^n_{i=1}\Big( \Big|\frac{\zeta_i}{1+\langle \zeta, P\rangle}\Big|+\Big|\frac{\langle \zeta, P\rangle P_i}{1+\langle \zeta, P\rangle}\Big|+\lam |P_i|+\Big|\frac{P_i}{(1+\langle \zeta, P\rangle) \lam }\Big|\Big)  \\
 \leq  &\, C \lam ^{-1} \sum^n_{i=1}\Big( \Big|\frac{\zeta_i}{1+\zeta_{n+1}}
\Big|+\lam ^{-1}(1+|\xi|^2)+1+\var ^2\Big)  \\
 \leq  &\, C \lam ^{-1}(1+|z|+\lam ^{-1}|\xi|^2).
\end{align*}
The central component satisfies
\[
|1-\varphi_{n+1}| \leq C \lam ^{-2}(1+|\xi|^2).
\]
The horizontal and central component bounds, together with \eqref{6.22 in JDE}, give
\[
|\widetilde{\varphi}| \leq C\Big(\sum^n_{i=1}|\varphi_i|+|1-\varphi_{n+1}|^{1/2}\Big)\leq C\lam ^{-1}(1+|\xi|+\lam ^{-1}|\xi|^2).
\]
Since $\beta<Q$, integration against the bubble weight yields
\[
\int_{|\xi| \leq \var  \lam } \frac{|\widetilde{\varphi}|^\beta\, \d z\,\d t}{(t^2+(1+|z|^2)^2)^{n+1}}  \leq C \lam ^{-\beta} \int_{|\xi| \leq  \var  \lam } \frac{( 1+|\xi|+\lam ^{-1}|\xi|^2) ^\beta\,\d z\,\d t}{(t^2+(1+|z|^2)^2)^{n+1}} \leq C \lam ^{-\beta}.
\]

To compare $\widetilde\varphi$ with a translated dilation, write
$$
\widetilde{\varphi}=\delta_{\lam ^{-1}}(V\circ E),
$$
where
$$V=\delta_\lam  \widetilde{P}\circ  \xi=(z+ \lam  \widetilde{P}_0,t+\lam ^2 \widetilde{P}_{n+1}-2\lam  \operatorname{Im}(z\bar{\widetilde{P}}_0)),$$ for
$\widetilde{P}:=F^{-1}(P):=(\widetilde{P}_0,\widetilde{P}_{n+1})=(\widetilde{P}_1,\ldots,\widetilde{P}_{n+1})$
and, writing $E=(E_1,E_2)\in\C^n\times\R$,
    \begin{align*}
        E=&\, V^{-1}\circ \delta_\lam  \widetilde{\varphi}\\
        =&\, (\lam \widetilde{\varphi}_0-z-\lam  \widetilde{P}_0,\lam ^2\widetilde{\varphi}_{n+1}-t-\lam ^2 \widetilde{P}_{n+1}+2\lam  \operatorname{Im}(z\bar{\widetilde{P}}_0)-2\lam  \operatorname{Im}((z+\lam  \widetilde{P}_0)\bar{\widetilde{\varphi}}_0)).
    \end{align*}

Lemma~\ref{Prop A1} gives
\be \label{E 1}
    |E_1|=|\lam \widetilde{\varphi}_0-z-\lam  \widetilde{P}_0|\leq o_\var (1)(1+|\xi|)+ C\lam ^{-1}(1+|\xi|^2),
    \ee
and, denoting the uncorrected central coordinate residual by $e_2$,
\be \label{E 2}
     |e_2|=|\lam ^2\widetilde{\varphi}_{n+1}-t-\lam ^2 \widetilde{P}_{n+1}+2\lam \operatorname{Im}(z\bar{\widetilde{P}}_0)|\leq  C\lam ^{-1}(1+|\xi|^3).
\ee
The definition of $E_1$ gives
$$
\lam  \operatorname{Im}((z+\lam  \widetilde{P}_0)\bar{\widetilde{\varphi}}_0)=\operatorname{Im}(z\bar{E}_1)+\lam  \operatorname{Im}(\widetilde{P}_0\bar{E}_1).
$$
Combining \eqref{E 1}, \eqref{E 2}, and \eqref{6.17 in JDE}, we obtain
\[
        |E|\leq C(|E_1|+(|e_2|+|z||E_1|+\lam  |\widetilde{P}_0||E_1|)^{1/2}) \leq (o_\var (1)+C\lam^{-1/2})(1+|\xi|).
\]

The homogeneity of $Q^{(\beta)}$, together with
its first-derivative bounds, now gives
\begin{align*}
Q^{(\beta)}(\widetilde{\varphi})=&\, Q^{(\beta)}(\delta_{\lam ^{-1}}(V\circ E))\\
=&\, \lam ^{-\beta} Q^{(\beta)}(V\circ E)\\
=&\, \lam ^{-\beta} Q^{(\beta)}(V)+\lam ^{-\beta} O(|V|^{\beta-1}+|E|^{\beta-1})|E|\\
=&\,   \lam ^{-\beta}Q^{(\beta)}(z+\lam  \widetilde{P}_0,t+\lam ^2 \widetilde{P}_{n+1}-2\lam \operatorname{Im}(z\bar{\widetilde{P}}_0))\\
&\, +\lam ^{-\beta}O(  (|\xi|+1)^{\beta-1}[(o_\var (1)+C\lam^{-1/2})(1+|\xi|)] \\
& \, +[(o_\var (1)+C\lam^{-1/2})(1+|\xi|)]^\beta) \\
=&\,   \lam ^{-\beta}Q^{(\beta)}(z+\lam  \widetilde{P}_0,t+\lam ^2 \widetilde{P}_{n+1}-2\lam \operatorname{Im}(z\bar{\widetilde{P}}_0))\\
&\, +O\big(\lam^{-\beta}(o_\var(1)+\lam^{-1/2})(1+|\xi|)^\beta\big),
\end{align*}
where we have used the estimate $Q^{(\beta)}(\xi\circ\eta)=Q^{(\beta)}(\xi)+O( |\xi|^{\beta-1}+|\eta|^{\beta-1}) |\eta|$ in the third equality.

Integrating the expansion against the spherical coordinate vector gives
\begin{align*}
&\, |\S^{2n+1}|^{-1} \int_{|\xi| \leq \var  \lam } Q^{(\beta)}(\widetilde{\varphi}) \zeta \frac{\d z\,\d t}{(t^2+(1+|z|^2)^2)^{n+1}} \\
=&\,|\S^{2n+1}|^{-1} \int_{|\xi| \leq \var  \lam }\lam ^{-\beta}Q^{(\beta)}(z+\lam  \widetilde{P}_0,t+\lam ^2 \widetilde{P}_{n+1}-2\lam \operatorname{Im}(z\bar{\widetilde{P}}_0)) \zeta \frac{\d z\,\d t}{(t^2+(1+|z|^2)^2)^{n+1}}\\
&\,+o_{\var }(1) \lam ^{-\beta}+o(\lam ^{-\beta}).
\end{align*}

We compute the leading moments in the real-coordinate order fixed in Lemma~\ref{lem condition for not degree}. Write
$\eta_\lam=\delta_\lam\widetilde P=(a_\lam,b_\lam,s_\lam)$, where
$a_\lam,b_\lam\in\R^n$ and $s_\lam\in\R$.
By \eqref{6.17 in JDE}, the points $\eta_\lam$ remain in a fixed compact subset of $\Hn$.
For each fixed localization radius $\var$, the error in the rescaled moment expansion is
$o_\var(1)+o_\lam(1)$, where the second term tends to zero as $\lam\to\infty$.
Taking first $\lam\to\infty$ and then $\var\to0$ gives a uniform $o(1)$ remainder below.
Integration by parts gives, for $1\leq i\leq n$,
\begin{align*}
&\lam^\beta\dashint_{\S^{2n+1}}
R\circ\varphi_{P,\lam}(\zeta)\hat\zeta_i\,\d\si\\
&\quad=\frac{2^{2n}}{(n+1)|\S^{2n+1}|}
\int_{\Hn}
\big(\overline\nabla_iQ^{(\beta)}+4b_{\lam,i}TQ^{(\beta)}\big)(\eta_\lam\circ\xi)
\frac{\d z\,\d t}{((1+|z|^2)^2+t^2)^{n+1}}+o(1).
\end{align*}
For $n+1\leq i\leq2n$, the same computation gives
\begin{align*}
&\lam^\beta\dashint_{\S^{2n+1}}
R\circ\varphi_{P,\lam}(\zeta)\hat\zeta_i\,\d\si\\
&\quad=\frac{2^{2n}}{(n+1)|\S^{2n+1}|}
\int_{\Hn}
\big(\overline\nabla_iQ^{(\beta)}-4a_{\lam,i-n}TQ^{(\beta)}\big)(\eta_\lam\circ\xi)
\frac{\d z\,\d t}{((1+|z|^2)^2+t^2)^{n+1}}+o(1).
\end{align*}
The central and radial components satisfy, respectively,
\begin{align*}
&\lam^\beta\dashint_{\S^{2n+1}}
R\circ\varphi_{P,\lam}(\zeta)\hat\zeta_{2n+2}\,\d\si\\
&\quad=\frac{2^{2n+1}}{(n+1)|\S^{2n+1}|}
\int_{\Hn}TQ^{(\beta)}(\eta_\lam\circ\xi)
\frac{\d z\,\d t}{((1+|z|^2)^2+t^2)^{n+1}}+o(1),\\
&\lam^\beta\dashint_{\S^{2n+1}}
R\circ\varphi_{P,\lam}(\zeta)\hat\zeta_{2n+1}\,\d\si\\
&\quad=\frac{2^{2n+1}}{|\S^{2n+1}|}
\int_{\Hn}Q^{(\beta)}(\eta_\lam\circ\xi)
\frac{1-|z|^4-t^2}{((1+|z|^2)^2+t^2)^{n+2}}\,\d z\,\d t+o(1).
\end{align*}
The remainders are uniform for $d_{\theta_0}(P,P_0)\leq C_3/\lam$:
\eqref{E 1} and \eqref{E 2} control the local errors, while $\beta<Q$ controls the integrable tails.
The transformation from the derivative moments to the tangential components is triangular and uniformly invertible for bounded $\eta_\lam$.
The lower bound defining $\var_6$ therefore yields, for all sufficiently large $\lam$,
\begin{equation}\label{eq:critical-moment-lower}
\Big|\dashint_{\S^{2n+1}} R \circ \varphi_{P, \lam }(\zeta) \zeta\,\d\si \Big| \geq C^{-1}\var _6 \lam ^{-\beta},
\end{equation}
for some $C>0$ independent of $P$ and $\lam$.

It remains to estimate the squared $L^2$ error.
In the same region
$d_{\theta_0}(P,P_0)\leq C_3/\lam$, we have
\begin{align*}
&\,  \dashint_{\S^{2n+1}}|R \circ \varphi_{P, \lam }(\zeta)-R(P)|^2\,\d\si  \\
 \leq&\, \frac{1}{|\S^{2n+1}|}\int_{|\xi| \leq \var  \lam }|R_0(\widetilde{\varphi}(\xi))-R_0(\widetilde{P})|^2 \frac{2^{2n+1}\,\d z\,\d t}{(t^2+(1+|z|^2)^2)^{n+1}}+C(\var ) \lam ^{-2n-2} \\
 \leq &\,  C \lam^{-2\beta} \int_{|\xi| \leq \var  \lam }(1+|\xi|)^{2 \beta} \frac{\d z\,\d t}{(t^2+(1+|z|^2)^2)^{n+1}}+C(\var ) \lam ^{-2n-2} \\
 =&\,  o(\lam ^{-\beta}).
\end{align*}
Comparing this estimate with
\eqref{eq:critical-moment-lower} proves
\eqref{eq:lem condition for not degree}, as required.
\end{proof}

Combining \eqref{eq:lem condition for not degree} with Theorem~\ref{thm:3.1} gives the following perturbative existence criterion.
\begin{cor}\label{cor1 of existence}
	Suppose $R\in\mathcal A_\al (\S^{2n+1})$ is some positive function satisfying for each critical point $P_0\in\S^{2n+1}$ of $R$, there exists some $\beta=\beta(P_0)\in(2,Q)$ such that in some Cayley coordinate system centered at $P_0$,
	\[
R(\xi)=R(0)+Q^{(\beta)}(\xi)+W(\xi)\quad \text{ for all }\,\xi\text{ close to 0},
	\]
	where $Q^{(\beta)}(\delta_{\lam}\xi)=\lam ^{\beta}Q^{(\beta)}(\xi)$ for any $\lam >0$, $\xi\in\Hn$, and
    	the remainder $W$ satisfies
\[
 W(\xi)=o(|\xi|^\beta),\quad
 |\widetilde\nabla W(\xi)|=o(|\xi|^{\beta-1}),\quad
 |TW(\xi)|=o(|\xi|^{\beta-2}).
\] Suppose also that
	\[
	c|\xi|^{\beta-1} \leq |\widetilde{\nabla} Q^{(\beta)}|+|\xi||T Q^{(\beta)}| \leq C|\xi|^{\beta-1}  \quad \text{ for all  $\xi$ close to 0},
	\]
	\begin{equation}\label{6.26 in jde}
\begin{pmatrix}
\displaystyle \int_{\Hn} \overline{\nabla} Q^{(\beta)}(\eta\circ\xi)
\frac{\d z\,\d t}{(t^2+(1+|z|^2)^2)^{n+1}} \\[3mm]
\displaystyle \int_{\Hn} Q^{(\beta)}(\eta\circ\xi)
\frac{1-|z|^4-t^2}{(t^2+(1+|z|^2)^2)^{n+2}}
\d z\,\d t
\end{pmatrix}
\neq 0,
\quad \forall\ \eta\in \Hn .
\end{equation}
	Then for $\lam$ large enough,
	\[
	\dashint_{\S^{2n+1}}R\circ\varphi_{P,\lam}(\zeta)\zeta\,\d\si \neq 0\quad\text{ for all } \, P\in\S^{2n+1}.
	\]
	If we further assume that
	\[
	\|R-1\|_{L^{\infty}(\S^{2n+1})}\leq \var_4,
	\]
	and
	\[
	\deg\Big(\dashint_{\S^{2n+1}}R\circ\varphi_{P,\lam}(\zeta)\zeta\,\d\si ,\mathbb{B}^{n+1},0\Big) \neq 0,
	\]
	then \eqref{maineq} has at least one positive $C^2$ solution $v$. The infinite-dimensional comparison in Theorem~\ref{thm:3.1} also gives the following identity. For each fixed sufficiently large $\lam $ and all sufficiently large $C$, we have
\[
	\deg\Big(v-\frac{n^2}{4}L_{\theta_0}^{-1}(R|v|^{4/(Q-2)}v),\mathcal{U}(\lam)
    ,0\Big)=-\deg\Big(\dashint_{\S^{2n+1}}R\circ\varphi_{P,\lam}(\zeta)\zeta\,\d\si ,\mathbb{B}^{n+1},0\Big) .
\]
\end{cor}
\begin{proof}
	Lemma~\ref{lem:3.8} gives the nonvanishing of
$G_\lam$ and the estimate
\eqref{eq:lem condition for not degree}.
Under the additional smallness assumption,
Theorem~\ref{thm:3.1} yields the stated degree
identity. If the Brouwer degree is nonzero,
the Leray--Schauder degree is also nonzero,
and hence \eqref{maineq} admits a positive
solution. The $C^2$ regularity follows from
\cite[Lemma~2.4]{QTWPartI}.
\end{proof}

\begin{lem}
     The nonvanishing conditions \eqref{0.8in jde} and \eqref{6.26 in jde} are equivalent.
\end{lem}
\begin{proof}
By \eqref{standard-bubble},
\[
U_{0,1}^{Q^*}(\xi)
=\frac{2^Q}{((1+|z|^2)^2+t^2)^{n+1}}.
\]
Let $\eta=(a,b,s)\in\Hn$. Using the homogeneity
of $Q^{(\beta)}$ and integrating by parts with
respect to the dilation vector field
$x\cdot\pa_x+y\cdot\pa_y+2t\pa_t$, we obtain
\begin{align}
& \, Q \int_{\Hn } Q^{(\beta)}(\eta \circ \xi) \frac{1-|z|^4-t^2}{((1+|z|^2)^2+t^2)^{n+2}} \, \d  \xi\notag \\ =&\,-\beta \int_{\Hn } \frac{Q^{(\beta)}(\eta \circ \xi)}{((1+|z|^2)^2+t^2)^{n+1}} \, \d  \xi\notag\\
&\,+\int_{\Hn } \frac{[\sum_{j=1}^n(a_j \overline{X}_j Q^{(\beta)}+b_j \overline{Y}_j Q^{(\beta)})+2 s T Q^{(\beta)}](\eta \circ \xi)}{((1+|z|^2)^2+t^2)^{n+1}} \, \d  \xi .\label{eq:moment-condition-identity}
\end{align}
The boundary term vanishes since $\beta<Q$. When the vector integral in \eqref{6.26 in jde} vanishes, the derivative term in \eqref{eq:moment-condition-identity} is zero. Since $\beta>0$, identity \eqref{eq:moment-condition-identity} shows that the scalar integrals in \eqref{0.8in jde} and \eqref{6.26 in jde} vanish simultaneously. Thus the two conditions are equivalent. 
\end{proof}

For the homogeneous expansions in
\eqref{eq:alln-model-expansion}, we obtain an explicit
formula for the Brouwer degree in
Corollary~\ref{cor1 of existence}.
\begin{prop}\label{Cor new}
     Assume the hypotheses of Theorem~\ref{thm:2}, except for the condition \eqref{sum index} and the assumption that $R \circ F_{\zeta_0}$ satisfies $(*)_\beta'$. Then for all sufficiently large $\lam $,
\begin{equation}\label{6.27 in jde}
 \dashint_{\S^{2n+1}}R\circ\varphi_{P,\lam }(\zeta)\zeta\,\d\si \ne0
       \quad\text{ for all } \, P\in\S^{2n+1},
\end{equation}
and
\begin{equation}\label{6.28 in jde}
 \deg \Big(\dashint_{\S^{2n+1}}R\circ\varphi_{P,\lam }(\zeta)\zeta\,\d\si ,
                 \mathbb B^{n+1},0\Big)
 =1+\sum_{\substack{P\in\operatorname{Crit}(R)\\
                         \sum_{j=1}^{2n+1}\widetilde a_j(P)<0}}(-1)^{i(P,R)}.
\end{equation}
If also $\|R-1\|_{L^\infty(\S^{2n+1})}\le\var_4$, with the admissible cutoffs and parameter
radius of Theorem~\ref{thm:3.1}, one has
\be\label{6.29 in jde}
 \deg\Big(v-\frac{n^2}{4}L_{\theta_0}^{-1}
                   (R|v|^{4/(Q-2)}v),\mathcal U(\lam ),0\Big)
 =-1-\sum_{\substack{P\in\operatorname{Crit}(R)\\
                         \sum_{j=1}^{2n+1}\widetilde a_j(P)<0}}(-1)^{i(P,R)}.
\ee
Under $\|R-1\|_{L^\infty(\S^{2n+1})}\le\var_4$, if \be\label{5.37 new}
    \sum_{\substack{P\in\operatorname{Crit}(R)\\
                         \sum_{j=1}^{2n+1}\widetilde a_j(P)<0}}(-1)^{i(P,R)}\neq -1,
\ee
equation \eqref{maineq} admits a positive $C^2$ solution.
\end{prop}

\begin{proof}
We first prove \eqref{6.27 in jde}.
   Fix $P_0 \in \operatorname{Crit}(R)$. Set
$$
H(\eta)=\int_{\Hn } \frac{Q_{P_0}^{(\beta)}(\eta \circ \xi)}{((1+|z|^2)^2+t^2)^{n+1}} \, \d  \xi, \quad M(\eta)=\int_{\Hn } \frac{\overline{\nabla} Q_{P_0}^{(\beta)}(\eta \circ \xi)}{((1+|z|^2)^2+t^2)^{n+1}} \, \d  \xi.
$$
Here $\overline\nabla$ is the derivative array associated with \eqref{right-vectorfields}. We write $D_\eta H$ for the ordinary real gradient and $\operatorname{ind}(M,0)$ for the local Brouwer degree of $M$ at an isolated zero. For $\eta=(w,s)\in\C^n\times\R$, the group law and \eqref{right-vectorfields} give
\begin{align}\label{5.38 new}
D H(\eta)=\begin{pmatrix}
I_n & 0 & 2 \operatorname{Im} w \\
0 & I_n & -2 \operatorname{Re} w \\
0 & 0 & 1
\end{pmatrix} M(\eta).
\end{align}
The matrix in \eqref{5.38 new} has determinant $1$, and $I_n$ is the $n\times n$ identity matrix. Since $\beta<Q$ and $\int_{\Hn}|\xi|^\beta(1+|\xi|)^{-2Q}\,\d\xi<\infty$, changing variables transfers differentiation to the weight and gives $H\in C^2$. The same argument applies when $Q_{P_0}^{(\beta)}$ is replaced by $J_{\beta,\vartheta}$.
Unitary invariance of the averaged $J_{\beta,\vartheta}$ term and the change $(z,t)\mapsto(\bar z,-t)$ give
$$
H(w, s)=H(w,-s) .
$$

Direct differentiation gives
$$
\pa _{t t} J_{\beta, \vartheta}=\frac{\beta}{2}(t^2+\vartheta|z|^4)^{\beta / 4-2}(\vartheta|z|^4+\frac{\beta-2}{2} t^2)>0, \quad \xi \neq 0.
$$
Consequently, \eqref{eq:alln-model-expansion} gives
\be\label{5.39 new}
\pa _s H(w, 0)=0, \quad \frac{1}{c} \pa _{s s} H(w, s)>0.
\ee
For the horizontal variables, \eqref{5.38 new} and the homogeneity of $J_{\beta,\vartheta}$ yield
\begin{align}
 \Big|D_w \int_{\Hn } \frac{J_{\beta, \vartheta}((w, 0) \circ \xi)}{((1+|z|^2)^2+t^2)^{n+1}}\,\d\xi\Big|
&\leq\begin{cases}
C|w|, & |w|\leq1,\\
C|w|^{\beta-1}, & |w|\geq1,
\end{cases}\nonumber\\
&  \leq C|w|(1+|w|)^{\beta-2} . \label{5.40 new}
\end{align}
The first case follows from bounded second derivatives on $|w| \leq 1$ and symmetry at $w=0$.
For $r\neq0$, symmetry and the fundamental theorem of calculus give
\begin{align}
 \frac{\beta}{r} \int_{\Hn } \frac{(r+x_1)|r+x_1|^{\beta-2}}{((1+|z|^2)^2+t^2)^{n+1}} \, \d  \xi &=\beta(\beta-1) \int_0^1 \int_{\Hn } \frac{|\tau r+x_1|^{\beta-2}}{((1+|z|^2)^2+t^2)^{n+1}} \, \d  \xi \, \d  \tau \notag\\
&  \geq C^{-1}(1+|r|)^{\beta-2} .\label{5.41 new}
\end{align}
For $r>0$, restrict to $1 / 2 \leq \tau \leq 1$, $1 \leq x_1 \leq 2$; the case $r<0$ follows by symmetry.

Choosing a coordinate with absolute value at least $|w|/\sqrt{2n}$ and combining \eqref{eq:alln-model-expansion}, \eqref{5.40 new}, and \eqref{5.41 new}, we obtain
\be\label{5.42 new}
|D_w H(w, 0)| \geq\Big(C^{-1} \min_{1\leq j\leq n}\{|a_j|,|b_j|\}-C|c|\Big)|w|(1+|w|)^{\beta-2}.
\ee
For sufficiently small $\var(n,\beta,\vartheta)$, the factor in parentheses in \eqref{5.42 new} is positive. If $M(w,s)=0$, equation \eqref{5.38 new} gives 
$$
\pa_s H(w,s)=0,\quad D_w H(w,s)=0.
$$
Equations \eqref{5.39 new} and \eqref{5.42 new} then give $s=0$ and $w=0$. Since $M(0)=0$ by symmetry, we obtain
\be\label{4.42}
M^{-1}(0)=\{0\}.
\ee

Symmetry makes $D^2H(0)$ diagonal, with
\begin{align}
D^2 H(0)= &\, \beta(\beta-1)\Big(\int_{\Hn } \frac{|x_1|^{\beta-2}}{((1+|z|^2)^2+t^2)^{n+1}} \, \d  \xi\Big) \operatorname{diag}(a_1, \ldots, a_n, b_1, \ldots, b_n, 0) \notag\\
&\, +c D_\eta^2\Big(\int_{\Hn } \frac{J_{\beta, \vartheta}(\eta \circ \xi)}{((1+|z|^2)^2+t^2)^{n+1}} \, \d  \xi\Big)\Big|_{\eta=0} .\label{new 1}
\end{align}
All horizontal second derivatives in \eqref{new 1} are finite, while \eqref{5.39 new} gives
$$
\frac{1}{c} \pa _{s s} H(0)>0 .
$$ Decreasing $\var$, the diagonal entries of $D^2 H(0)$ therefore have the signs of
$$
a_1, \ldots, a_n, b_1, \ldots, b_n, c .
$$
Combining \eqref{5.38 new}, \eqref{4.42}, and \eqref{new 1}, we obtain
\be\label{5.43 new}
D M(0)=D^2 H(0), \quad \operatorname{ind}(M, 0)=\operatorname{sgn}(c \prod_{j=1}^n a_j b_j)=(-1)^{i(P_0, R)} .
\ee

The same maximal-coordinate argument applies to $Q_{P_0}^{(\beta)}$ when $|z| \geq|\xi| / 2$. When $|z|<|\xi|/2$, the homogeneous norm and \eqref{eq:alln-model-expansion} give
$$
|t| \geq \frac{\sqrt{15}}{4}|\xi|^2, \quad T Q_{P_0}^{(\beta)}=\frac{\beta c}{2} t(t^2+\vartheta|z|^4)^{\beta / 4-1} .
$$
Hence, for sufficiently small $\var$,
\be\label{5.44 new}
C^{-1}|\xi|^{\beta-1} \leq|\widetilde{\nabla} Q_{P_0}^{(\beta)}(\xi)|+|\xi||T Q_{P_0}^{(\beta)}(\xi)| \leq C|\xi|^{\beta-1}, \quad \xi \neq 0,
\ee 
where $C>1$ is independent of $\xi$.
Estimates \eqref{5.44 new} and \eqref{2.8 new} show that the critical points of $R$ are isolated and hence finite.
By homogeneity,
$$
(x\cdot \pa_x+y\cdot \pa_y+2 t \pa _t) Q_{P_0}^{(\beta)}=\beta Q_{P_0}^{(\beta)} .
$$
Integration by parts, as in \eqref{eq:moment-condition-identity}, together with \eqref{eq:alln-model-expansion} and \eqref{eq:alln-kappa}, gives
\begin{align}
&\, \int_{\Hn } \frac{Q_{P_0}^{(\beta)}(\xi)(1-|z|^4-t^2)}{((1+|z|^2)^2+t^2)^{n+2}} \, \d  \xi \notag\\
=&\,-\frac{\beta}{Q} \int_{\Hn } \frac{Q_{P_0}^{(\beta)}(\xi)}{((1+|z|^2)^2+t^2)^{n+1}} \, \d  \xi \notag\\
=&\,-\frac{\beta}{Q}\Big(\int_{\Hn } \frac{|x_1|^\beta}{((1+|z|^2)^2+t^2)^{n+1}} \, \d  \xi\Big)\Big(\sum_{j=1}^n(a_j+b_j)+\kappa_{\vartheta} c\Big) \neq 0 .\label{5.45 new}
\end{align}
The identities \eqref{4.42} and \eqref{5.45 new} give \eqref{6.26 in jde}. Together with \eqref{5.44 new}, Lemma~\ref{lem:3.8} then yields \eqref{6.27 in jde}.

We next prove \eqref{6.28 in jde}. For $G_\lam $ in \eqref{0.10}, set
$$
f_\lam (P)=\operatorname{Re}\langle G_\lam (P), P\rangle, \quad V_\lam (P)=G_\lam (P)-f_\lam (P) P .
$$
Estimate \eqref{eq:scaled-moment-lower} and the gradient lower bound away from $\operatorname{Crit}(R)$ give, for sufficiently large $\lam$,
\be\label{4.46 new}
V_\lam (P) \neq 0 \quad \text { for all } P \in \S^{2 n+1} \text { satisfying } \min _{P_0 \in \operatorname{Crit}(R)} d_{\theta_0}(P, P_0) \geq \frac{C_3}{\lam } .
\ee
It therefore remains to consider
$$
 d_{\theta_0}(P, P_0) \leq C_3 / \lam  \quad \text { for some }P_0 \in \operatorname{Crit}(R) .
$$
Take $P_0$ as the north pole and use the unscaled chart $\mathcal F_0$ from Lemma~\ref{lem condition for not degree}. Write $\eta=(a,b,s)$ and $M=(M_X,M_Y,M_T)$. After letting $\var\to0$, the componentwise limits in the proof of Lemma~\ref{lem:3.8}, together with \eqref{6.17 in JDE}, give
\begin{align}
& \,\lam ^\beta(V_{\lam , 1}, \ldots, V_{\lam , 2 n}, V_{\lam , 2 n+2})(\mathcal F_0(\delta_{\lam ^{-1}} \eta))\notag \\
=&\,\frac{2^{2 n}}{(n+1)\om_{2n+1}}(M_X(\eta)+4 b M_T(\eta), M_Y(\eta)-4 a M_T(\eta), 2 M_T(\eta))+o(1),\label{4.47 new}
\end{align}
where the components of $V_\lam $ are taken in the same real-coordinate order as in the proof of Lemma~\ref{lem:3.8}, and $o(1)$ is uniform for $\eta$ in bounded sets. The triangular transformation in \eqref{4.47 new} has determinant $2>0$ and is homotopic to the identity through invertible matrices.
The radial expansion and \eqref{5.45 new} give
$$
\lam ^\beta f_\lam (P_0)=-\frac{2^{2 n+1} \beta}{Q\om_{2n+1}} H(0)+o(1), \quad H(0) \neq 0 .
$$
Since the leading radial term is continuous in $\eta$, we may fix $\rho>0$ sufficiently small so that, for all sufficiently large $\lam $,
\be\label{4.48 new}
\operatorname{sgn} f_\lam (P)=-\operatorname{sgn}\Big(\sum_{j=1}^{2n+1} \tilde{a}_j(P_0)\Big), \quad P \in \overline{\B(P_0, \rho / \lam )} .
\ee
Equations \eqref{4.42} and \eqref{4.47 new} exclude zeros of $V_\lam$ in
$$
\frac{\rho}{\lam } \leq d_{\theta_0}(P, P_0) \leq \frac{C_3}{\lam } .
$$
Together with \eqref{4.46 new}, the exclusion places all zeros of $V_\lam$ in the disjoint balls $\B(P_0,\rho/\lam)$. Local degree stability, \eqref{5.43 new}, and \eqref{4.47 new} give
\be\label{4.49 new}
\operatorname{ind}(V_\lam , \B(P_0, \rho / \lam ))=\operatorname{ind}(M, 0)=(-1)^{i(P_0, R)},
\ee 
where $\operatorname{ind}(V_\lam,\B(P_0,\rho/\lam))$ denotes the local Brouwer degree computed in Cayley coordinates.
If necessary, make a sufficiently small tangential perturbation supported in these balls so that the zeros are nondegenerate, retaining the notation. The sign condition \eqref{4.48 new} keeps $G_\lam=V_\lam+f_\lam P$ nonzero during the perturbation, preserving its degree and the indices \eqref{4.49 new}. The boundary-index formula and $\chi(\S^{2n+1})=0$, where $\chi$ is the Euler characteristic, give
\begin{align*}
\operatorname{deg}\Big(\frac{G_\lam }{|G_\lam |}\Big)  =1-\sum_{\substack{P\in\S^{2n+1},\ V_\lam(P)=0 \\
f_\lam (P)<0}} \operatorname{ind}(V_\lam , P) =&\,1+\sum_{\substack{P\in\S^{2n+1},\ V_\lam(P)=0 \\
f_\lam (P)>0}} \operatorname{ind}(V_\lam , P) \\
 =&\,1+\sum_{\substack{P_0 \in \operatorname{Crit}(R) \\
\sum_{j=1}^{2n+1} \tilde{a}_j(P_0)<0}}(-1)^{i(P_0, R)},
\end{align*}
where the last equality follows from \eqref{4.48 new} and \eqref{4.49 new}. Hence \eqref{6.28 in jde} holds.

If $\|R-1\|_{L^\infty(\S^{2n+1})} \leq \var_4$, combining \eqref{eq: degree} with \eqref{6.28 in jde} gives \eqref{6.29 in jde}. Under \eqref{5.37 new}, the degree in \eqref{6.29 in jde} is nonzero, and \eqref{maineq} admits a positive $C^2$ solution. 
\end{proof}

\section{Proofs of Theorems~\ref{thm:0.5}, \ref{thm:2} and \ref{thm:0.6}}\label{sec:main-existence}
We now combine the degree comparison established in
Section~\ref{sec:perturbation} with the compactness results
of Part~I \cite{QTWPartI}. We first prove
Theorem~\ref{thm:0.5} by extending the degree computation
from small perturbations of a constant curvature to
the prescribed Webster curvature. We then establish
Theorem~\ref{thm:2} using the explicit degree formula,
and conclude with the borderline flatness case
in Theorem~\ref{thm:0.6}.

\begin{proof}[Proof of Theorem~\ref{thm:0.5}]
The nonvanishing conclusion \eqref{0.10} follows from Lemma~\ref{lem:3.8}. We prove existence by continuing the degree along the curvature homotopy.

Choose $Q-2<\gamma<\min_{P\in\operatorname{Crit}(R)}\beta(P)$.
The globalization result \cite[Proposition~2.9]{QTWPartI}, applied with
exponent $\gamma$, gives $(*)'_\gamma$ globally for $R_\mu$, with
$L_\mu=\mu$ and auxiliary functions and structural constants independent
of $\mu\in(0,1]$.
    
Set $a(v):=(\dashint_{\S^{2n+1}}v^{Q^*}\,\d\si )^{1/Q^*}$. We claim that there exists $\var_7>0$ such that $\|R_\mu-1\|_{L^\infty(\S^{2n+1})}<\var_4$ for $\mu\in[0,\var_7]$, and every solution $v$ of \eqref{maineq} with $R=R_\mu=1+\mu(R-1)$ has
$\varpi^{-1}(v/a(v))=(w,p)\in\M_0\times\mathbb B^{n+1}$ with $w\in\mathcal N_1$.
    
Suppose, to the contrary, that along a sequence $\mu\to0$ there are solutions $v_\mu$ for which $\|w_\mu-1\|\geq\var_3$, where $\varpi^{-1}(v_\mu/a(v_\mu))=(w_\mu,p_\mu)\in\M_0\times\mathbb B^{n+1}$. By \cite[Theorem~4.4]{QTWPartI}, after passing to a subsequence, either $\{v_\mu\}$ is bounded in $L^\infty(\S^{2n+1})$ or it has precisely one isolated simple blow-up point.

Write $z_\mu:=a(v_\mu)w_\mu=T_{\varphi_\mu}v_\mu$, where $\varphi_\mu=\varphi_{p_\mu}$ as in \eqref{varphi new}. Then
    \be\label{eq:wmu}
    L_{\theta_0}(z_\mu)=\frac{n^2}{4}(R_\mu\circ\varphi_\mu)z_\mu^{\frac{Q+2}{Q-2}}.
    \ee
   The classification for $R\equiv1$ in the bounded case \cite{JL1988} and the normalization argument in \cite[proof of Theorem~1.6]{QTWPartI} give CR automorphisms $\psi_\mu$ such that
$$
\|T_{\psi_\mu} v_\mu-1\|_{L^{\infty}(\S^{2n+1})}=o(1) .
$$
The continuity with respect to the normalized density in $L^1(\S^{2n+1})$, established in Lemma~\ref{lem:balanced-decomposition}, gives $\tilde p_\mu=o(1)$ such that
$$
\frac{T_{\varphi_{\tilde{p}_\mu}}(T_{\psi_\mu} v_\mu)}{a(v_\mu)} \in \M_0 .
$$
The decomposition of CR automorphisms and uniqueness in Lemma~\ref{lem:balanced-decomposition} give, for some $U_\mu\in U(n+1)$,
$$
T_{U_\mu} z_\mu=T_{\varphi_{\tilde{p}_\mu}}(T_{\psi_\mu} v_\mu)=1+o(1) \quad \text { uniformly }.
$$
Hence $z_\mu=1+o(1)$ uniformly.
    Testing \eqref{eq:wmu} against $z_\mu-1$ and using \eqref{norm 4}, we obtain
\begin{align}
\|z_\mu-1\|^2 & =\dashint_{\S^{2 n+1}}(z_\mu-1) L_{\theta_0}(z_\mu-1) \, \d  \si \nonumber \\
& =\frac{n^2}{4} \dashint_{\S^{2 n+1}}(z_\mu-1)[(R_\mu \circ \varphi_\mu) z_\mu^{\frac{Q+2}{Q-2}}-1] \, \d  \si \nonumber \\
& \leq C \dashint_{\S^{2 n+1}}|z_\mu-1| \, \d  \si  \to  0 .\label{eq:homotopy-energy-convergence}
\end{align}
 Thus, $a(v_\mu)^{Q^*}=\dashint_{\S^{2n+1}} z_\mu^{Q^*}\,\d\si \to  1$ as $\mu\to  0$. Together with \eqref{eq:homotopy-energy-convergence}, this gives $w_\mu\to1$ strongly in $S_1^2(\S^{2n+1})$, contradicting $\|w_\mu-1\|\geq\var_3$ and proving the claim.

    The linear dependence in \eqref{eq:curvature-homotopy} and Lemma~\ref{lem:3.8} give
$$
\|R_\mu \circ \varphi_{P, \lam }-R_\mu(P)\|_{L^2(\S^{2n+1})}^2 \leq \mu \om (\lam )\Big|\dashint_{\S^{2 n+1}} R_\mu \circ \varphi_{P, \lam }(\zeta) \zeta\,\d\si \Big| .
$$
Thus $R_{\var_7}$ satisfies \eqref{condition for not degree}.
    
We next consider the remaining parameter interval. The assumptions on $R_{\zeta_0}$ imply that $R_\mu\circ F_{\zeta_0}$ satisfies $(*)_\beta'$ with $L$ replaced by $\mu L$, while $r$ and $C_*$ are unchanged. Since
$$
\var_7 L \leq \mu L \leq L, \quad \mu \in[\var_7, 1],
$$
\cite[Theorem~1.3]{QTWPartI} and subelliptic regularity \cite{FS1974} give uniform positive lower bounds and uniform bounds in the space $\mathcal X$ defined in Section~\ref{sec:degree-comparison}. Choose $M$ sufficiently large and $m>0$ sufficiently small so that every such solution lies in
$$
\mathcal{U}=\{v \in \mathcal{X}:\|v\|_\mathcal{X}<M, \min_{\S^{2n+1}}v>m\} .
$$
    %it follows from Theorem \ref{thm:5.4}, Theorem \ref{thm:5.5}  and the Harnack inequality that there exists some constant $C^*>1$ such that for all $\mu\in [\var_7,1]$,
   %\[
    %    1/C^* \leq v_\mu \le C^*,
%\]
 %   where $v_\mu$ is any solution of \eqref{maineq} with $R=R_\mu$.

 For the fixed parameter $\mu=\var_7$, the positive solution set is compact. The claim and continuity of $\varpi^{-1}$ place every solution in $\mathcal N_3(\lam)$ for a sufficiently large finite $\lam$. Choose the norm and positivity cutoffs sufficiently large and small, respectively. Excision then identifies the degree for $R_{\var_7}$ on $\mathcal U$ with the degree on $\mathcal U(\lam)$ defined in \eqref{eq:degree-domain}.
    
    Homotopy invariance and Corollary~\ref{cor1 of existence} now give
    \begin{align*}
        \deg\Big(v-\frac{n^2}{4}L_{\theta_0}^{-1}(Rv^{\frac{Q+2}{Q-2}}), \mathcal{U},0\Big)
        =&\,\deg\Big(v-\frac{n^2}{4}L_{\theta_0}^{-1}(R_{\var_7}v^{\frac{Q+2}{Q-2}}), \mathcal{U},0\Big)\\
        =&\,-\deg\Big(\dashint_{\S^{2n+1}}R_{\var_7}\circ\varphi_{P,\lam}(\zeta)\zeta\,\d\si ,\mathbb{B}^{n+1},0\Big)\\
        =&\,-\deg\Big(\dashint_{\S^{2n+1}}R\circ\varphi_{P,\lam}(\zeta)\zeta\,\d\si ,\mathbb{B}^{n+1},0\Big)
        \neq0.
    \end{align*}
    The nonzero degree yields a positive solution of \eqref{maineq}.
\end{proof}

\begin{proof}[Proof of Theorem~\ref{thm:2}]
By Proposition~\ref{Cor new}, the Brouwer degree
of $G_\lam$ is given by \eqref{6.28 in jde}.
The index condition \eqref{sum index} therefore implies
\eqref{0.11}. Since the remaining hypotheses of
Theorem~\ref{thm:0.5} are satisfied, the existence
of a positive $C^2$ solution follows.
%For the model \eqref{eq:alln-model-expansion}, the weighted derivative condition and homogeneous nondegeneracy follow from the model verification in Part~I.  Lemma~\ref{lem:alln-averaged-model} computes the translated moment vector.  Under the stated smallness of $c$, its zero set is controlled by the horizontal coefficients, while the value moment contributes the coefficient $\kappa_\vartheta c$.  The local degree at a critical point with negative total coefficient is therefore $(-1)^{i(\zeta,R)}$.  Summing the local contributions and comparing with the degree at the constant coefficient gives
%\[
%\deg(\mathcal G_\lam ,\mathbb B^{n+1},0) =1+\sum_{\substack{\zeta\in\operatorname{Crit}(R)\\
%\sum_j\widetilde a_j(\zeta)<0}}(-1)^{i(\zeta,R)}.
%\]
%The hypothesis of Theorem~\ref{thm:2} makes this degree nonzero.  Theorem~\ref{thm:0.5} then gives a positive solution.  The $C^2$ regularity follows from the intrinsic regularity result of Part~I.
\end{proof}

\begin{proof}[Proof of Theorem~\ref{thm:0.6}]
The nonvanishing assertion \eqref{0.10} again follows
from Lemma~\ref{lem:3.8}. We adapt the compactness
argument in the proof of Theorem~\ref{thm:0.5}
to allow the critical flatness order $Q-2$ at
at most one critical point.

By \cite[Proposition~2.9]{QTWPartI}, the curvature
homotopy \eqref{eq:curvature-homotopy} satisfies
$(*)'_{Q-2}$ with $L_\mu=\mu$.
Fix $\var>0$ and suppose, by contradiction,
that a sequence of positive solutions corresponding
to $\mu_i\in[\var,1]$ blows up.
By \cite[Theorem~4.3]{QTWPartI}, the blow-up
points are isolated, simple, and uniformly separated.

If at least two blow-up points occur, then at least
one of them, say $P$, satisfies $\beta(P)>Q-2$.
The local estimates
\cite[Propositions~3.6 and~3.9]{QTWPartI},
together with the Harnack inequality, give
comparable peak heights. The contribution from
another blow-up point then produces a strictly
positive regular part at $P$.
The corresponding Pohozaev boundary term contradicts
\cite[Corollary~3.14 and Lemma~2.2]{QTWPartI},
as in the proof of \cite[Theorem~4.4]{QTWPartI}.

A single blow-up point is excluded by
\cite[Theorem~4.5]{QTWPartI}.
Consequently, the Harnack inequality and subelliptic
estimates yield compactness of the positive
solution family in $\mathcal X$ for
$\mu\in[\var,1]$.

It remains to consider the limit $\mu\to0$.
In this regime, $L(Q-2,i)=o(1)$, so
\cite[Theorem~4.4]{QTWPartI} applies.
The normalization argument in the proof of
\cite[Theorem~1.6]{QTWPartI} then gives the
same convergence as in
\eqref{eq:homotopy-energy-convergence}.
The excision and homotopy invariance arguments
used in the proof of Theorem~\ref{thm:0.5}
therefore apply without further changes.
Together with \eqref{0.11}, they give a nonzero
Leray--Schauder degree and hence a positive
$C^2$ solution of \eqref{maineq}.
\end{proof}

\appendix
\section{Technical lemmas and Aubin-type inequalities}\label{appendix}
We first justify the orientation and degree comparison used in Theorem~\ref{thm:3.1}, and then prove the coordinate estimates needed in Lemma~\ref{lem:3.8}. Propositions~\ref{prop:cr-aubin} and~\ref{prop:cr-aubin-fixed-mass} give the Aubin-type inequalities announced in the introduction.

\begin{lem}\label{positive}
 Let $\d  \nu$ be a positive probability measure on $\S^{2 n+1}$, with positive continuous density and $\int_{\S^{2 n+1}} \zeta \, \d  \nu=0$. For $b \in \mathbb{B}^{n+1}$, define
$$
\varphi_b(\zeta)=\frac{c_b \zeta+b+(1+c_b)^{-1}\langle\zeta, b\rangle b}{1+\langle\zeta, b\rangle}, \quad c_b=\sqrt{1-|b|^2} .
$$
Then, for every $b_0 \in \mathbb{B}^{n+1}$,
$$
\operatorname{det}\Big[D_b \int_{\S^{2 n+1}}(\varphi_{b_0}^{-1} \circ \varphi_b)(\zeta) \, \d  \nu\Big|_{b=b_0}\Big]>0,
$$
where $D_b$ denotes differentiation with respect to the real coordinates of $b \in \C^{n+1} \cong \mathbb{R}^{2 n+2}$.
\end{lem}
\begin{proof}
     Fix $b=b_0$. Differentiation in the direction $h\in\C^{n+1}$ gives
$$
\frac{\d}{\d \var}(\varphi_b^{-1} \circ \varphi_{b+\var h})(\zeta)\Big|_{\var=0}=B_b h-\langle\zeta, B_b h\rangle \zeta+A_h \zeta,
$$
where $A_h$ is a matrix independent of $\zeta$, $I$ is the identity, $b^*$ is the conjugate transpose of the column vector $b$, and
$$
B_b=c_b^{-1} I+\frac{b b^*}{c_b^2(1+c_b)}>0 .
$$
Since $\int_{\S^{2n+1}} \zeta \, \d  \nu=0$, the term $A_h\zeta$ integrates to zero. Thus the required differential is $M B_b$, where
$$
M v=\int_{\S^{2n+1}}(v-\langle\zeta, v\rangle \zeta) \, \d  \nu.
$$
For $v \neq 0$, positivity of the density implies
$$
\operatorname{Re}\langle M v, v\rangle=\int_{\S^{2n+1}}(|v|^2-\operatorname{Re}(\langle\zeta, v\rangle^2)) \, \d  \nu>0.
$$
Therefore $(1-t) I+t M$ is invertible for $0 \leq t \leq 1$, giving $\operatorname{det} M>0$. Since $B_b$ is positive definite, $\operatorname{det}(M B_b)>0$, as required.
\end{proof}

The next lemma relates the Leray--Schauder degree to the finite-dimensional equation for the Lagrange multipliers.

\begin{lem}\label{prop A2}
   Let $0<\al '<\al <1$, and suppose that
$$
R \in C^1(\S^{2 n+1}), \quad\|R-1\|_{L^\infty(\S^{2n+1})} \leq \var_4,
$$
where $\var_4>0$ is sufficiently small. Use the notation of Section~\ref{sec:degree-comparison}, and fix $\lam >1$, with $s=(\lam -1) / \lam $. If
$$
\Lam _p \neq 0 \quad \text { for }|p|=s,
$$
then, for all sufficiently large cutoffs $C$, the degree on $\mathcal{U}(\lam )$ is well defined and
$$
\operatorname{deg}(F_R, \mathcal{U}(\lam ), 0)=-\operatorname{deg}(\Lam _p, \mathbb{B}_s^{n+1}, 0),
$$
where
$$
F_R(v)=v-c_0 L_{\theta_0}^{-1}(R|v|^{q-2} v), \quad q=Q^*, \quad c_0=\frac{n^2}{4} .
$$
\end{lem}
\begin{proof}
   Subelliptic estimates and compact embedding show that $F_R$ is the identity minus a compact map on $\overline{\mathcal U(\lam)}$, with $\mathcal U(\lam)$ defined in \eqref{eq:degree-domain}. 
   In the coordinates of Lemma~\ref{lempe4}, estimates \eqref{pe4}--\eqref{pe6} give
\begin{align}\label{A1 new}
D_{\widetilde{w}}^2 I(w(\widetilde{w}), p)[h, h] \geq c\|h\|^2, \quad|p| \leq s,
\end{align}
on one fixed coordinate ball containing $\mathcal{N}_1$, after decreasing $\var_3$. Hence,
$$
\begin{aligned}
\langle D_{\widetilde{w}} I(w(\widetilde{w}), p)-D_{\widetilde{w}} I(w(\widetilde{w}_p), p), \widetilde{w}-\widetilde{w}_p\rangle \geq c\|\widetilde{w}-\widetilde{w}_p\|^2.
\end{aligned}
$$
The strict convexity \eqref{A1 new} makes $w_p$ the unique constrained critical point in the coordinate neighborhood.
Therefore, by \eqref{eq:lag-mul} and \eqref{lam p}, the positive zeros in $\overline{\mathcal{N}_3(\lam )}$ are precisely
$$
u_p=a_p T_{\varphi_p}^{-1} w_p, \quad a_p^{4 /(Q-2)}=\frac{4 \lam _p}{n^2}, \quad \quad \Lam _p=0 .
$$

 After decreasing $\var_4$, Lemma~\ref{lempe5} and subelliptic estimates for \eqref{eq:lag-mul} give, for fixed $\lam$,
\begin{equation}\label{eq:reduced-family-bounds}
\sup _{|p| \leq s}\|w_p-1\|<\var_3, \quad \sup _{|p| \leq s}\Big(\|u_p\|_{\mathcal X}+\frac{1}{\min_{\S^{2n+1}}u_p}\Big)<\infty .
\end{equation}
The boundary condition $\Lam_p\neq0$ on $|p|=s$ and \eqref{eq:reduced-family-bounds} give $F_R\neq0$ on $\pa\mathcal U(\lam)$ for all sufficiently large $C$.

Assume first that all zeros in $\mathcal{U}(\lam )$ are nondegenerate. Set
$$
J_R(v)=\frac{1}{2} \dashint_{\S^{2n+1}} v L_{\theta_0} v\,\d\si -\frac{n^2}{4 Q^*} \dashint_{\S^{2n+1}} R|v|^{Q^*}\,\d\si  .
$$
Differentiation in the amplitude variable gives 
\be\label{A2 new}
\frac{\d^2}{\d a^2} J_R(a T_{\varphi_p}^{-1} w_p)\Big|_{a=a_p}=-\frac{4}{Q-2} \dashint_{\S^{2n+1}}
w_p L_{\theta_0} w_p\,\d\si <0,
\ee 
and
\be\label{A3 new}
J_R(u_p)=\frac{1}{Q}\Big(\frac{4}{n^2}\Big)^{(Q-2) / 2} I(w_p, p)^{Q / 2} .
\ee 

We reparametrize $w_p, \Lam _p$, and $I$ by $b$ as in \eqref{varphi new}.
 For the normalized variable $v$, consider the following constraints
$$
\dashint_{\S^{2 n+1}}|v|^{Q^*}\,\d\si =1, \quad \dashint_{\S^{2 n+1}} \varphi_b^{-1}(\zeta)|v|^{Q^*}\,\d\si =0 .
$$
Since $Q^*>2$, these constraints are $C^2$ on $S_1^2(\S^{2 n+1})$, with independent differentials near the positive solution. By Lemma~\ref{positive} and the implicit function theorem, choose local $C^2$ constraint coordinates $v=v(b, z)$, with $z$ in a fixed closed subspace of $S_1^2(\S^{2 n+1})$, and write $u=a v(b, z)$. After eliminating $a>0$, the energy becomes
$$
\Phi(b, z)=\frac{1}{Q}\Big(\frac{4}{n^2}\Big)^{(Q-2) / 2} E_R(v(b, z))^{Q / 2} .
$$

Estimate \eqref{A1 new} and the implicit function theorem give a local graph of normal critical points satisfying
\be\label{A4 new}
\Phi_z(b, z(b))=0, \quad \Phi_{z z}(b, z(b))>0 .
\ee
Uniqueness in \eqref{A1 new} gives $v(b,z(b))=T_{\varphi_b}^{-1}w_b$. Differentiating the first identity in \eqref{A4 new} yields $D_bz(b)=-\Phi_{zz}^{-1}\Phi_{zb}$.
At $u_p$, substitution into \eqref{A3 new} gives
$$
D_b^2[\Phi(b, z(b))] =\Phi_{b b}-\Phi_{b z} \Phi_{z z}^{-1} \Phi_{z b} =C_p D_b^2[I(w_b, b)], \quad C_p>0 .
$$
We write $m(B)$ for the number of negative eigenvalues of a self-adjoint Hessian $B$, counted with multiplicity, and set $m(J_R,u_p):=m(D^2J_R(u_p))$, the Morse index in the energy space $S_1^2(\S^{2n+1})$. The notation $\operatorname{ind}$ in \eqref{5.43 new} denotes the local Brouwer index, not the Morse index. Eliminating the mixed terms and using \eqref{A2 new} and \eqref{A4 new}, we obtain
\be\label{A6 new}
m(J_R, u_p)=1+m(D_b^2[I(w_b, b)]).
\ee

Differentiating the definition of $F_R$ at $u_p>0$, we obtain
\begin{align}\label{A7 new}
D F_R(u_p)=\operatorname{Id}-\frac{n^2}{4}(Q^*-1) L_{\theta_0}^{-1}(R u_p^{Q^*-2} \cdot) .
\end{align}
Its compact part is self-adjoint in the energy inner product and compact on $S_1^2(\S^{2 n+1})$, by the compact embedding $S_1^2(\S^{2n+1}) \hookrightarrow L^2(\S^{2n+1})$ and the boundedness of $L_{\theta_0}^{-1}: L^2(\S^{2n+1}) \to  S_1^2(\S^{2n+1})$. Subelliptic regularity identifies its nonzero generalized eigenspaces in $S_1^2(\S^{2n+1})$ and $\mathcal X$, with the same algebraic multiplicities. Equations \eqref{A6 new} and \eqref{A7 new} therefore give
\begin{align*}
\operatorname{deg}(F_R, B_\delta^{\mathcal X}(u_p), 0) =(-1)^{m(J_R, u_p)} 
 =-\operatorname{sgn} \operatorname{det} D_b^2[I(w_b, b)] 
=-\operatorname{sgn} \operatorname{det} D_b \Lam _b =-\operatorname{sgn} \operatorname{det} D_p \Lam _p .
\end{align*}
The third equality uses \eqref{eq:reduced-first-variation}, Lemma~\ref{positive}, and $(-1)^{\dim_{\R}\C^{n+1}}=(-1)^{2n+2}=1$.
Here $B_\delta^{\mathcal X}(u_p)$ is the norm ball in $\mathcal X$. The last equality uses the orientation-preserving change $p\mapsto b$ in \eqref{varphi new}.
Summing over the solutions yields
$$
\operatorname{deg}(F_R, \mathcal{U}(\lam ), 0)  =-\sum_{\substack{p\in\mathbb B_s^{n+1}\\\Lam_p=0}} \operatorname{sgn} \operatorname{det} D_p \Lam _p  =-\operatorname{deg}(\Lam _p, \mathbb{B}_s^{n+1}, 0) .
$$

Fix $\lam$. Estimate \eqref{A1 new}, the implicit function theorem, and subelliptic estimates make \eqref{eq:reduced-family-bounds} uniform under sufficiently small $C^1$ perturbations of $R$. Choose $C$ larger than the bounds in \eqref{eq:reduced-family-bounds} and keep $\mathcal U(\lam)$ in \eqref{eq:degree-domain} fixed. Compact embedding gives compactness of the corresponding solutions in $\mathcal X$ under these perturbations of $R$. Hence both boundary exclusions persist under sufficiently small perturbations.

The linearization \eqref{A7 new}, self-adjointness in the energy inner product, and subelliptic regularity identify the cokernel with solutions $\psi$ of the linearized equation. If
$$
\dashint_{\S^{2n+1}} h u^{Q^*-1} \psi\,\d\si =0 \quad \text { for every smooth } h,
$$
then $\psi=0$, since $u>0$. A finite cover of the compact zero set therefore provides finitely many smooth perturbation directions for $R$ for which the parameterized equation is transverse to zero.
The parameterized equation is smooth in the unknown positive function and the finite-dimensional perturbation parameters. Finite-dimensional Sard theory gives arbitrarily small regular parameters for which every zero in $\mathcal U(\lam)$ is nondegenerate. Compactness of the solution family excludes additional zeros outside these neighborhoods. Homotopy invariance on the fixed domains proves the identity for the original prescribed curvature $R$.

The proof is complete.
\end{proof}

We now prove the Cayley-coordinate estimates used in Lemma~\ref{lem:3.8}.

\begin{lem}\label{Prop A1}
Let $P_0$ be the north pole, and let $F$ be the inverse Cayley transform defined in \eqref{Cayley}.
Fix $C_3>0$. There exists a constant $\var_0>0$ such that the following holds.

For every $0<\var<\var_0$, let $\lam>0$ be sufficiently large, and suppose that $P\in\S^{2n+1}$, $d_{\theta_0}(P,P_0)\leq C_3\lam^{-1}$. Write $\widetilde{P}=F^{-1}(P)=(\widetilde{P}_0,\widetilde{P}_{n+1})=(\widetilde{P}_1,\cdots, \widetilde{P}_{n+1})$. For $\xi=(z,t)\in \H^n$, $|\xi|\leq \var \lam$, set $\widetilde{\varphi}=F^{-1}\circ \varphi_{P,\lam} \circ F(\xi)$ and $\zeta=F(\xi)$, where $\varphi=\varphi_{P,\lam }$ is defined as in Lemma~\ref{varphi}. The subscript $0$ in $\widetilde P_0$ and $\widetilde\varphi_0$ denotes the first $n$ complex components; their $(n+1)$st components are real.

Then,
\[
\lam  \widetilde{\varphi}_0=z+\lam  \widetilde{P}_0+E_1,\quad \text{ and }\quad 
    \lam ^2\widetilde{\varphi}_{n+1}=t+\lam ^2 \widetilde{P}_{n+1}-2\lam \operatorname{Im}(z\bar{\widetilde{P}}_0)+e_2,
\]
where $E_1\in \C^n$, $e_2\in \R$, 
\be\label{E1}
        |E_1|\leq C[o_\var (1)(1+|\xi|)+\lam ^{-1}(1+|\xi|^2)],
\ee
and 
\be \label{E2}
    |e_2|\leq C\lam ^{-1}(1+|\xi|^3).
\ee
Here, $C$ depends only on $n$ and $C_3$, and $o_\var(1)\to  0$ as $\var\to  0^+$.
\end{lem}
\begin{proof}
We first prove \eqref{E1}. It suffices to consider $E_{11}=\lam\widetilde\varphi_1-z_1-\lam\widetilde P_1$, where $z=(z_1,\ldots,z_n)$; the other components are estimated in the same way.

    For brevity, let $s=1+\langle \zeta,P \rangle $, $q=1-\langle \zeta,P \rangle$ and $A=\frac{\lam ^2 s}{\lam ^2 s+q}$. Set $V_i:=\lam  A^{-1}\varphi_i$ for $1\leq i\leq n+1$, using Lemma~\ref{varphi}; thus $\varphi_i=\lam ^{-1}AV_i$. Here $q$ is a local scalar abbreviation, not the critical exponent $Q^*$.
    Hence
    \be\label{A5}
    E_{11}=\lam  \widetilde{\varphi}_1-z_1-\lam  \widetilde{P}_1=\frac{\lam \varphi_1}{1+\varphi_{n+1}}-z_1-\lam  \widetilde{P}_1= \frac{AV_1}{1+\varphi_{n+1}}-z_1-\lam  \widetilde{P}_1
    = \frac{V_1}{2H}-z_1-\lam  \widetilde{P}_1,
     \ee
     where $H=(1+\varphi_{n+1})/(2A)$ is an auxiliary complex-valued factor in this proof. Its expansion is
     \begin{align*}
         2H=\frac{1}{A}+\frac{V_{n+1}}{\lam }=&\, 1+\frac{q}{\lam ^2 s}+P_{n+1}+\frac{2(\zeta_{n+1}-\langle \zeta,P \rangle P_{n+1})}{s\lam }-\frac{q P_{n+1}}{s\lam ^2}\\
         =&\, 1+P_{n+1}+\frac{q(1-P_{n+1})}{\lam ^2 s}+\frac{2(\zeta_{n+1}-\langle \zeta,P \rangle P_{n+1})}{s\lam }.
          \end{align*}
     Estimates \eqref{6.17 in JDE}--\eqref{6.20 in JDE} give
         \begin{align}
             |H-1|=&\, \Big|\frac{P_{n+1}-1}{2}+\frac{q(1-P_{n+1})}{2\lam ^2 s}+\frac{\zeta_{n+1}-\langle \zeta,P \rangle P_{n+1}}{s\lam }\Big|\notag\\
             \leq&\,  C\Big(\Big|\frac{P_{n+1}-1}{2}\Big|+\Big|\frac{q(1-P_{n+1})}{2\lam ^2 s}\Big|+\frac{|\zeta_{n+1}||1-|P_{n+1}|^2|}{\lam |s|}+\sum^n_{i=1}\frac{|\zeta_i||P_i||P_{n+1}|}{\lam |s|}\Big)\notag\\
             \leq&\,  C\lam ^{-2}(1+|\xi|)+C\lam ^{-3}(1+|\xi|^2).\label{H-1}
         \end{align}
     Estimate \eqref{H-1} bounds $|H|$ away from zero and gives
     \be \label{A6}
         \Big|\frac{1}{H}-1\Big|\leq C|H-1|\leq C\lam^{-2}(1+|\xi|)+C\lam ^{-3}(1+|\xi|^2).
     \ee
    The definition of $V_i$ gives
        \[
            \frac{V_1}{2}=\frac{\zeta_1}{s}+\frac{\lam  P_1}{2}-\frac{\langle \zeta,P \rangle P_1}{s}-\frac{q P_1}{2s \lam }.
        \]
        The inverse Cayley formula and \eqref{6.17 in JDE} give
        $$P_1=\widetilde{P}_1(1+P_{n+1})=\widetilde{P}_1(2-O(\lam ^{-2})),$$
        or equivalently,
        \be \label{A13}
            \frac{\lam  P_1}{2}=\lam  \widetilde{P}_1-O(\lam ^{-1}).
        \ee
    Combining \eqref{6.17 in JDE}--\eqref{6.20 in JDE} with \eqref{A13}, we obtain
    \begin{align}
        \Big|\frac{V_1}{2}-z_1-\lam  \widetilde{P}_1\Big| \leq&\,  \Big|\frac{\zeta_1}{s}-\frac{\zeta_1}{1+\zeta_{n+1}}\Big|+\Big|\frac{\langle \zeta,P \rangle P_1}{s}\Big|+\Big|\frac{qP_1}{2 s\lam }\Big|+O(\lam ^{-1})\nonumber\\
        \leq&\,  o_\var (1)|z_1|+ C\lam ^{-1}(1+|\xi|^2).\label{A7}
        \end{align}
    Combining \eqref{A5}, \eqref{A6}, \eqref{A7} and $|\xi|\leq \var  \lam $, we have
    \begin{align*}
        |E_{11}|=&\, \Big|\Big(\frac{V_1}{2}-z_1-\lam  \widetilde{P}_1\Big)+\frac{V_1}{2}\Big(\frac{1}{H}-1\Big)\Big|\\
        \leq&\,  o_\var (1)|z_1|+ C\lam ^{-1}(1+|\xi|^2)+C(1+|\xi|)(\lam^{-1}(1+|\xi|)+\lam ^{-3}(1+|\xi|^2))\\
        \leq&\,  o_\var(1) (1+|\xi|)+C\lam ^{-1}(1+|\xi|^2).
        \end{align*}

We next prove \eqref{E2}.
    A direct calculation gives
    \begin{align*}
        \frac{1-\varphi_{n+1}}{1+\varphi_{n+1}}=&\, \frac{1-\lam ^{-1}AV_{n+1}}{2AH}
        = \frac{1}{H}\Big(\frac{1}{2A}-\frac{V_{n+1}}{2\lam }\Big)\\
        =&\, \frac{1}{H}\Big(\frac{1-P_{n+1}}{2}+\frac{q(1+P_{n+1})}{2\lam ^2 s}-\frac{\zeta_{n+1}- \langle \zeta,P\rangle P_{n+1}}{s\lam }\Big).
         \end{align*}
    The inverse Cayley formula then gives
    \begin{align}
        \lam ^2 \widetilde{\varphi}_{n+1}=&\, \lam ^2 \operatorname{Re}\Big(\mathrm{i}\frac{1-\varphi_{n+1}}{1+\varphi_{n+1}}\Big)\nonumber\\
        =&\, \lam ^2 \operatorname{Re}\Big(\mathrm{i}\frac{1}{H}\Big(\frac{1-P_{n+1}}{2}+\frac{q(1+P_{n+1})}{2\lam ^2 s}-\frac{\zeta_{n+1}- \langle \zeta,P\rangle P_{n+1}}{s\lam }\Big)\Big).\label{A8}
        \end{align}
    Estimate \eqref{6.17 in JDE} yields
    \[
            \frac{1-P_{n+1}}{1+P_{n+1}}-\frac{1-P_{n+1}}{2}= \frac{(1-P_{n+1})^2}{2(1+P_{n+1})}=O(|1-P_{n+1}|^2)=O(\lam ^{-4}).
    \]
    Substituting $\widetilde P_{n+1}=\operatorname{Re}(\mathrm{i}(1-P_{n+1})/(1+P_{n+1}))$, we obtain
    \be \label{A9}
        \lam ^2 \operatorname{Re}\Big(\mathrm{i}\frac{1-P_{n+1}}{2}\Big)=\lam ^2 \widetilde{P}_{n+1}-O(\lam ^{-2}).
    \ee
    The identity
    $$
    \frac{q}{s}-\frac{1-\zeta_{n+1}}{1+\zeta_{n+1}}=2\Big(\frac{1}{s}-\frac{1}{1+\zeta_{n+1}}\Big),
    $$
    combined with \eqref{6.17 in JDE} and \eqref{6.19 in JDE}, gives
    \begin{align*}
    \Big|\operatorname{Re}\Big(\mathrm{i}\frac{q}{s}\Big)-\operatorname{Re}\Big(\mathrm{i}\frac{1-\zeta_{n+1}}{1+\zeta_{n+1}}\Big)\Big|\leq &\,2|\frac{1}{s}-\frac{1}{1+\zeta_{n+1}}|\\ \leq&\, 2\Big|\frac{1}{s}\Big|\Big|\frac{\langle\zeta, P\rangle-\zeta_{n+1}}{1+\zeta_{n+1}}\Big| \\  \leq&\, C(1+|\xi|^2)[\lam ^{-1}|z|+\lam ^{-2}(1+|\xi|^2)] \\  \leq &\, C \lam ^{-1}(1+|\xi|^3).
     \end{align*}
    Estimate \eqref{6.17 in JDE} then gives
    \be\label{A10}
       \Big| \operatorname{Re}\Big(\mathrm{i}\frac{q(1+P_{n+1})}{2s}\Big)-t\Big|\leq \Big|\operatorname{Re}\Big(\mathrm{i}\frac{q}{s}\Big)-t\Big|+O\Big(\lam ^{-2}\Big|\frac{q}{s}\Big|\Big)\leq C \lam ^{-1}(1+|\xi|^3).
    \ee
    Notice that for each $i=1,\ldots,n$,
        \begin{align*}
            \Big|\lam  \operatorname{Re}\Big(\mathrm{i}\frac{\zeta_i\bar{P}_i}{s}\Big)-\lam  \operatorname{Re}(\mathrm{i}z_i\bar{P}_i)\Big|=&\, \lam \Big|\operatorname{Re}\Big(\mathrm{i}\bar{P}_i\Big(\frac{\zeta_i}{s}-\frac{\zeta_i}{1+\zeta_{n+1}}\Big)\Big)\Big|\\
            \leq&\,  \lam |P_i|\Big|\frac{\zeta_i}{s}-\frac{\zeta_i}{1+\zeta_{n+1}}\Big|\\
            \leq&\,  C\lam^{-1}(1+|\xi|^2),
        \end{align*}
    where we have used \eqref{6.17 in JDE} and \eqref{6.19 in JDE} in the last inequality. Estimate \eqref{6.17 in JDE} then gives
        \begin{align}
            &\,-\lam \operatorname{Re} \Big(\mathrm{i}
                \frac{\zeta_{n+1}-\langle\zeta,P\rangle P_{n+1}}{s}\Big)\nonumber\\
           =&\,-\lam \operatorname{Re}  \Big(\mathrm{i}
                \frac{\zeta_{n+1}(1-P_{n+1})}{s} \Big)
               +\lam \operatorname{Re} \Big(\mathrm{i}
                \frac{(\langle\zeta,P\rangle-\zeta_{n+1})P_{n+1}}{s} \Big)\nonumber\\
          =&\, -\lam \operatorname{Re} \Big(\mathrm{i}\frac{\zeta_{n+1}(1-P_{n+1})}{s} \Big)+\lam \operatorname{Re} \Big(\mathrm{i}\frac{(\sum^n_{i=1}\zeta_i\bar{P}_i+\zeta_{n+1}(\bar{P}_{n+1}-1))P_{n+1}}{s} \Big)\nonumber\\
            =&\,  O(\lam ^{-1}(1+|\xi|^2))+\sum^n_{i=1}\lam  \operatorname{Re}\Big(\mathrm{i}\frac{\zeta_i\bar{P}_i}{s}\Big)\nonumber\\
            =&\,  O(\lam ^{-1}(1+|\xi|^2))+\sum^n_{i=1}\lam \operatorname{Re}(\mathrm{i} z_i\bar{P}_i).\label{A11}
        \end{align}
    Combining \eqref{A9}, \eqref{A10} and \eqref{A11}, we have
        \begin{align}
            M_0:=&\, \lam ^2 \operatorname{Re}\Big(\mathrm{i}\Big(\frac{1-P_{n+1}}{2}+\frac{q(1+P_{n+1})}{2\lam ^2 s}-\frac{\zeta_{n+1}- \langle \zeta,P\rangle P_{n+1}}{s\lam }\Big)\Big)\nonumber\\
            =&\, \lam ^2 \widetilde{P}_{n+1}+t+\sum^n_{i=1}\lam \operatorname{Re}(\mathrm{i}z_i\bar{P}_i)+O_\var ,\label{A12}
        \end{align}
    where $|O_\var |\leq C\lam ^{-1}(1+|\xi|^3)$.
As in the estimate leading to \eqref{H-1}, we have
         \begin{align*}
             &\, \Big|\frac{1-P_{n+1}}{2}+\frac{q(1+P_{n+1})}{2\lam ^2 s}-\frac{\zeta_{n+1}-\langle \zeta,P \rangle P_{n+1}}{s\lam }\Big|\\
          \leq &\,\Big|\frac{1-P_{n+1}}{2}\Big|+\Big|\frac{q(1+P_{n+1})}{2\lam ^2 s}\Big|+\frac{|\zeta_{n+1}||1-|P_{n+1}|^2|}{\lam |s|}+\sum^n_{i=1}\frac{|\zeta_i||P_i||P_{n+1}|}{\lam |s|}\\
            \leq &\, C\lam ^{-2}(1+|\xi|^2).
         \end{align*}
     The bound \eqref{A6} for $H^{-1}-1$ and the identity \eqref{A8} give
    \begin{align*}
        |\lam ^2 \widetilde{\varphi}_{n+1}-M_0|=&\, \lam ^2 \Big|\operatorname{Re}\Big(\mathrm{i}(\frac{1}{H}-1)\Big(\frac{1-P_{n+1}}{2}+\frac{q(1+P_{n+1})}{2\lam ^2 s}-\frac{\zeta_{n+1}- \langle \zeta,P\rangle P_{n+1}}{s\lam }\Big)\Big)\Big|\\
        \leq&\,  \lam ^2 \Big|\frac{1}{H}-1\Big|\Big|\frac{1-P_{n+1}}{2}+\frac{q(1+P_{n+1})}{2\lam ^2 s}-\frac{\zeta_{n+1}- \langle \zeta,P\rangle P_{n+1}}{s\lam }\Big|\\
        \leq&\,  C\lam ^{-1}(1+|\xi|^3).
        \end{align*}
     Finally, \eqref{A12} and the componentwise form of \eqref{A13} give
        \begin{align*}
            \lam ^2 \widetilde{\varphi}_{n+1} =&\, \lam ^2 \widetilde{P}_{n+1}+t+\sum^n_{i=1}\lam \operatorname{Re}(\mathrm{i}z_i\bar{P}_i)+O_\var \\
            =&\, \lam ^2 \widetilde{P}_{n+1}+t-\sum^n_{i=1}2\lam \operatorname{Im}(z_i\bar{\widetilde{P}}_i)+O_\var ,
        \end{align*}
which proves \eqref{E2}.
\end{proof}

We conclude with an Aubin-type inequality. Its critical case is contained in \cite[Corollary~1.2]{Yan2023}. The localization proof below also gives a constant uniform in the subcritical exponent.
\begin{prop}\label{prop:cr-aubin}
For any $\var>0$, there exists $C_\var\ge0$, depending only on $n$ and $\var$, such that, for every $p\in (2,Q^*]$ and every  $v\in S_1^2(\S^{2n+1})$ satisfying
\begin{equation}\label{eq:cr-aubin-balance}
\dashint_{\S^{2n+1}}\zeta|v|^p\,\d\si =0,
\end{equation}
one has
\begin{equation}\label{eq:cr-aubin}
\frac{n^2}{4}\Big(\dashint_{\S^{2n+1}}|v|^p\,\d\si \Big)^{2/p}
\le 2^{2/p-1}(1+\var)\dashint_{\S^{2n+1}}vL_{\theta_0}v\,\d\si 
+C_\var\dashint_{\S^{2n+1}}v^2\,\d\si .
\end{equation}
At $p=Q^*$, the limiting factor $2^{2/p-1}=2^{-2/Q}$ is optimal, even when an arbitrary finite multiple of the $L^2(\S^{2n+1})$ norm squared is allowed on the right-hand side.
\end{prop}
\begin{proof}
Write $\d\mu=\d\si /|\S^{2n+1}|$, $c_0=n^2/4$, and
\[
\mathcal E(v)=\int_{\S^{2n+1}}vL_{\theta_0}v\,\d\mu
=\int_{\S^{2n+1}}(|\nabla_{\theta_0}v|^2+c_0v^2)\,\d\mu.
\]
Homogeneity allows the normalization $\int_{\S^{2n+1}}|v|^p\,\d\mu=1$. For $0<\delta<1$ and $P\in\S^{2n+1}$, set
\[
C(P,\delta)=\{\zeta\in\S^{2n+1}:\operatorname{Re}\langle\zeta,P\rangle>1-\delta\}.
\]
For $m=\int_{C(P,\delta)}|v|^p\,\d\mu$, the balance condition \eqref{eq:cr-aubin-balance} gives
\[
0=\int_{\S^{2n+1}}\operatorname{Re}\langle\zeta,P\rangle|v|^p\,\d\mu
\ge(1-\delta)m-(1-m),\quad m\le(2-\delta)^{-1}.
\]
Choose a finite smooth partition $\{\chi_j\}_{j=1}^{N_\delta}$ with $\sum_{j=1}^{N_\delta}\chi_j^2=1$, each $\chi_j$ supported in such a cap. H\"older's inequality and \eqref{eq:sharp-sphere-sobolev} give
\begin{align*}
1&=\sum_{j=1}^{N_\delta}\int_{\S^{2n+1}}\chi_j^2|v|^p\,\d\mu\\
&\le(2-\delta)^{2/p-1}\sum_{j=1}^{N_\delta}\Big(\int_{\S^{2n+1}}|\chi_jv|^p\,\d\mu\Big)^{2/p}\\
&\le\frac{(2-\delta)^{2/p-1}}{c_0}\sum_{j=1}^{N_\delta}\mathcal E(\chi_jv).
\end{align*}
Here we also used $p\le Q^*$ and $\mu(\S^{2n+1})=1$. Since $\sum_{j=1}^{N_\delta}\chi_j\nabla_{\theta_0}\chi_j=0$, the product rule gives
\[
\sum_{j=1}^{N_\delta}\mathcal E(\chi_jv)
=\mathcal E(v)+\int_{\S^{2n+1}}\Big(\sum_{j=1}^{N_\delta}|\nabla_{\theta_0}\chi_j|^2\Big)v^2\,\d\mu.
\]
Choose $\delta$ so small that $(1-\delta/2)^{-2/Q}\le1+\var$. Then, uniformly for $2<p\le Q^*$, we have
\[
(2-\delta)^{2/p-1}\le2^{2/p-1}(1+\var),\quad
(2-\delta)^{2/p-1}\le1.
\]
The localization identity and the choice of $\delta$ prove \eqref{eq:cr-aubin}, with $C_\var=\|\sum_{j=1}^{N_\delta}|\nabla_{\theta_0}\chi_j|^2\|_{L^\infty(\S^{2n+1})}$.

For optimality at the critical exponent, let $b_\lam $ be a positive conformal image of $1$ concentrating at $P$, normalized by $\int_{\S^{2n+1}} b_\lam ^{Q^*}\,\d\mu\allowbreak=1$, and set
\[
w_\lam (\zeta)=\frac{b_\lam (\zeta)+b_\lam (-\zeta)}
{(\int_{\S^{2n+1}}|b_\lam (\zeta)+b_\lam (-\zeta)|^{Q^*}\,\d\mu)^{1/Q^*}}.
\]
Then $w_\lam $ is antipodally symmetric, so \eqref{eq:cr-aubin-balance} holds. The mixed energy and $L^{Q^*}(\S^{2n+1})$ terms vanish because the bubbles concentrate at distinct points. Since $\mathcal E(b_\lam )=c_0$, optimality follows from
\[
\int_{\S^{2n+1}}|w_\lam |^{Q^*}\,\d\mu=1,\quad
\mathcal E(w_\lam )\to  c_0\,2^{2/Q},\quad
\int_{\S^{2n+1}} w_\lam ^2\,\d\mu\to 0.
\]
\end{proof}

The following result  improves the gradient term while keeping the $L^2(\S^{2n+1})$ term fixed.
\begin{prop}\label{prop:cr-aubin-fixed-mass}
There exist some constants $a_*\in (2^{-2/Q},1)$ and $p_*\in (2,Q^*)$, depending only on $n$, such that, for every $p_*\le p\le Q^*$ and every  $v\in S_1^2(\S^{2n+1})$ satisfying \eqref{eq:cr-aubin-balance}, one has
\begin{align}    
\frac{n^2}{4}\Big(\dashint_{\S^{2n+1}}|v|^p\,\d\si \Big)^{2/p}
\le&\, a_*\dashint_{\S^{2n+1}}vL_{\theta_0}v\,\d\si
+(1-a_*)\frac{n^2}{4}\dashint_{\S^{2n+1}}v^2\,\d\si \nonumber
\\=&\,
a_*\dashint_{\S^{2n+1}}|\nabla_{\theta_0}v|^2\,\d\si
+\frac{n^2}{4}\dashint_{\S^{2n+1}}v^2\,\d\si .\label{eq:cr-aubin-fixed-mass}
\end{align}
\end{prop}
\begin{proof}
Write $q=Q^*$, $c_0=n^2/4$ and $\d\mu=\d\si /|\S^{2n+1}|$. With the norm in \eqref{norm 4}, set
\[
I_a(v)=a\|v\|^2+(1-a)c_0\int_{\S^{2n+1}}v^2\,\d\mu.
\]
We first establish a local estimate near $1$. Fix $p_0\in(2,q)$. The implicit-function argument of Lemma~\ref{lempe4}, applied to the constraints
\[
\int_{\S^{2n+1}}|v|^p\,\d\mu=1,\quad
\int_{\S^{2n+1}}\zeta|v|^p\,\d\mu=0,
\]
gives, uniformly for $p\in[p_0,q]$, local coordinates
\[
v=1+w+\mu_p(w)+2\operatorname{Re}\langle\eta_p(w),\zeta\rangle,
\quad w\in T_1\M_0,
\]
for constrained functions close to $1$. Indeed, the derivative of the constraint map with respect to $(\mu_p,\eta_p)$ at $1$ is $\operatorname{diag}(p,\frac{p}{n+1}I_{2n+2})$ in real coordinates, and is uniformly invertible. Taylor expansion gives
\[
\mu_p(w)=-\frac{p-1}{2}\int_{\S^{2n+1}}w^2\,\d\mu+o(\|w\|^2),\quad
|\eta_p(w)|=O(\|w\|^2),
\]
uniformly in $p$. To see the uniformity, choose $0<\nu<\min\{1,p_0-2\}$ and use
\[
\Big||1+h|^p-1-ph-\frac{p(p-1)}2h^2\Big|
\le C(n,p_0)(|h|^{2+\nu}+|h|^q),
\]
together with the Sobolev embedding. Thus
\begin{equation}\label{eq:cr-aubin-local-expansion}
I_a(v)=c_0+a\|w\|^2
-(p-2+a)c_0\int_{\S^{2n+1}}w^2\,\d\mu+o(\|w\|^2),
\end{equation}
uniformly for $p\in[p_0,q]$ and $a\in[3/4,1]$. By Lemma~\ref{lempe3} and \eqref{eq:spherical-eigenvalues},
\[
\|w\|^2\ge(c_0+n)\int_{\S^{2n+1}}w^2\,\d\mu.
\]
Since $p\le q$, the quadratic term in \eqref{eq:cr-aubin-local-expansion} is bounded below by
\[
\frac{4a-2}{n+4}\|w\|^2\ge\frac{1}{n+4}\|w\|^2.
\]
Consequently, $I_a(v)\ge c_0$ for all constrained $v$ sufficiently close to $1$, uniformly in these ranges of $p$ and $a$.

Suppose now that \eqref{eq:cr-aubin-fixed-mass} fails for every $a_*<1$ and $p_*<q$. Then there exist $a_i\uparrow1$, $p_i\to q$ with $p_i\le q$, and nonnegative functions $v_i$ satisfying
\[
\int_{\S^{2n+1}}v_i^{p_i}\,\d\mu=1,\quad
\int_{\S^{2n+1}}\zeta v_i^{p_i}\,\d\mu=0,\quad
I_{a_i}(v_i)<c_0.
\]
Here we used homogeneity and replaced $v_i$ by $|v_i|$ if necessary. The sharp inequality \eqref{eq:sharp-sphere-sobolev} and H\"older's inequality yield
\[
c_0\le c_0\|v_i\|_{L^q(\S^{2n+1},\d\mu)}^2\le\|v_i\|^2<\frac{c_0}{a_i}.
\]
Hence $\|v_i\|^2\to c_0$ and $\|v_i\|_{L^q(\S^{2n+1},\d\mu)}\to1$. Choose $\var>0$ so small that $2^{-2/Q}(1+\var)<1$. Proposition~\ref{prop:cr-aubin} then gives $\liminf_i\int_{\S^{2n+1}}v_i^2\,\d\mu>0$.

After passing to a subsequence, $v_i\rightharpoonup v$ in $S_1^2(\S^{2n+1})$, $v_i\to v$ in $L^2(\S^{2n+1},\d\mu)$ and almost everywhere, with $v\ge0$ and $v\ne0$. Set $b=\int_{\S^{2n+1}}v^q\,\d\mu\in(0,1]$. The Brezis--Lieb lemma \cite[Theorem~1]{BL1983}, the energy decomposition, and \eqref{eq:sharp-sphere-sobolev} give
\begin{equation}\label{eq:cr-aubin-mass-splitting}
1=\lim_{i\to\infty}\frac{\|v_i\|^2}{c_0}
\ge b^{2/q}+(1-b)^{2/q}.
\end{equation}
Since $0<2/q<1$ and $b>0$, \eqref{eq:cr-aubin-mass-splitting} gives $b=1$. Equality of the limiting energies then gives $v_i\to v$ strongly in $S_1^2(\S^{2n+1})$ and identifies $v$ as an extremal of \eqref{eq:sharp-sphere-sobolev}. Strong convergence in $L^q(\S^{2n+1},\d\mu)$ and $p_i\to q$ allow passage to the balance constraint. The characterization of balanced extremals recalled after \eqref{eq:energy-quotient} therefore gives $v=1$. The local lower bound obtained from \eqref{eq:cr-aubin-local-expansion} contradicts $I_{a_i}(v_i)<c_0$.

The contradiction proves the existence of $a_*<1$ and $p_*<q$. Increasing $a_*$ if necessary gives $a_*>2^{-2/Q}$, since $I_a(v)$ is nondecreasing in $a$.
\end{proof}
\begin{rem}
    The antipodal two-bubble sequence in the proof of Proposition~\ref{prop:cr-aubin} shows that any admissible value of $a_*$ at $p=Q^*$ must satisfy $a_*\ge2^{-2/Q}$. Proposition~\ref{prop:cr-aubin-fixed-mass} does not determine the least admissible value.
\end{rem}

\medskip

\noindent\textbf{Data availability statement.}~	 Data sharing is not applicable to this article as no datasets were generated or analysed during the current study.

\medskip

\noindent\textbf{Declarations.}~ \textbf{Conflict of interest.} ~ On behalf of all authors, the corresponding author states that there is no conflict of interest.

%\medskip

%\noindent\textbf{AI assistance statement.} ~OpenAI models were used for language editing, %consistency checks, and assistance in drafting mathematical revisions and proof arguments. The %authors are responsible for independently verifying the mathematical content and for the final %version submitted for publication.

%\medskip


{\small
\begin{thebibliography}{99}

%\bibitem{Afeltra2024}
%C.~Afeltra.
%\newblock A compactness result for the {CR} {Y}amabe problem in three
%  dimensions.
%\newblock {\em Commun. Contemp. Math.}, 27(8):Paper No. 2550003, 35, 2025.

\bibitem{Afeltra2024-2}
C.~Afeltra.
\newblock Blow-up analysis and degree theory for the {W}ebster curvature
  prescription problem in three dimensions.
\newblock {\em Calc. Var. Partial Differential Equations}, 64(8):Paper No. 260,
  41, 2025.

  \bibitem{BC1991}
A.~Bahri and J.-M. Coron.
\newblock The scalar-curvature problem on the standard three-dimensional
  sphere.
\newblock {\em J. Funct. Anal.}, 95(1):106--172, 1991.

\bibitem{BFM2013}
T.P. Branson, L.~Fontana, and C.~Morpurgo.
\newblock Moser-{T}rudinger and {B}eckner-{O}nofri's inequalities on the {CR}
  sphere.
\newblock {\em Ann. of Math. (2)}, 177(1):1--52, 2013.

\bibitem{BL1983}
H.~Brezis and E.~Lieb.
\newblock A relation between pointwise convergence of functions and convergence of functionals.
\newblock {\em Proc. Amer. Math. Soc.}, 88(3):486--490, 1983.

\bibitem{CGL1993}
G.~Citti, N.~Garofalo, and E.~Lanconelli.
\newblock Harnack's inequality for sum of squares of vector fields plus a
  potential.
\newblock {\em Amer. J. Math.}, 115(3):699--734, 1993.

\bibitem{CLMR2023}
G.~Catino, Y.~Y. Li, D.~D. Monticelli, and A.~Roncoroni.
\newblock A {L}iouville theorem in the {H}eisenberg group.
\newblock {\em J. Eur. Math. Soc. (JEMS)}.
\newblock \href{https://doi.org/10.4171/JEMS/1705}{DOI:10.4171/JEMS/1705, published online, 2025.}

  \bibitem{CGY1993}
S.-Y. Chang, M.~J. Gursky, and P.~C. Yang.
\newblock The scalar curvature equation on {$2$}- and {$3$}-spheres.
\newblock {\em Calc. Var. Partial Differential Equations}, 1(2):205--229, 1993.

\bibitem{CY1987}
S.-Y. Chang and P.~C. Yang.
\newblock Prescribing {G}aussian curvature on {$\S^2$}.
\newblock {\em Acta Math.}, 159(3-4):215--259, 1987.

\bibitem{CY1991}
S.-Y. Chang and P.~C. Yang.
\newblock A perturbation result in prescribing scalar curvature on {$S^n$}.
\newblock {\em Duke Math. J.}, 64(1):27--69, 1991.

\bibitem{Cheng}
J. Cheng.
\newblock Curvature functions for the sphere in pseudo-Hermitian geometry.
\newblock {\em Tokyo J. Math.}, 14(1):151--163, 1991.

\bibitem{CAY2013}
H.~Chtioui, M.~Ahmedou, and R.~Yacoub.
\newblock Existence and multiplicity results for the prescribed {W}ebster
  scalar curvature problem on three {CR} manifolds.
\newblock {\em J. Geom. Anal.}, 23(2):878--894, 2013.

\bibitem{DT2006}
S.~Dragomir and G.~Tomassini.
\newblock {\em Differential geometry and analysis on {CR} manifolds}, volume
  246 of {\em Progress in Mathematics}.
\newblock Birkh\"{a}user Boston, Inc., Boston, MA, 2006.

\bibitem{ES1986}
J.~F. Escobar and R. Schoen.
\newblock  Conformal metrics with prescribed scalar curvature.
\newblock {\em Invent. Math.}, 86(2): 243--254, 1986.

\bibitem{FU2004}
V.~Felli and F.~Uguzzoni.
\newblock Some existence results for the {W}ebster scalar curvature problem in
  presence of symmetry.
\newblock {\em Ann. Mat. Pura Appl. (4)}, 183(4):469--493, 2004.

\bibitem{F1975}
G.~B. Folland.
\newblock Subelliptic estimates and function spaces on nilpotent {L}ie groups.
\newblock {\em Ark. Mat.}, 13(2):161--207, 1975.

\bibitem{FS1974}
G.~B. Folland and E.~M. Stein.
\newblock Estimates for the {$\bar \pa  _{b}$} complex and analysis on the
  {H}eisenberg group.
\newblock {\em Comm. Pure Appl. Math.}, 27:429--522, 1974.

\bibitem{GH2017}
N.~Gamara and B.~Hafassa.
\newblock The {$\beta$}-flatness condition in {CR} spheres.
\newblock {\em Adv. Nonlinear Stud.}, 17(1):193--213, 2017.

\bibitem{GV2001}
N.~Garofalo and D.~Vassilev.
\newblock Symmetry properties of positive entire solutions of {Y}amabe-type
  equations on groups of {H}eisenberg type.
\newblock {\em Duke Math. J.}, 106(3):411--448, 2001.

\bibitem{G1981}
D.~Geller.
\newblock The {L}aplacian and the {K}ohn {L}aplacian for the sphere.
\newblock {\em J. Differential Geometry}, 15(3):417--435 (1981), 1980.



%\bibitem{H1990}
%E.~Hebey.
%\newblock Changements de m\'{e}triques conformes sur la sph\`ere. {L}e
%  probl\`eme de {N}irenberg.
%\newblock {\em Bull. Sci. Math.}, 114(2):215--242, 1990.

\bibitem{Ho20152}
P.~T. Ho.
\newblock The {W}ebster scalar curvature flow on {CR} sphere. {P}art {II}.
\newblock {\em Adv. Math.}, 268:836--905, 2015.

\bibitem{JL1987}
D.~Jerison and J.~M. Lee.
\newblock The {Y}amabe problem on {CR} manifolds.
\newblock {\em J. Differential Geom.}, 25(2):167--197, 1987.

\bibitem{JL1988}
D.~Jerison and J.~M. Lee.
\newblock Extremals for the {S}obolev inequality on the {H}eisenberg group and
  the {CR} {Y}amabe problem.
\newblock {\em J. Amer. Math. Soc.}, 1(1):1--13, 1988.

\bibitem{JLX2014}
T. Jin, Y. Y. Li and J. Xiong,
\newblock On a fractional Nirenberg problem, Part I: blow up analysis and compactness of solutions,
\newblock {\em J. Eur. Math. Soc.}, 16(6):1111--1171, 2014.

\bibitem{JLX2015}
T. Jin, Y. Y. Li and J. Xiong,
\newblock On a fractional Nirenberg problem, Part II: existence of solutions,
\newblock {\em Int. Math. Res. Not. IMRN}, 6:1555--1589, 2015.

\bibitem{L2026}
J.~Li.
\newblock Classification of positive entire solutions of the {CR} {Y}amabe
  equation on the {H}eisenberg group.
\newblock \href{https://arxiv.org/abs/2608.10642}{arXiv:2608.10642, 2026}.

\bibitem{LLM2026}
J.~Li, Y.~Liu, and X.~Ma.
\newblock Classification of positive entire solutions of the critical {CR}
  {Y}amabe equation on the {H}eisenberg group.
\newblock
  \href{https://doi.org/10.13140/RG.2.2.13416.43527}{DOI:10.13140/RG.2.2.13416.43527,
  2026}.

\bibitem{Li1995}
Y.~Y. Li.
\newblock Prescribing scalar curvature on {$\S^n$} and related problems. {I}.
\newblock {\em J. Differential Equations}, 120(2):319--410, 1995.

\bibitem{Li1996}
Y.~Y. Li.
\newblock Prescribing scalar curvature on {$\S^n$} and related problems. {II}.
  {E}xistence and compactness.
\newblock {\em Comm. Pure Appl. Math.}, 49(6):541--597, 1996.

\bibitem{MU2002}
A.~Malchiodi and F.~Uguzzoni.
\newblock A perturbation result for the {W}ebster scalar curvature problem on
  the {CR} sphere.
\newblock {\em J. Math. Pures Appl. (9)}, 81(10):983--997, 2002.

\bibitem{M1973}
J. Moser.
\newblock {\em On a {N}onlinear {P}roblem in {D}ifferential {G}eometry}.
Dynamical Systems (Proc. Sympos., Univ. Bahia, Salvador, 1971), 273--280, {\em New York: Academic Press}, 1973.

\bibitem{QTWPartI}
J. Qiang, Z. Tang and H. Wang,
\emph{The CR Nirenberg problem, Part I: blow-up analysis and compactness},
preprint, 2026.

\bibitem{SZ1996}
R.~Schoen and D.~Zhang.
\newblock Prescribed scalar curvature on the {$n$}-sphere.
\newblock {\em Calc. Var. Partial Differential Equations}, 4(1):1--25, 1996.

%\bibitem{R2003}
%F.~Robert.
%\newblock Positive solutions for a fourth order equation invariant under
%  isometries.
%\newblock {\em Proc. Amer. Math. Soc.}, 131(5):1423--1431, 2003.

\bibitem{Yan2023}
Z.~Yan.
\newblock Improved higher-order {S}obolev inequalities on {CR} sphere.
\newblock {\em J. Funct. Anal.}, 284(10):Paper No.~109890, 2023.

\end{thebibliography}
}
\end{document}